\documentclass[12pt]{amsart}
\usepackage{amssymb}
\usepackage{verbatim}
\usepackage{mathrsfs}
\usepackage{mathtools}

\usepackage{xcolor}

\newtheorem{thm}{Theorem}[section]

\newtheorem{lem}[thm]{Lemma}
\newtheorem{cor}[thm]{Corollary}

\theoremstyle{definition}

\theoremstyle{remark}

\newtheorem*{rem}{Remark}

\numberwithin{equation}{section}

\newcommand{\M}{\mathcal{M}}
\newcommand{\MT}{\mathcal{MT}}
\newcommand{\Co}{\mathcal{C}}
\newcommand{\ZF}{\zeta_7}
\newcommand{\ET}{\mathcal{ET}}
\newcommand{\HyG}{ {}_2F_1 }
\newcommand{\HyGI}{ {}_2\mathrm{I}_1 }
\newcommand{\ZC}{\mathcal{Z}_4}
\newcommand{\AlphaVec}{\bar{\alpha}}
\newcommand{\TD}{\mathcal{TD}}
\newcommand{\DU}{\eta} 
\newcommand{\U}{\mathcal{U}}
\newcommand{\E}{\mathcal{E}}
\newcommand{\RF}{\mathcal{D}}
\newcommand{\Sok}{\mathbf{S}}
\newcommand{\ZE}{\mathcal{F}}
\newcommand{\De}{\mathcal{I}}

\newcommand{\GenHyG}[5]{ {}_{#1}F_{#2} \left( \begin{matrix} #3 \\ #4 \end{matrix} ; #5 \right) }
\newcommand{\GenHyGI}[5]{ {}_{#1}\mathrm{I}_{#2} \left( \begin{matrix} #3 \\ #4 \end{matrix} ; #5 \right) }

\newcommand{\term}{T}
\newcommand{\mainU}{MU}

\newcommand{\Z}{\mathbf{Z}}

\newcommand{\Y}{\mathcal{Y}}
\newcommand{\y}{\mathbf{y}}
\newcommand{\W}{\mathbf{W}}

\DeclareMathOperator{\res}{res}

\DeclareMathOperator{\artanh}{artanh}

\mathtoolsset{showonlyrefs}

\begin{document}

\title[The shifted fourth moment of modular form $L$-functions]{The shifted fourth moment of modular form $L$-functions in the weight aspect}

\author[O.Balkanova]{Olga Balkanova}
\address{
Steklov Mathematical Institute of Russian Academy of Sciences, 8 Gubkina st., Moscow, 119991, Russia}
\email{balkanova@mi-ras.ru}

\author[D. Frolenkov]{Dmitry  Frolenkov}
\address{
Steklov Mathematical Institute of Russian Academy of Sciences, 8 Gubkina st., Moscow, 119991, Russia}
\email{frolenkov@mi-ras.ru}
\thanks{Research  was   supported   by  the  Theoretical   Physics   and   Mathematics
Advancement Foundation "BASIS".}

\begin{abstract}
We prove a reciprocity type formula for the fourth moment of $L$-functions associated to holomorphic primitive cusp forms of level one and large weight which relates it to the eighth moment of the Riemann zeta function and the dual weighted fourth moments of automorphic $L$-functions (both holomorphic and Maass). The main objective of the paper is to study the structure of the main term for possible generalization of the method to higher moments.

\end{abstract}

\dedicatory{To the 85th anniversary of the birth of Professor N.V. Kuznetsov}

\maketitle


\section{Introduction}
Consider the shifted fourth moment
\begin{equation}\label{4moment def}
\M_4(\alpha_1,\alpha_2,\alpha_3,\alpha_4):=\sum_{f\in H_{2k}} \omega_f\prod_{j=1}^4L_f(1/2+\alpha_j)
\end{equation}
of $L$-functions associated to holomorphic primitive cusp forms of weight $2k \geq 2$ and level one.
Following the "recipe" of Conrey, Farmer, Keating, Rubinstein and Snaith (see \cite{CFKRS}) the main term of  \eqref{4moment def}   is
conjectured to be
\begin{equation}\label{MT(a1,a2,a3,a4) conj}
\MT_4(\alpha_1,\alpha_2,\alpha_3,\alpha_4)=
\sum_{\epsilon_1,\epsilon_2,\epsilon_3,\epsilon_4=\pm1}
\Co(\epsilon_1\alpha_1,\epsilon_2\alpha_2,\epsilon_3\alpha_3,\epsilon_4\alpha_4)
\ZF(\epsilon_1\alpha_1,\epsilon_2\alpha_2,\epsilon_3\alpha_3,\epsilon_4\alpha_4),
\end{equation}
where
\begin{equation}\label{7zeta def}
\ZF(\epsilon_1\alpha_1,\epsilon_2\alpha_2,\epsilon_3\alpha_3,\epsilon_4\alpha_4)=
\zeta^{-1}\left(2+\sum_{j=1}^4\epsilon_j\alpha_j\right)
\prod_{1\le i<j\le4}\zeta(1+\epsilon_i\alpha_i+\epsilon_j\alpha_j),
\end{equation}
\begin{equation}\label{Co def}
\Co(\epsilon_1\alpha_1,\epsilon_2\alpha_2,\epsilon_3\alpha_3,\epsilon_4\alpha_4)=
\prod_{j=1}^4\epsilon_j^{k}\prod_{\epsilon_j=-1}X_k(\alpha_j)
\end{equation}
and
\begin{equation}\label{Xdef}
X_{k}(\alpha):=(2\pi)^{2\alpha}\frac{\Gamma(k-\alpha)}{\Gamma(k+\alpha)}.
\end{equation}

Here $\Gamma(x)$ stands for the Gamma functions and $\zeta(x)$ for the Riemann zeta function.

Throughout the paper, we use the notation $\AlphaVec:=(\alpha_1,\alpha_2,\alpha_3,\alpha_4)$.

Our main result is a reciprocity type formula  for \eqref{4moment def}. See \cite{BF2momsymsq} for a brief survey of such formulas.
The obtained relation translates the original moment to different fourth moments of automorphic $L$-functions and the eighth moment of the Riemann zeta-function.
\begin{thm}\label{thm:main}
For $|\alpha_j|\ll\epsilon_0$ the following formula holds:
\begin{multline}\label{4mom result}
\M_4(\alpha_1,\alpha_2,\alpha_3,\alpha_4)=\MT_4(\alpha_1,\alpha_2,\alpha_3,\alpha_4)\\+
\left(\term_{d}^{1}+\term_{h}^{1}+\term_{c}^{1}\right)(\alpha_1,\alpha_2,\alpha_3,\alpha_4)+
\left(\term_{d}^{2}+\term_{h}^{2}+\term_{c}^{2}\right)(\alpha_1,\alpha_2,\alpha_3,\alpha_4)\\+
(-1)^kX_k(\alpha_2)\left(\term_{d}^{1}+\term_{h}^{1}+\term_{c}^{1}\right)(\alpha_1,-\alpha_2,\alpha_3,\alpha_4)\\+
(-1)^kX_k(\alpha_2)\left(\term_{d}^{2}+\term_{h}^{2}+\term_{c}^{2}\right)(\alpha_1,-\alpha_2,\alpha_3,\alpha_4)+O(k^{-1+\epsilon}),
\end{multline}
where $\term_{d}^{1}(\AlphaVec)$ and $\term_{d}^{2}(\AlphaVec)$ are the weighted fourth moments of Maass form $L$-functions defined by \eqref{Ed11} and \eqref{Ed21}, $\term_{h}^{1}(\AlphaVec)$ and $\term_{h}^{2}(\AlphaVec)$ are the weighted fourth moments of $L$-functions associated to holomorphic cusp forms defined by \eqref{Eh11} and  \eqref{Eh21},
$\term_{c}^{1}(\AlphaVec)$ and $\term_{c}^{2}(\AlphaVec)$ are the weighted eighth moments of the Riemann zeta function defined by \eqref{Ec011def} and \eqref{Ec021def}.
\end{thm}

\begin{rem}
Note that the restriction $|\alpha_j|\ll\epsilon_0$ is of technical nature as it allows us to simplify  the analysis of some special functions that arise while estimating error terms.
\end{rem}

Formula \eqref{4mom result} is a generalization of the result of Kuznetsov, who  in 1994 in one of his preprints (in Russian)  obtained a similar formula in case when
\begin{equation}
\alpha_1=\alpha_2=0, \quad \alpha_3=\alpha_4=1/2+\rho.
\end{equation}
Kuznetsov used this result to prove the following asymptotic formula for the fourth moment with additional averaging over $k$.
\begin{thm}\label{thm:4thmomkuznetsov}(Kuznetsov, preprint $1994$)
One has
\begin{equation}\label{4thmom Kuznetsov}
\sum_{\substack{L<k\le 2L\\k\equiv 0\pmod{ 2} }}
\sum_{f\in H_{2k}} \omega_fL^4_f(1/2)=LP_6(\log L)+O(L^{1/2+\epsilon}).
\end{equation}
\end{thm}
This formula is an analogue of  Ivic's  asymptotic formula \cite{Ivic} for the fourth moment of Maass form $L$-functions.  For completeness  we add the proof of Theorem \ref{thm:4thmomkuznetsov} in the Appendix.
The following result  was not proved by Kuznetsov but follows immediately from the proof of Theorem  \ref{thm:4thmomkuznetsov}. We also provide its proof in the Appendix.
\begin{thm}\label{cor:4thmom short}
For $K^{1/3+\epsilon}<G<K^{1-\epsilon}$ the following formula holds:
\begin{equation}\label{4thmom short}
\sum_{k}\exp\left(-\left(\frac{k-K}{G}\right)^2\right)
\sum_{f\in H_{4k}} \omega_fL^4_f(1/2)\ll G^{1+\epsilon}.
\end{equation}
\end{thm}
Theorem \ref{cor:4thmom short} is a holomorphic analogue of Jutila's result \cite{Jut} on the fourth moment of Maass form $L$-functions in short intervals.  Using
\eqref{4thmom short} with $G=K^{1/3+\epsilon}$ we recover the subconvexity  bound
\begin{equation}\label{subconvexity}
L_f(1/2)\ll k^{1/3+\epsilon}
\end{equation}
due to Peng \cite{Peng}.

Our proof of \eqref{4mom result} relies heavily on the ideas of Kuznetsov. First, assuming  that $\Re{\alpha_3},\Re{\alpha_4}$ are sufficiently large and using a series representation for the product of two $L$-functions
\begin{equation}
L_f(1/2+\alpha_3)L_f(1/2+\alpha_4),
\end{equation}
we transform the fourth moment to the sum of twisted seconds moments
\begin{equation}
\sum_{n=1}^{\infty}\frac{\sigma_{\alpha_3-\alpha_4}(n)}{n^{1/2+\alpha_3}}\sum_{f\in H_{2k}} \omega_f\lambda_f(n)L_f(1/2+\alpha_1)L_f(1/2+\alpha_2).
\end{equation}
Next, for the second moment we apply a convolution type formula due to Kuznetsov (see Section \ref{sec:2mom} and \cite[Theorem 4.2]{BFJEMS} for the proof) which represents it as a sum of several additive divisor problems. The main term of the second moment gives rise to four of the sixteen main terms in $\MT_4(\alpha_1,\alpha_2,\alpha_3,\alpha_4)$. 

The remaining twelve main terms are hidden in the double sums of triple divisor functions:
\begin{equation}\label{triple div sums}
\sum_{m=1}^{\infty}\sum_{n=1}^{\infty}\frac{\sigma_{\beta}(m)\sigma_{\gamma}(n)\sigma_{\delta}(m+n)}{m^{\theta_1}n^{\theta_2}}F\left(\frac{m}{n}\right),
\end{equation}
where $F(x)$ is a Gauss hypergeometric function multiplied by some rational function. In order to separate the main term contribution from \eqref{triple div sums}, Kuznetsov suggested to split the sum smoothly (see Section \ref{sec:4mom to triple div}) into two parts:
\begin{multline}\label{triple div sums split}
\sum_{m=1}^{\infty}\sum_{n=1}^{\infty}\frac{\sigma_{\beta}(m)\sigma_{\gamma}(n)\sigma_{\delta}(m+n)}{m^{\theta_1}n^{\theta_2}}F\left(\frac{m}{n}\right)=
\sum_{m=1}^{\infty}\frac{\sigma_{\beta}(m)}{m^{\theta_1}}\sum_{n>m}^{\infty}\frac{\sigma_{\gamma}(n)\sigma_{\delta}(n+m)}{n^{\theta_2}}F\left(\frac{m}{n}\right)\\+
\sum_{n=1}^{\infty}\frac{\sigma_{\gamma}(n)}{n^{\theta_2}}\sum_{m\ge n}^{\infty}\frac{\sigma_{\beta}(m)\sigma_{\delta}(m+n)}{m^{\theta_1}n^{\theta_2}}F\left(\frac{m}{n}\right),
\end{multline}
and then treat each of the inner sums  using a convolution formula for the additive divisor problem (see \cite{Mo} and Section \ref{sec:ADP}). 
The main terms arising from the additive divisor problems contribute to the main term of the fourth moment $\MT_4(\alpha_1,\alpha_2,\alpha_3,\alpha_4)$, but do not contain all of its summands. The last remaining part of $\MT_4(\alpha_1,\alpha_2,\alpha_3,\alpha_4)$ is hidden in the moment of the Riemann-zeta function that appears in the convolution formula for the additive divisor problem.

In fact, the main difference (except that we consider the general case) between  our work and the preprint of  Kuznetsov lies in the approach to the evaluation of the main term. It turns out that all summands which produce the main term are integrals of some hypergeometric function.  To evaluate such integrals Kuznetsov used an asymptotic expansion for the hypergeometric functions in terms of Bessel functions followed by a quite complicated evaluation of these integrals. In order to obtain not an asymptotic formula for the main term but an explicit expression $\MT_4(\alpha_1,\alpha_2,\alpha_3,\alpha_4)$ we decided to apply instead the technique based on the Mellin-Barnes integrals.

The formula \eqref{4mom result} should also be compared to the one proved in \cite[Theorem 3.1]{HK12}, with the choice of test functions like in  \cite[(4.1)]{HK12}. In this case, \cite[Theorem 3.1]{HK12} can be reduced to the fourth moment at the central point $\alpha_j=0$ averaged over $k$  with some complicated ratio of Gamma factors as weights. In contrast, the formula \eqref{4mom result} allows us to average over $k$ with arbitrary weights. In fact, \cite[Theorem 3.1]{HK12} is a reciprocity type formula for the fourth moment of $L$-functions attached to Maass forms, obtained by Kuznetsov \cite{Kuz1989} with corrections due to Motohashi \cite{Motoh2003}. However, in order to avoid various convergence issues, the averaged fourth moment of $L$-functions attached to holomorphic forms should also be included in this reciprocity formula. This allows using it in the investigation of the averaged fourth moment in the holomorphic case. However, one cannot average the fourth moment with arbitrary weight functions because of the presence of quite restricted conditions.  This problem was solved in  \cite{HK12}    by constructing admissible weights approximating the characteristic function of an interval.

Finally, let us compare the proof of formula \eqref{4mom result} (reciprocity formula in the holomorphic case) with  the one given in \cite{Kuz1989}, \cite{Motoh2003} for Maass forms. It is well known that some convergence issues arise in the case of Maass forms. One way to avoid such problems is to apply the so-called "root number trick"   by considering the moment $\sum_{j}\epsilon_jL_j(1/2)^4$ instead of  $\sum_{j}L_j(1/2)^4$. This is possible since $\epsilon_j=\pm1$ and $L_j(1/2)=0$ if $\epsilon_j=-1$. Note that this approach does not work when we consider the general shifted fourth moment. Another method that can be applied for arbitrary shifts (see \cite{Motoh2003}) is an application of a  regularized Kuznetsov trace formula.
However, it seems that these two methods are not applicable in the holomorphic case, where the same difficulties with convergence arise.
 Therefore, in our situation the crucial idea, which allows to avoid all convergence issues,  is the smooth decomposition \eqref{triple div sums split} proposed by Kuznetsov.

\section{Notation and preliminaries}
Let $e(x):=e^{2\pi ix}.$ Consider the number of divisors functions:
\begin{equation}\label{sigma def}
\sigma_v(n):=\sum_{d|n}d^v
\end{equation}
and
\begin{equation}\label{tau to sigma}
\tau_v(n):=\sum_{n_1n_2=n}\left( \frac{n_1}{n_2}\right)^v=n^{-v}\sigma_{2v}(n).
\end{equation}
The Riemann zeta function $\zeta(s)$ satisfies the functional equation
\begin{equation}\label{zeta functional}
\zeta(s)=\RF(s)\zeta(1-s),\quad
\RF(s)=2(2\pi)^{s-1}\Gamma(1-s)\sin\frac{\pi s}{2}.
\end{equation}

The Ramanujan identity is given by
\begin{equation}\label{RamId}
\sum_{n=1}^{\infty}\frac{\sigma_a(n)\sigma_b(n)}{n^s}=\frac{\zeta(s)\zeta(s-a)\zeta(s-b)\zeta(s-a-b)}{\zeta(2s-a-b)}.
\end{equation}

Throughout the paper $H_{2k}$ denotes  the Hecke normalized  basis  of the space of holomorphic cusp forms of weight $2k \geq 2$ and level one.
For $f\in H_{2k}$ the following Fourier expansion exists:
\begin{equation}
f(z)=\sum_{n\geq 1}\lambda_{f}(n)n^{k-1/2}\exp(2\pi inz),
\quad \lambda_{f}(1)=1.
\end{equation}
Throughout the paper $B_{d}$ denotes  the orthonormal basis $\{u_j\}$ of the space of Maass cusp forms of level one. We assume that any $u_j\in B_d$ is an eigenfunction of all Hecke operators and the hyperbolic Laplacian and denote by $\{\lambda_{j}(n)\}$ the eigenvalues of Hecke operators acting on $u_{j}$. Also $\kappa_{j}=1/4+t_{j}^2$ stands for (with  $t_j>0$)  the eigenvalues of the hyperbolic Laplacian acting on $u_{j}$. Finally, one has
\begin{equation*}
u_{j}(-\bar{z})=\varepsilon_j u_{j}(z),\quad \varepsilon_j=\pm1.
\end{equation*}
We call $u_j$ even or odd if $\varepsilon_j=1$ or $\varepsilon_j=-1$, respectfully.
It is known that
\begin{equation*}
u_{j}(x+iy)=\sqrt{y}\sum_{n\neq 0}\rho_{j}(n)K_{it_j}(2\pi|n|y)e(nx),
\end{equation*}
where $K_{\mu}(x)$ is the $K$-Bessel function  and
$\rho_{j}(n)=\rho_{j}(1)\lambda_{j}(n).$
Both $\lambda_{f}(n)$ and $\lambda_{j}(n)$ satisfy the following multiplicity relation
\begin{equation}\label{eq:mult}
\lambda(m)\lambda(n)=\sum_{d|(m,n)}\lambda\left( \frac{mn}{d^2}\right).
\end{equation}
For each  $f\in H_{2k}$ and $u_j\in B_d$  we define the associated $L$-functions as
\begin{equation}
L_f(s):=\sum_{n=1}^{\infty}\frac{\lambda_f(n)}{n^s},\quad
L_j(s):=\sum_{n=1}^{\infty}\frac{\lambda_j(n)}{n^s}, \quad \Re{s}>1.
\end{equation}
It follows from \eqref{eq:mult} and \eqref{sigma def} that
\begin{equation}\label{Lprod}
L_f(1/2+\alpha)L_f(1/2+\beta)=\zeta(1+\alpha+\beta)
\sum_{n=1}^{\infty}\frac{\lambda_f(n)\sigma_{\alpha-\beta}(n)}{n^{1/2+\alpha}},
\end{equation}
and the same holds for $L_j(s)$.

Let
\begin{equation}\label{omegaj}
\omega_{j}:=\frac{|\rho_{j}(1)|^2}{\cosh{\pi t_j}}
\end{equation}
be the standard normalizing coefficient for Maass cusp forms and
\begin{equation}\label{harmonic weight}
\omega_f:=\frac{ \Gamma(2k-1)}{(4\pi )^{2k-1}\langle f,f\rangle_1}
\end{equation}
be the so-called harmonic weight  (see \cite[Lemma 2.5]{ILS}),
where $\langle f,f\rangle_1$ is the Petersson inner product on the space $H_{2k}$.

We will frequently use the Stirling bound \cite[5.11.9]{HMF} for the Gamma function:
\begin{equation}\label{Stirling0}
|\Gamma(x+iy)|\ll|y|^{x-1/2}e^{-\pi|y|/2}.
\end{equation}
It follows from the Stirling formula \cite[5.11.3, 5.11.13]{HMF} that for a bounded $|\alpha|$ the following formula holds:
\begin{equation}\label{Xapprox}
X_{k}(\alpha)=\frac{(2\pi)^{2\alpha}}{k^{2\alpha}}\left(1+O\left(\frac{1}{k}\right)\right).
\end{equation}

\begin{lem}\label{Lemma beta integrals}
For $\Re{d}<0$ and $-1<\Re{c}<-1-\Re{d}$ one has
\begin{equation}\label{beta +0infinity}
\int_0^{\infty}x^c(1+x)^ddx=\frac{\Gamma(c+1)\Gamma(-1-c-d)}{\Gamma(-d)}.
\end{equation}
\end{lem}
\begin{proof}
Equation \eqref{beta +0infinity} follows from \cite[5.12.3]{HMF}.
\end{proof}
Let ${}_p\mathrm{I}_q$ denote the generalized hypergeometric function multiplied by suitable Gamma factors:
\begin{equation}\label{pIq def}
\GenHyGI{p}{q}{a_1,\cdots,a_p}{b_1,\cdots,b_q}{z}:=
\frac{\Gamma(a_1)\cdots \Gamma(a_p)}{\Gamma(b_1)\cdots \Gamma(b_q)} \GenHyG{p}{q}{a_1,\cdots,a_p}{b_1,\cdots,b_q}{z}=
\sum_{j=0}^{\infty}\frac{\Gamma(a_1+j)\cdots \Gamma(a_p+j)}{\Gamma(b_1+j)\cdots \Gamma(b_q+j)}\frac{z^j}{j!}.
\end{equation}
When $p=2,q=1$ we will use the simplified notation $\HyGI(a_1,a_2,b_1;x)$. Note that
\begin{equation}\label{2I1 to 3I2}
\HyGI(a,b,c;x)=\frac{\Gamma(a)\Gamma(b)}{\Gamma(c)}+x\GenHyGI{3}{2}{1+a,1+b,1}{2,1+c}{x}.
\end{equation}
In order to find an asymptotic expansion for the integrals of the form
\begin{equation}\label{I Hux def}
I=\int_a^{b}g(x)e(f(x))dx,
\end{equation}
we will use the following version of the saddle point method due to Huxley \cite{Hux}. Recall that the saddle point is a solution of the equation $f'(x)=0$.
\begin{lem}\label{Lemma Huxley}
Let $f(x)$ and $g(x)$ be some smooth functions defined on the interval $[a,b]$ such that $g(a)=g(b)=0$ and
\begin{equation}\label{Hux fg conditions}
f^{(i)}(x)\ll\frac{\Theta_f}{\Omega_f^i},\quad
g^{(j)}(x)\ll\frac{1}{\Omega_g^j},\quad
f^{(2)}(x)\gg\frac{\Theta_f}{\Omega_f^2}
\end{equation}
for $i=1,2,3,4$ and $j=0,1,2$. Suppose that $f'(x_0)=0$ for $a<x_0<b$, and that $f(x)$ changes its sign from negative to positive at $x=x_0$. Let $\kappa=\min(x_0-a,b-x_0)$. Then
\begin{equation}\label{I Hux asympt}
I=\frac{g(x_0)e(f(x_0)+1/8)}{\sqrt{f''(x_0)}}+O\left(\frac{\Omega_f^4}{\kappa^3\Theta_f^2}+
\frac{\Omega_f}{\Theta_f^{3/2}}+\frac{\Omega_f^3}{\Theta_f^{3/2}\Omega_g^2}\right).
\end{equation}
\end{lem}

\section{Convolution formula for the second moment}\label{sec:2mom}
As mentioned in the introduction, our method of investigating the fourth moment relies on the exact formula for the twisted second moment.
 Such formulas was proved independently by Kuznetsov (in his preprint from 1994 in Russian) and by Iwaniec-Sarnak (see \cite[Theorem 17]{IS}).

\begin{thm}\label{thm:kuznetsov}
For $|\Re{\alpha_1}|<k-1$, $|\Re{\alpha_2}|<k-1$ we have
\begin{multline}\label{eq:secondmoment}
\M_2(l;\alpha_1,\alpha_2):=\sum_{f \in H_{2k}}\omega_f\lambda_f(l)L_f(1/2+\alpha_1)L_f(1/2+\alpha_2)=\\=
\MT_2(l;\alpha_1,\alpha_2)+\sum_{i=1}^3\ET^{(i)}_2(l;\alpha_1,\alpha_2),
\end{multline}
where
\begin{multline}\label{MT2 def}
\MT_2(l;\alpha_1,\alpha_2):=
\frac{\sigma_{\alpha_1-\alpha_2}(l)}{l^{1/2+\alpha_1}}\zeta(1+\alpha_1+\alpha_2)+
(-1)^kX_k(\alpha_1)\frac{\sigma_{\alpha_1+\alpha_2}(l)}{l^{1/2+\alpha_2}}\zeta(1-\alpha_1+\alpha_2)\\+
(-1)^kX_k(\alpha_2)\frac{\sigma_{\alpha_1+\alpha_2}(l)}{l^{1/2+\alpha_1}}\zeta(1+\alpha_1-\alpha_2)+
\frac{\sigma_{\alpha_1-\alpha_2}(l)}{l^{1/2-\alpha_2}}\zeta(1-\alpha_1-\alpha_2)X_k(\alpha_1)X_k(\alpha_2),
\end{multline}
\begin{equation}\label{ET2 1}
\ET^{(1)}_2(l;\alpha_1,\alpha_2):=\frac{(-1)^k}{\sqrt{l}}\sum_{1 \leq n \leq l-1}
\frac{\sigma_{\alpha_1-\alpha_2}(n)\sigma_{\alpha_1+\alpha_2}(l-n)}{n^{(\alpha_1-\alpha_2)/2}(l-n)^{(\alpha_1+\alpha_2)/2}}
\phi_k\left(\frac{n}{l};\alpha_1,\alpha_2\right),
\end{equation}
\begin{equation}\label{ET2 2}
\ET^{(2)}_2(l;\alpha_1,\alpha_2):=\frac{1}{\sqrt{l}}\sum_{n=1}^{\infty}
\frac{\sigma_{\alpha_1+\alpha_2}(n)\sigma_{\alpha_1-\alpha_2}(n+l)}{n^{(\alpha_1+\alpha_2)/2}(n+l)^{(\alpha_1-\alpha_2)/2}}
\Phi_k\left(\frac{l}{n+l};\alpha_1,\alpha_2\right),
\end{equation}
\begin{equation}\label{ET2 3}
\ET^{(3)}_2(l;\alpha_1,\alpha_2):=(-1)^kX_k(\alpha_2)\ET^{(2)}_2(l;\alpha_1,-\alpha_2),
\end{equation}
and the weight functions are defined as
\begin{equation}\label{phik}
\phi_k(x;\alpha_1,\alpha_2):=\tilde{\phi}_k(x;\alpha_1,\alpha_2)+\tilde{\phi}_k(x;\alpha_2,\alpha_1),
\end{equation}
\begin{multline}\label{eq:tildf}
\tilde{\phi}_k(x;\alpha_1,\alpha_2):=\frac{(2\pi)^{1+\alpha_1+\alpha_2}}{2\sin{\pi\frac{\alpha_1-\alpha_2}{2}}}
\frac{\Gamma(k-\alpha_1)}{\Gamma(1-\alpha_1+\alpha_2)\Gamma(k+\alpha_1)}
x^{(\alpha_2-\alpha_1)/2}(1-x)^{-(\alpha_1+\alpha_2)/2}\\\times
\HyG(k-\alpha_1,1-k-\alpha_1,1-\alpha_1+\alpha_2;x),
\end{multline}
\begin{multline}\label{Phi_k2}
\Phi_k(x;\alpha_1,\alpha_2):=2(2\pi)^{\alpha_1+\alpha_2}
\cos{\pi\frac{\alpha_1+\alpha_2}{2}}x^k(1-x)^{-(\alpha_1+\alpha_2)/2}\HyGI(k-\alpha_1,k-\alpha_2,2k;x).
\end{multline}
\end{thm}
\begin{proof}
The proof of the theorem formulated in slightly different notations can be found in \cite[Theorem 4.2]{BFJEMS}. To obtain the stated version, one needs to take $u=(\alpha_1+\alpha_2)/2$ and  $v=(\alpha_1-\alpha_2)/2$ in \cite[Theorem 4.2]{BFJEMS} and to use  the relation \eqref{tau to sigma}.
After that it is left to prove a new relation \eqref{ET2 3}. It follows from \cite[(4.6),(4.10)]{BFJEMS} that
\begin{equation}\label{ET2 3psi}
\ET^{(3)}_2(l;\alpha_1,\alpha_2)=
\frac{(-1)^k}{\sqrt{l}}\sum_{n=1}^{\infty}
\frac{\sigma_{\alpha_1-\alpha_2}(n)\sigma_{\alpha_1+\alpha_2}(n+l)}{n^{(\alpha_1-\alpha_2)/2}(n+l)^{(\alpha_1+\alpha_2)/2}}
\psi_k\left(\frac{l}{n};\alpha_1,\alpha_2\right),
\end{equation}
where
\begin{multline}\label{psi_k}
\psi_k(x;\alpha_1,\alpha_2)=2(2\pi)^{\alpha_1+\alpha_2}
\cos{\pi\frac{\alpha_1-\alpha_2}{2}}x^k(1+x)^{-(\alpha_1+\alpha_2)/2}\HyGI(k-\alpha_2,k-\alpha_1,2k;-x).
\end{multline}
Applying \cite[15.8.1]{HMF}, we find that
\begin{multline}\label{psi_k to Phi k}
\psi_k(x;\alpha_1,\alpha_2)=2(2\pi)^{\alpha_1+\alpha_2}
\cos{\pi\frac{\alpha_1-\alpha_2}{2}}\frac{\Gamma(k-\alpha_2)}{\Gamma(k+\alpha_2)}
\left(\frac{x}{1+x}\right)^k(1+x)^{(\alpha_1-\alpha_2)/2}\\\times
\HyGI\left(k-\alpha_1,k+\alpha_2,2k;\frac{x}{1+x}\right)=
(2\pi)^{2\alpha_2}\frac{\Gamma(k-\alpha_2)}{\Gamma(k+\alpha_2)}\Phi_k\left(\frac{x}{1+x};\alpha_1,-\alpha_2\right).
\end{multline}
Now \eqref{ET2 3} follows from \eqref{ET2 3psi}, \eqref{psi_k to Phi k} and \eqref{ET2 2}.
\end{proof}
\section{Special functions}

In this section, we collect various results on special functions that will be used later in the proof of the main theorem.

\begin{lem}\label{Lemma phik x to 1-x}
For $0<x<1$  the following identity holds:
\begin{equation}\label{phik 1-x to x}
\phi_k(1-x;\alpha_1,\alpha_2)=(-1)^kX_k(\alpha_2)\phi_k(x;\alpha_1,-\alpha_2).
\end{equation}
\end{lem}
\begin{proof}
According to \cite[(33) p.107]{BE} one has
\begin{multline}\label{2F1(x) to 2F1(1-x)}
\HyG(k-\alpha_1,1-k-\alpha_1,1-\alpha_1+\alpha_2;x)=\\=
\frac{\Gamma(1-\alpha_1+\alpha_2)\Gamma(\alpha_1+\alpha_2)}{\Gamma(1-k+\alpha_2)\Gamma(k+\alpha_2)}\HyG(k-\alpha_1,1-k-\alpha_1,1-\alpha_1-\alpha_2;1-x)\\+
\frac{\Gamma(1-\alpha_1+\alpha_2)\Gamma(-\alpha_1-\alpha_2)}{\Gamma(1-k-\alpha_1)\Gamma(k-\alpha_1)}(1-x)^{\alpha_1+\alpha_2}
\HyG(k+\alpha_2,1-k+\alpha_2,1+\alpha_1+\alpha_2;1-x).
\end{multline}
Making the change $\alpha_1\leftrightarrow\alpha_2$ in \eqref{2F1(x) to 2F1(1-x)}  and applying to each of the hypergeometric functions the following relation (see \cite[15.8.1]{HMF}):
\begin{equation}
\HyG(a,b,c;1-x)=x^{c-a-b}\HyG(c-a,c-b,c;1-x),
\end{equation}
we obtain
\begin{multline}\label{2F1(x) to 2F1(1-x)2}
\HyG(k-\alpha_2,1-k-\alpha_2,1-\alpha_2+\alpha_1;x)=\\=
\frac{\Gamma(1-\alpha_2+\alpha_1)\Gamma(\alpha_1+\alpha_2)}{\Gamma(1-k+\alpha_1)\Gamma(k+\alpha_1)}x^{\alpha_2-\alpha_1}
\HyG(k-\alpha_1,1-k-\alpha_1,1-\alpha_1-\alpha_2;1-x)\\+
\frac{\Gamma(1-\alpha_2+\alpha_1)\Gamma(-\alpha_1-\alpha_2)}{\Gamma(1-k-\alpha_2)\Gamma(k-\alpha_2)}(1-x)^{\alpha_1+\alpha_2}x^{\alpha_2-\alpha_1}
\HyG(k+\alpha_2,1-k+\alpha_2,1+\alpha_1+\alpha_2;1-x).
\end{multline}
Using \eqref{phik}, \eqref{eq:tildf}, \eqref{2F1(x) to 2F1(1-x)}, \eqref{2F1(x) to 2F1(1-x)2}  and the relations
\begin{equation}
\frac{\Gamma(k-\alpha_1)}{\Gamma(1-k+\alpha_2)}-\frac{\Gamma(k-\alpha_2)}{\Gamma(1-k+\alpha_1)}=
\frac{2(-1)^k}{\pi}\Gamma(k-\alpha_1)\Gamma(k-\alpha_2)\sin\pi\frac{\alpha_1-\alpha_2}{2}\cos\pi\frac{\alpha_1+\alpha_2}{2},
\end{equation}
\begin{equation}
\frac{1}{\Gamma(k+\alpha_1)\Gamma(1-k-\alpha_1)}-
\frac{1}{\Gamma(k+\alpha_2)\Gamma(1-k-\alpha_2)}=
\frac{2(-1)^k}{\pi}\sin\pi\frac{\alpha_1-\alpha_2}{2}\cos\pi\frac{\alpha_1+\alpha_2}{2},
\end{equation}
which follow from \cite[5.5.3]{HMF}, we show that
\begin{multline}\label{phik to 2F1(1-x)}
\phi_k(x;\alpha_1,\alpha_2)=
\frac{(-1)^k}{\pi}\cos\pi\frac{\alpha_1+\alpha_2}{2}\Gamma(\alpha_1+\alpha_2)
(2\pi)^{1+\alpha_1+\alpha_2}x^{(\alpha_2-\alpha_1)/2}(1-x)^{-(\alpha_1+\alpha_2)/2}\\\times
\frac{\Gamma(k-\alpha_1)\Gamma(k-\alpha_2)}{\Gamma(k+\alpha_1)\Gamma(k+\alpha_2)}
\HyG(k-\alpha_1,1-k-\alpha_1,1-\alpha_1-\alpha_2;1-x)\\+
\frac{(-1)^k}{\pi}\cos\pi\frac{\alpha_1+\alpha_2}{2}\Gamma(-\alpha_1-\alpha_2)
(2\pi)^{1+\alpha_1+\alpha_2}\\\times
x^{(\alpha_2-\alpha_1)/2}(1-x)^{(\alpha_1+\alpha_2)/2}\HyG(k+\alpha_2,1-k+\alpha_2,1+\alpha_1+\alpha_2;1-x).
\end{multline}
Applying \cite[5.5.3]{HMF}, we infer that
\begin{equation}\label{Gamma1}
\frac{\Gamma(\alpha_1+\alpha_2)}{\pi}\cos\pi\frac{\alpha_1+\alpha_2}{2}=
\frac{1}{2\sin\pi\frac{\alpha_1+\alpha_2}{2}\Gamma(1-\alpha_1-\alpha_2)},
\end{equation}
\begin{equation}\label{Gamma2}
\frac{\Gamma(-\alpha_1-\alpha_2)}{\pi}\cos\pi\frac{\alpha_1+\alpha_2}{2}=
\frac{1}{2\sin\pi\frac{-\alpha_1-\alpha_2}{2}\Gamma(1+\alpha_1+\alpha_2)}.
\end{equation}
Substituting \eqref{Gamma1}, \eqref{Gamma2} to \eqref{phik to 2F1(1-x)} and using \eqref{phik}, \eqref{eq:tildf}, we finally prove that
\begin{equation}\label{phik x to 1-x}
\phi_k(x;\alpha_1,\alpha_2)=(-1)^kX_k(\alpha_2)\phi_k(1-x;\alpha_1,-\alpha_2).
\end{equation}
\end{proof}

Let
\begin{equation}\label{Gamma(k alpha s) def}
\Gamma(k,\alpha_1,\alpha_2;s):=
\frac{\Gamma(k-1/2+s/2)}{\Gamma(k+1/2-s/2)}\Gamma(1/2-\alpha_1-s/2)\Gamma(1/2-\alpha_2-s/2),
\end{equation}
\begin{equation}\label{J def}
J_s(k,\alpha_1,\alpha_2;x):=\frac{1}{2\pi i}\int_{(\sigma)}\Gamma(k,\alpha_1,\alpha_2;s)\sin\left(\pi\frac{s+\alpha_1+\alpha_2}{2}\right)x^{s/2}ds,
\end{equation}
where $1-2k<\sigma<1-2\max(|\Re{\alpha_1}|,|\Re{\alpha_2}|)$. The next lemma is proven in \cite[Lemma 2.2]{FrolSbor2020}.
\begin{lem}\label{Lemma phik Mellin}
For $0<x<1$ and $\Re{(\alpha_1+\alpha_2)}>0$, $|\Re{\alpha_1}|<k$, $|\Re{\alpha_2}|<k$ 
\begin{equation}\label{phik Mellin}
\phi_k(x;\alpha_1,\alpha_2)=(2\pi)^{\alpha_1+\alpha_2}x^{-1/2+\alpha_1/2+\alpha_2/2}(1-x)^{-\alpha_1/2-\alpha_2/2}J_s(k,\alpha_1,\alpha_2;x),
\end{equation}
\begin{equation}\label{Phik Mellin}
\Phi_k(x;\alpha_1,\alpha_2)=(-1)^k(2\pi)^{\alpha_1+\alpha_2}x^{1/2}(1-x)^{-\alpha_1/2-\alpha_2/2}J_s(k,\alpha_1,\alpha_2;1/x).
\end{equation}
\end{lem}


\begin{lem}\label{lem:phiPhi approx}
The following asymptotic formulas hold:
\begin{multline}\label{phi approx}
\phi_k(x;\alpha_1,\alpha_2)\sim x^{-|\Re{(\alpha_1-\alpha_2)}|/2}\hbox{ as } x\rightarrow0,\quad
\phi_k(x;\alpha_1,\alpha_2)\sim(1-x)^{-|\Re{(\alpha_1+\alpha_2)}|/2}\hbox{ as } x\rightarrow1,
\end{multline}
\begin{multline}\label{Phi approx}
\Phi_k\left(x;\alpha_1,\alpha_2\right)\sim x^k\hbox{ as } x\rightarrow0,\quad
\Phi_k\left(x;\alpha_1,\alpha_2\right)\sim(1-x)^{-|\Re{(\alpha_1+\alpha_2)}|/2}\hbox{ as } x\rightarrow1.
\end{multline}
\end{lem}
\begin{proof}
In case when $x\rightarrow1$, the formula \eqref{phi approx} follows from \eqref{phik} and \eqref{eq:tildf} after using \cite[15.4.20, 15.4.23]{HMF}.

In case when $x\rightarrow0$, the formula \eqref{phi approx} follows from the series representation for the hypergeometric function.  One can prove \eqref{Phi approx} in a similar way.
\end{proof}
\begin{lem}\label{lem:difphik}
For $0<x<1$ the function
\begin{equation}\label{Y tildephik}
Y(x;\alpha_1,\alpha_2):=\sqrt{x(1-x)}\tilde{\phi}_k(x;\alpha_1,\alpha_2)
\end{equation}
is a solution of the differential equation
\begin{equation}\label{Ydifeq}
Y''(x)+\De(x;\alpha_1,\alpha_2)Y(x)=0,
\end{equation}
where
\begin{multline}\label{Ydifeq r-def}
\De(x;\alpha_1,\alpha_2)=\frac{(2k-1)^2}{4x(1-x)}+
\frac{1-(\alpha_2-\alpha_1)^2}{4x^2}+
\frac{1-(\alpha_2+\alpha_1)^2}{4(1-x)^2}\\+
\frac{1-(\alpha_2-\alpha_1)^2-(\alpha_2+\alpha_1)^2}{4x(1-x)}.
\end{multline}
\end{lem}
\begin{proof}
This is a generalization of \cite[Corollary 5.6]{BFJEMS}, see also \cite[Section 2.7.2, (8), (9)]{BE}.
\end{proof}
\begin{lem}\label{lem:phikest}

Let $|\alpha_1|,|\alpha_2|\ll1$. For $0<x\ll k^{-2+\epsilon}$ one has
\begin{equation}\label{tildf est1}
\tilde{\phi}_k(x;\alpha_1,\alpha_2)\ll k^{-2\Re{\alpha_1}}
x^{(\Re{\alpha_2}-\Re{\alpha_1})/2}.
\end{equation}
For $k^{-2+\epsilon}\ll x<1-\delta$ one has
\begin{equation}\label{tildf est2}
\tilde{\phi}_k(x;\alpha_1,\alpha_2)\ll
\frac{1}{k^{1/2+\Re{\alpha_2}+\Re{\alpha_1}}x^{1/4}}.
\end{equation}
\end{lem}
\begin{proof}
It follows from \cite[Lemma 3.10, Lemma 3.12, (3.56)]{BCF} that for $0<x\ll k^{-2+\epsilon}$ one has
\begin{equation}\label{phik 2F1 est1}
\HyG(k-\alpha_1,1-k-\alpha_1,1-\alpha_1+\alpha_2;x)\ll1,
\end{equation}
and for $k^{-2+\epsilon}\ll x<\delta$
\begin{multline}\label{phik 2F1 est2}
\HyG(k-\alpha_1,1-k-\alpha_1,1-\alpha_1+\alpha_2;x)\ll\frac{1}{(k\arccos\sqrt{1-2x})^{1/2+\Re{\alpha_2}-\Re{\alpha_1}}}\\ \ll
\frac{1}{(k\sqrt{x})^{1/2+\Re{\alpha_2}-\Re{\alpha_1}}}.
\end{multline}
For $\delta<x<1-\delta$ the estimate \eqref{phik 2F1 est2} follows from \cite[Lemma 3.9]{BCF}. Finally, using \eqref{eq:tildf} we prove the lemma.
\end{proof}

Similar estimates are also valid when $\alpha_1=\alpha_2=0$. The following result is a consequence of  \cite[Theorems 5.10, 5.14]{BFJEMS}.

\begin{lem}\label{lem:phikest0}
The following estimates hold:
\begin{equation}
\phi_k(x)\ll\begin{cases}
\log(k\sqrt{x}),\quad  0<x\ll k^{-2+\epsilon};\\
\log(k\sqrt{1-x}),\quad 1-k^{-2+\epsilon}<x<1;\\
k^{-1/2}x^{-1/4},\quad k^{-2+\epsilon}\ll x<1/2;\\
k^{-1/2}(1-x)^{-1/4},\quad 1/2<x<1-k^{-2+\epsilon}.
\end{cases}
\end{equation}
\end{lem}

\begin{lem}\label{lem:Phikest}
Let $u=k-1/2$. Then for $|\alpha_1|,|\alpha_2|\ll1$, $\xi>0$ and $k\rightarrow+\infty$ one has
\begin{equation}\label{Phik est}
\Phi_k\left(\frac{1}{\cosh^2{\left(\sqrt{\xi}/2\right)}};\alpha_1,\alpha_2\right)\left( \xi\sinh^2{\left(\sqrt{\xi}\right)}\right)^{1/4}
\ll \frac{\sqrt{\xi}}{k^{\alpha_1+\alpha_2}}K_{\alpha_1+\alpha_2}(u\sqrt{\xi}).
\end{equation}
\end{lem}
\begin{proof}
As in the proof of Lemma \ref{lem:phikest} (see also \cite[Lemma 3.10, Lemma 3.12]{BCF}) one may try to refer to the results in \cite{KD2}. Now we briefly explain the difficulties arising  on this way. First, one needs to transform the hypergeometric function in \eqref{Phi_k2} to the one of the form treated in \cite[Table 1-3]{KD2}. A possible way to do this is to apply \cite[Section 2.9,  (7), (23), (33)]{BE}, which results in a combination of the functions
\begin{equation}\label{Phik to 2F1 KD1}
\HyG(k\pm\alpha_1,1-k\pm\alpha_1,1\pm\alpha_1\pm\alpha_2;1-x^{-1}).
\end{equation}
An asymptotic formula for \eqref{Phik to 2F1 KD1} in terms of the $I$-Bessel functions is provided by \cite[Theorem 3.1]{KD2}. It turns out that \eqref{Phik to 2F1 KD1} grows exponentially, while $\Phi_k(x,\alpha_1,\alpha_2)$ should decay exponentially. This means that there is a massive cancellation between functions \eqref{Phik to 2F1 KD1} in their linear combination representing $\Phi_k$. This makes impossible to apply the results of \cite{KD2}.
Therefore, we are going to adapt the results of \cite[Section 5.3]{BFJEMS}.  As in \cite[Lemma 5.15]{BFJEMS} the function
\begin{equation}\label{yPhi def}
y(x;\alpha_1,\alpha_2):=\Phi_k\left(x;\alpha_1,\alpha_2\right)\sqrt{1-x}
\end{equation}
satisfies the differential equation
\begin{equation}\label{eq:diffurf21}
y''(x)-(u^2f(x)+g(x))y(x)=0,
\end{equation}
\begin{equation}\label{newfg}
f(x):=\frac{1}{x^2(1-x)}, \quad g(x):=-\frac{1}{4x^2}-\frac{1-(\alpha_1+\alpha_2)^2}{4(1-x)^2}-\frac{1-4\alpha_1\alpha_2}{4x(1-x)}.
\end{equation}
Note that the function $f(x)$ coincides with the one in \cite[(5.65)]{BFJEMS} and the function $g(x)$ differs slightly.  Making the same change of variable as in
\cite[(5.70)]{BFJEMS}:
\begin{equation}
\xi:=4 \artanh^2{(\sqrt{1-x})},
\end{equation}
we obtain that the function (see \cite[(5.69)]{BFJEMS})
\begin{equation}
Z(x):=\frac{y(x)}{\alpha(x)}, \quad \alpha(x):=\frac{(x^2-x^3)^{1/4}}{2(\artanh{(\sqrt{1-x})})^{1/2}}
\end{equation}
satisfies
\begin{equation}\label{eq:diffurZpsi}
\frac{d^2Z}{d\xi^2}+\left[ -\frac{u^2}{4\xi}+\frac{1}{4\xi^2}-\frac{\psi(\xi)}{\xi}\right]Z=0,
\end{equation}
where
\begin{equation}\label{psi difeq}
\psi(\xi)=\frac{1}{16}\left(\frac{1}{\xi}-\frac{1-4(\alpha_1+\alpha_2)^2}{\sinh^2{(\sqrt{\xi})}}+\frac{4\alpha_1\alpha_2}{\cosh^2{(\sqrt{\xi}/2)}} \right).
\end{equation}
The function \eqref{psi difeq} is only slightly different from \cite[(5.68)]{BFJEMS}, however is no more analytic at $\xi=0$.  To overcome this difficulty, we change the form of an
approximating  differential equation from \cite[(5.71)]{BFJEMS} to
\begin{equation}\label{model difeq}
Z''+\left(-\frac{u^2}{4\xi}+\frac{1-(\alpha_1+\alpha_2)^2}{4\xi} \right)Z=0.
\end{equation}
So  we rewrite \eqref{eq:diffurZpsi}, \eqref{psi difeq} in the form
\begin{equation}\label{eq:diffurZpsi2}
\frac{d^2Z}{d\xi^2}+\left[ -\frac{u^2}{4\xi}+\frac{1-(\alpha_1+\alpha_2)^2}{4\xi^2}-\frac{\psi(\xi)}{\xi}\right]Z=0,
\end{equation}
where
\begin{equation}\label{psi difeq2}
\psi(\xi)=\frac{1-4(\alpha_1+\alpha_2)^2}{16}\left(\frac{1}{\xi}-\frac{1}{\sinh^2{(\sqrt{\xi})}}\right)+\frac{\alpha_1\alpha_2}{4\cosh^2{(\sqrt{\xi}/2)}}.
\end{equation}
Now the function $\psi(\xi)$ is analytic at $\xi=0$, and we can perform all computations in the same way as in \cite[Section 5.3]{BFJEMS}.
Solutions of \eqref{model difeq} are given by (see \cite[Section 10.36, eq.~10.13.2]{HMF}):
\begin{equation}
Z_L=\sqrt{\xi}L_{\alpha_1+\alpha_2}(u\sqrt{\xi}),
\end{equation}
where
\begin{equation}
L_v(x):=\begin{cases}
I_v(x)\\e^{\pi i v}K_v(x)
\end{cases}.
\end{equation}
Arguing in the same way as in  \cite[Theorem 5.16]{BFJEMS} and using the analogue of   \cite[(5.92),(5.93)]{BFJEMS}, we obtain
\begin{equation}\label{eq:phiklim}
\Phi_k\left(\frac{1}{\cosh^2{\left(\sqrt{\xi}/2\right)}} \right)\left( \xi\sinh^2{\left(\sqrt{\xi}\right)}\right)^{1/4}=
C_KZ_K(\xi),
\end{equation}
where
\begin{equation}\label{eq:ckexpl}
C_K=2\frac{\Gamma(k-\alpha_1)\Gamma(k-\alpha_2)}{\Gamma(2k)}\frac{2^{2k}\sqrt{u}}{\sqrt{\pi}}\frac{(2\pi)^{\alpha_1+\alpha_2}\cos\pi\frac{\alpha_1+\alpha_2}{2}}{e^{\pi i(\alpha_1+\alpha_2)}}\left(1+O\left(\frac{1}{u}\right)\right)\ll\frac{1}{k^{\Re{\alpha_1}+\Re{\alpha_2}}},
\end{equation}
and $Z_K(\xi)$ is the same as in \cite[(5.86)]{BFJEMS} with $K_0(x)$, $K_1(x)$ replaced by $K_{\alpha_1+\alpha_2}(x),$ $K_{1+\alpha_1+\alpha_2}(x)$, respectively.
Note the coefficients $A_K(n;\xi)$, $B_K(n;\xi)$ now also depend on $\alpha_1,\alpha_2$, but if $|\alpha_1|,|\alpha_2|<\epsilon$, then they are still uniformly bounded by some constant. It follows from the structure of  $Z_K(\xi)$  that it is bounded by $\sqrt{\xi}K_{\alpha_1+\alpha_2}(u\sqrt{\xi})$.
This completes the proof of \eqref{Phik est}.
\end{proof}
Applying  \eqref{Phik est} and \cite[10.25.3,10.30.3]{HMF}, we prove the following estimate.
\begin{cor}\label{cor:Phik est}
For $|\alpha_1|,|\alpha_2|\ll1$ and $x\gg k^{-2+\epsilon}$ one has
\begin{equation}\label{Phik LGest}
\Phi_k\left(\frac{1}{1+x};\alpha_1,\alpha_2\right)\ll
\frac{e^{-(2k-1)\log(\sqrt{1+x}+\sqrt{x})}}{(1+x)^{1/4}x^{1/4}k^{1/2+\Re{\alpha_1}+\Re{\alpha_2}}}.
\end{equation}
If $\alpha_1=\alpha_2=0$ then  one has
\begin{equation}\label{Phik0 LGest0}
\Phi_k\left(x\right)\ll
\frac{x^{1/2}e^{-(2k-1)\log((1+\sqrt{1-x})/\sqrt{x})}}{(1-x)^{1/4}k^{1/2}} \quad \text{ for } x<1-k^{-2+\epsilon},
\end{equation}
 and
\begin{equation}\label{Phik0 LGest1}
\Phi_k\left(x\right)\ll\log(k\sqrt{1-x})\quad \text{ for }1-k^{-2+\epsilon}<x<1.
\end{equation}

\end{cor}

Now we are going to obtain an asymptotic formula for the hypergeometric function that will arise in the additive divisor problem. This result is an analogue of
\cite[Lemma 2.4]{Zav}.

\begin{lem}\label{lem:W0 2F1}
For $|\alpha|,|\beta|\ll1$, $y>y_0>0$ and $r\rightarrow+\infty$ one has
\begin{multline}\label{W02F1 asympt}
\HyG\left(\frac{1+\alpha-\beta}{2}+ir,\frac{1-\alpha-\beta}{2}+ir,1+2ir;\frac{-1}{y}\right)=\\=
\frac{(4y)^{ir}y^{1/4-\alpha+\beta/2}}{(\sqrt{1+y}+\sqrt{y})^{2ir}(1+y)^{1/4-\alpha+\beta/2}}\left(1+\sum_{n=1}^N\frac{a_n(y)}{r^n}\right)+
O\left(\frac{1}{yr^{N+1}}\right),
\end{multline}
where $a_n(y)\ll y^{-1}$.
\end{lem}
\begin{proof}
Applying \cite[15.8.1]{HMF}, we obtain
\begin{multline}\label{2F1(-)to2F1(+)}
\HyG\left(\frac{1+\alpha-\beta}{2}+ir,\frac{1-\alpha-\beta}{2}+ir,1+2ir;\frac{-1}{y}\right)=\\=
\left(\frac{1+y}{y}\right)^{-1/2-\alpha/2+\beta/2-ir}
\HyG\left(\frac{1+\alpha-\beta}{2}+ir,\frac{1+\alpha+\beta}{2}+ir,1+2ir;\frac{1}{1+y}\right).
\end{multline}
The function
\begin{equation}\label{y(x)def}
y(x)=x^{1/2+ir}(1-x)^{1/2+\alpha/2}\HyG\left(\frac{1+\alpha-\beta}{2}+ir,\frac{1+\alpha+\beta}{2}+ir,1+2ir;x\right)
\end{equation}
satisfies \cite[p.96]{BE} the differential equation
\begin{equation}\label{y(x)diff eq}
y''(x)+H(x)y(x)=0,
\end{equation}
where
\begin{equation}\label{H(x)def}
H(x)=\frac{r^2}{x^2(1-x)}+\frac{1}{4x^2}+\frac{1-\alpha^2}{4(1-x)^2}+\frac{1+\beta^2-\alpha^2}{4x(1-x)}.
\end{equation}
Arguing in the same way as in \cite[(5.20)-(5.23)]{BFJEMS} and choosing  $\alpha(x)=x^{1/2}(1-x)^{1/4}$ (therefore, making the change of variables  $x=\cosh^{-2}\frac{\xi}{2}$), we obtain that the function
\begin{equation}\label{U def}
U(\xi)=\frac{\tanh^{1/2+\alpha}(\xi/2)}{\cosh^{2ir}(\xi/2)}
\HyG\left(\frac{1+\alpha-\beta}{2}+ir,\frac{1+\alpha+\beta}{2}+ir,1+2ir;\frac{1}{\cosh^2\frac{\xi}{2}}\right)
\end{equation}
satisfies
\begin{equation*}
U''(\xi)+\left( r^2+\frac{1-4\alpha^2}{4\sinh^2\xi}+\frac{\beta^2-\alpha^2}{4\cosh^2{\xi/2}}\right)U(\xi)=0.
\end{equation*}
Either arguing in the same way as in \cite[Lemma 2.4]{Zav} or rather applying \cite[Section 10.3, Theorem 3.1]{O} and the discussions afterwards,
one can prove that
\begin{equation}\label{U asympt}
U(\xi)=2^{2ir}e^{-ir\xi}\left(1+\sum_{n=1}^N\frac{A_n(\xi)}{ir}\right)+O\left(\frac{1}{e^{\xi}r^{N+1}}\right).
\end{equation}
Now \eqref{W02F1 asympt}  follows from \eqref{2F1(-)to2F1(+)}, \eqref{U def} and \eqref{U asympt} after some straightforward computations.
\end{proof}
Note that the result \eqref{W02F1 asympt} could be obtained with the use of the saddle point method like in \cite[p.173]{Jut}.

For the sake of simplicity, let us introduce some additional notation:
\begin{equation}\label{k0 def}
k_0(x,v):=\frac{1}{2\cos(\pi (1/2+v))}\left(J_{2v}(x)-J_{-2v}(x) \right),
\end{equation}
\begin{equation}
k_1(x,v):=\frac{2}{\pi}\sin(\pi(1/2+v))K_{2v}(x).
\end{equation}
According to \cite[Lemma 2.2]{BFJEMS} the following result holds.
\begin{lem}\label{lemma:mellinkernels}
Let \begin{equation}\label{eq:gammauv}
\gamma(u,v):=\frac{2^{2u-1}}{\pi}\Gamma(u+v)\Gamma(u-v).
\end{equation}
For $3/2>Re{w}> 2|\Re{v}|$ one has
\begin{equation}\label{k0 Mellin}
\int_{0}^{\infty}k_0(x,v)x^{w-1}dx=\gamma(w/2,v)\cos{(\pi w/2)},
\end{equation}
and  for $Re{w}> 2|\Re{v}|$ one has
\begin{equation}\label{k1 Mellin}
\int_{0}^{\infty}k_1(x,v)x^{w-1}dx=\gamma(w/2,v)\sin{(\pi (1/2+v))}.
\end{equation}
\end{lem}

\begin{cor}\label{cor:mellinkernels}
For $\Delta_0,\Delta_1> 2|\Re{v}|$ the following formulas hold:
\begin{equation}\label{k0 Mellin inv}
k_0(x,v)=\frac{1}{2\pi i}\int_{(\Delta_0)}\gamma(w/2,v)\cos{(\pi w/2)}x^{-w}dw,
\end{equation}
\begin{equation}\label{k1 Mellin inv}
k_1(x,v)=\frac{1}{2\pi i}\int_{(\Delta_1)}\gamma(w/2,v)\sin{(\pi (1/2+v))}x^{-w}dw.
\end{equation}
For $\Delta_0,\Delta_1<\min(1,1+2\Re{v})$ one has
\begin{equation}\label{k0 Mellin inv2}
k_0(x,v)=\frac{2^{2+2v}}{\pi x^{2+2v}}\frac{1}{2\pi i}\int_{(\Delta_0)}\Gamma(1-s)
\Gamma(1+2v-s)\cos\pi(1+v-s)\frac{x^{2s}}{4^s}ds,
\end{equation}
\begin{equation}\label{k1 Mellin inv2}
k_1(x,v)=\frac{2^{2+2v}\sin{(\pi (1/2+v))}}{\pi x^{2+2v}}\frac{1}{2\pi i}\int_{(\Delta_1)}\Gamma(1-s)
\Gamma(1+2v-s)\frac{x^{2s}}{4^s}ds.
\end{equation}
\end{cor}
\begin{proof}
Formulas \eqref{k0 Mellin inv}, \eqref{k1 Mellin inv} follow immediately from \eqref{k0 Mellin}, \eqref{k1 Mellin} via the Mellin inversion.

Making the change of variables $\frac{w}{2}=1+v-s$ in \eqref{k0 Mellin inv}, \eqref{k1 Mellin inv}, we prove \eqref{k0 Mellin inv2}, \eqref{k1 Mellin inv2}.
\end{proof}

\section{The binary additive divisor problem}\label{sec:ADP}
Following Motohashi's book \cite[p. 531]{Mo}, let
\begin{equation}\label{BADP def1}
A(l,\alpha,\beta;W):=\sum_{n=1}^{\infty}\sigma_{\alpha}(n)\sigma_{\beta}(n+l)W\left(\frac{n}{l}\right).
\end{equation}
The problem of investigating the asymptotic  behaviour of \eqref{BADP def1} has a long history (see \cite{BF2} for references). It is shown in \cite{BF2} that the best known estimate on the reminder in the asymptotic formula for \eqref{BADP def1} can be obtained, among other methods, using the results of \cite{Mo}. We will apply an exact formula for \eqref{BADP def1} proved by Motohashi \cite{Mo} (a similar formula has also been obtained by Kuznetsov \cite{Kuz83}).

\begin{thm}\label{Th:BADP}
Assume that $-1-\Re{\beta}<\Re{\alpha}<1+\Re{\beta}$. Let $W$ be a sufficiently good function such that all the series and integrals below are absolutely convergent. Then
\begin{equation}\label{Al=U+error}
A(l,\alpha,\beta;W)=U(l,\alpha,\beta;W)+E_{h}(l,\alpha,\beta;W)+E_{d}(l,\alpha,\beta;W)+E_{c}(l,\alpha,\beta;W),
\end{equation}
where
\begin{equation}\label{Udef}
U(l,\alpha,\beta;W)=\int_0^{\infty}W(x)\mu(l,\alpha,\beta;x)dx,
\end{equation}
\begin{multline}\label{mul def}
\mu(l,\alpha,\beta;x)=\\
\sigma_{1+\alpha+\beta}(l)\frac{\zeta(1+\alpha)\zeta(1+\beta)}{\zeta(2+\alpha+\beta)}x^{\alpha}(1+x)^{\beta}+
l^{\alpha}\sigma_{1-\alpha+\beta}(l)\frac{\zeta(1-\alpha)\zeta(1+\beta)}{\zeta(2-\alpha+\beta)}(1+x)^{\beta}+\\
l^{\beta}\sigma_{1+\alpha-\beta}(l)\frac{\zeta(1+\alpha)\zeta(1-\beta)}{\zeta(2+\alpha-\beta)}x^{\alpha}+
l^{\alpha+\beta}\sigma_{1-\alpha-\beta}(l)\frac{\zeta(1-\alpha)\zeta(1-\beta)}{\zeta(2-\alpha-\beta)},
\end{multline}
\begin{multline}\label{Eh def}
E_{h}(l,\alpha,\beta;W)=2(2\pi)^{\beta-1}l^{(1+\alpha+\beta)/2}
\sum_{m=6}^{\infty}\frac{2m-1}{\pi^2}\sum_{f\in H_{2m}} \omega_f\lambda_f(l)L_f\left(\frac{1+\alpha-\beta}{2}\right)\\\times
L_f\left(\frac{1-\alpha-\beta}{2}\right)\widehat{W}(\alpha,\beta;m),
\end{multline}
\begin{multline}\label{Ed def}
E_{d}(l,\alpha,\beta;W)=2(2\pi)^{\beta-1}l^{(1+\alpha+\beta)/2}
\sum_{j} \omega_j\lambda_j(l)L_j\left(\frac{1+\alpha-\beta}{2}\right)\\\times
L_j\left(\frac{1-\alpha-\beta}{2}\right)
\left(\widehat{W}^{+}(\alpha,\beta;t_j)+\varepsilon_j\widehat{W}^{-}(\alpha,\beta;t_j)\right),
\end{multline}
\begin{multline}\label{Ec def}
E_{c}(l,\alpha,\beta;W)=4(2\pi)^{\beta-1}l^{(1+\alpha+\beta)/2}\frac{1}{2\pi i}
\int_{(0)}\frac{\sigma_{2\xi}(l)\ZC(\alpha,\beta,\xi)}{l^{\xi}\zeta(1+2\xi)\zeta(1-2\xi)}\\\times
\left(\widehat{W}^{+}(\alpha,\beta;-i\xi)+\widehat{W}^{-}(\alpha,\beta;-i\xi)\right)d\xi
\end{multline}
with
\begin{multline}\label{Z4 def}
\ZC(\alpha,\beta,\xi)=\zeta\left(\frac{1-\alpha-\beta}{2}+\xi\right)
\zeta\left(\frac{1+\alpha-\beta}{2}+\xi\right)
\zeta\left(\frac{1-\alpha-\beta}{2}-\xi\right)\\ \times
\zeta\left(\frac{1+\alpha-\beta}{2}-\xi\right),
\end{multline}

\begin{equation}\label{hatW m}
\widehat{W}(\alpha,\beta;m)=\frac{\pi(-1)^m}{2}\cos\frac{\pi\alpha}{2}\int_0^{\infty}\frac{W(y)}{y^{m-\frac{\alpha+\beta}{2}}}
\HyGI\left(m+\alpha-\frac{\beta}{2},m-\alpha-\frac{\beta}{2},2m;\frac{-1}{y}\right)dy,
\end{equation}
\begin{equation}\label{hatW+}
\widehat{W}^{+}(\alpha,\beta;r)=
\frac{\pi i\cos\frac{\pi\alpha}{2}}{4\sinh(\pi r)}\widehat{W}^{0}(\alpha,\beta;ir)-
\frac{\pi i\cos\frac{\pi\alpha}{2}}{4\sinh(\pi r)}
\widehat{W}^{0}(\alpha,\beta;-ir),
\end{equation}
\begin{equation}\label{hatW-}
\widehat{W}^{-}(\alpha,\beta;r)=
\frac{-\pi i\sin\pi(ir-\beta/2)}{4\sinh(\pi r)}
\widehat{W}^{0}(\alpha,\beta;ir)-
\frac{\pi i\sin\pi(ir+\beta/2)}{4\sinh(\pi r)}\widehat{W}^{0}(\alpha,\beta;-ir),
\end{equation}
\begin{multline}\label{hatW+ + hatW-}
\widehat{W}^{+}(\alpha,\beta;-i\xi)+\widehat{W}^{-}(\alpha,\beta;-i\xi)=
\pi\frac{\sin\pi(\xi-\beta/2)-\cos\frac{\pi\alpha}{2}}{4\sin(\pi\xi)}
\widehat{W}^{0}(\alpha,\beta;\xi)\\+
\pi\frac{\sin\pi(\xi+\beta/2)+\cos\frac{\pi\alpha}{2}}{4\sin(\pi\xi)}
\widehat{W}^{0}(\alpha,\beta;-\xi),
\end{multline}
\begin{equation}\label{hatW0}
\widehat{W}^{0}(\alpha,\beta;t)=
\int_{0}^{\infty}W(y)y^{\frac{\alpha+\beta-1}{2}-t}
\HyGI\left(\frac{1+\alpha-\beta}{2}+t,\frac{1-\alpha-\beta}{2}+t,1+2t;\frac{-1}{y}\right)dy.
\end{equation}
\end{thm}
\begin{proof}
The main term, given by \eqref{Udef}, \eqref{mul def}, was evaluated in \cite[p.   554]{Mo}. To prove a new version of formulas for other terms, we write some integral transforms in \cite{Mo} in a slightly different  form. Using \eqref{k1 Mellin inv2} and \eqref{k0 Mellin inv2}, the integrals \cite[(3.14), (3.15)]{Mo} can be rewritten as
\begin{equation}\label{Mot(3.14)}
\psi_{+}(x,\alpha,\beta)=\frac{\pi x^{1-\beta}}{2^{1-\beta}}\int_0^{\infty}W(y)k_1(x\sqrt{y},\alpha/2)y^{\alpha/2}dy,
\end{equation}
\begin{equation}\label{Mot(3.15)}
\psi_{-}(x,\alpha,\beta)=\frac{\pi x^{1-\beta}}{2^{1-\beta}}\int_0^{\infty}W(y)k_0(x\sqrt{y},\alpha/2)y^{\alpha/2}dy.
\end{equation}
The integral transforms $\hat{\psi}(r)$, $\check{\psi}(r)$ in \cite[Lemma 1, Lemma 2]{Mo} can be written as follows:
\begin{equation}\label{Mot:psi hat check}
\hat{\psi}(r)=\pi\int_0^{\infty}k_0(x,ir)\frac{\psi(x)}{x}dx,\quad
\check{\psi}(r)=\pi\int_0^{\infty}k_1(x,ir)\frac{\psi(x)}{x}dx.
\end{equation}
Using \cite[10.4.1]{HMF} one has
\begin{equation}\label{Mot:psihat k}
\widehat{\psi}(i(1/2-k))=\pi(-1)^k\int_0^{\infty}J_{2k-1}\frac{\psi(x)}{x}dx.
\end{equation}
Application of \cite[Lemma 1, Lemma 2]{Mo} to \cite[(3.12), (3.13)]{Mo} provides the following analogue of \cite[(3.34)]{Mo}:
\begin{multline}\label{Eh eq1}
E_{h}(l,\alpha,\beta;W)=2(2\pi)^{\beta-1}l^{(1+\alpha+\beta)/2}
\sum_{m=6}^{\infty}\frac{2m-1}{\pi^2}\sum_{f\in H_{2m}} \omega_f\lambda_f(l)L_f\left(\frac{1+\alpha-\beta}{2}\right)\\\times
L_f\left(\frac{1-\alpha-\beta}{2}\right)\widehat{\psi}(i(1/2-m)),
\end{multline}
\begin{multline}\label{Ed eq1}
E_{d}(l,\alpha,\beta;W)=2(2\pi)^{\beta-1}l^{(1+\alpha+\beta)/2}
\sum_{j} \omega_j\lambda_j(l)L_j\left(\frac{1+\alpha-\beta}{2}\right)
L_j\left(\frac{1-\alpha-\beta}{2}\right)\\\times
\left(\widehat{\psi}^{+}(t_j)+\varepsilon_j\check{\psi}^{-}(t_j)\right),
\end{multline}
\begin{multline}\label{Ec eq1}
E_{c}(l,\alpha,\beta;W)=(2\pi)^{\beta-1}l^{(1+\alpha+\beta)/2}\frac{2}{\pi}
\int_{-\infty}^{\infty}\frac{\sigma_{2ir}(l)\ZC(\alpha,\beta,ir)}{l^{ir}\zeta(1+2ir)\zeta(1-2ir)}
\left(\widehat{\psi}^{+}(r)+\check{\psi}^{-}(r)\right)dr.
\end{multline}
Note that in order to obtain \eqref{Eh eq1} we used the following relation between the coefficients $\alpha_{j,k}$ \cite[p.535]{Mo} and \eqref{harmonic weight}:
\begin{equation}\label{alphajkMot to harmonic weight}
\alpha_{j,k}=\frac{16\Gamma(2k)}{(4\pi)^{2k+1}}|\rho_{j,k}(1)|^2=
\frac{\Gamma(2k-1)(2k-1)}{\pi^2(4\pi)^{2k-1}}|\rho_{j,k}(1)|^2=
\frac{(2k-1)}{\pi^2}\omega_f.
\end{equation}
Finally, using  \cite[6.576.3]{GR}, we prove that
\begin{equation}
\widehat{\psi}^{+}(r)=\widehat{W}^{+}(\alpha,\beta;r),\quad
\check{\psi}^{-}(r)=\widehat{W}^{-}(\alpha,\beta;r),\quad
\widehat{\psi}(i(1/2-m))=\widehat{W}(\alpha,\beta;m).
\end{equation}
\end{proof}
\begin{lem}\label{lem:hatWmtoPhi}
The following identity holds:
\begin{multline}\label{hatWmtoPhi}
\widehat{W}(\alpha,\beta;m)=\frac{\pi(-1)^m}{4(2\pi)^{\beta}}\frac{\cos\frac{\pi\alpha}{2}}{\cos\pi\alpha}X_m\left(\frac{\beta}{2}-\alpha\right)
\int_0^{\infty}W(y)y^{\frac{\alpha}{2}}(1+y)^{\frac{\beta}{2}}\\\times
\Phi_m\left(\frac{1}{1+y};\alpha-\frac{\beta}{2},\alpha+\frac{\beta}{2}\right)dy.
\end{multline}
\end{lem}
\begin{proof}
Using \cite[15.8.1]{HMF}, we show that
\begin{multline}\label{hatWm2I1 to2I1}
\HyGI\left(m+\alpha-\beta/2,m-\alpha-\beta/2,2m;\frac{-1}{y}\right)dy=
\frac{(1+y)^{-m+\alpha+\beta/2}}{y^{-m+\alpha+\beta/2}}\\\times
\frac{\Gamma(m+\alpha-\beta/2)}{\Gamma(m-\alpha+\beta/2)}
\HyGI\left(m-\alpha+\frac{\beta}{2},m-\alpha-\frac{\beta}{2},2m;\frac{1}{1+y}\right)dy.
\end{multline}
According to \eqref{Phi_k2} one has
\begin{multline}\label{2I1toPhim}
\HyGI\left(m-\alpha+\frac{\beta}{2},m-\alpha-\frac{\beta}{2},2m;\frac{1}{1+y}\right)\\=
\Phi_m\left(\frac{1}{1+y};\alpha-\frac{\beta}{2},\alpha+\frac{\beta}{2}\right)
\frac{(1+y)^{m-\alpha}y^{\alpha}}{2(2\pi)^{2\alpha}\cos(\pi\alpha)}.
\end{multline}
Finally, \eqref{hatWmtoPhi} follows from \eqref{hatW m}, \eqref{hatWm2I1 to2I1} and \eqref{2I1toPhim}.
\end{proof}

\section{Reduction of the fourth moment to triple divisor sums }\label{sec:4mom to triple div}
Recall that $\AlphaVec=(\alpha_1,\alpha_2,\alpha_3,\alpha_4)$. Let us define
the following triple divisor sums:
\begin{multline}\label{TD1}
\TD^{(1)}(\AlphaVec):=(-1)^k
\zeta(1+\alpha_3+\alpha_4)\sum_{m=1}^{\infty}\sum_{n=1}^{\infty}
\frac{\sigma_{\alpha_1-\alpha_2}(n)\sigma_{\alpha_1+\alpha_2}(m)\sigma_{\alpha_3-\alpha_4}(n+m)}
{m^{\alpha_1/2+\alpha_2/2}n^{\alpha_1/2-\alpha_2/2}(m+n)^{1+\alpha_3}}\\ \times\phi_k\left(\frac{n}{n+m};\alpha_1,\alpha_2\right),
\end{multline}
\begin{equation}\label{TD2}
\TD^{(2)}(\AlphaVec):=
\zeta(1+\alpha_3+\alpha_4)\sum_{l=1}^{\infty}\sum_{n=1}^{\infty}
\frac{\sigma_{\alpha_3-\alpha_4}(l)\sigma_{\alpha_1+\alpha_2}(n)\sigma_{\alpha_1-\alpha_2}(n+l)}
{n^{\alpha_1/2+\alpha_2/2}l^{1+\alpha_3}(n+l)^{\alpha_1/2-\alpha_2/2}}\Phi_k\left(\frac{l}{n+l};\alpha_1,\alpha_2\right),
\end{equation}
\begin{equation}\label{TD3}
\TD^{(3)}(\alpha_1,\alpha_2,\alpha_3,\alpha_4):=(-1)^kX_k(\alpha_2)\TD^{(2)}(\alpha_1,-\alpha_2,\alpha_3,\alpha_4).
\end{equation}

\begin{lem}
If  $\min_{j=3,4}\Re{\alpha_j}>1+\max_{j=1,2}|\Re{\alpha_j}|$, then
\begin{equation}\label{4mom to 2mom}
\M_4(\AlphaVec)=
\MT_{2,4}(\AlphaVec)+
\sum_{i=1}^3\TD^{(i)}(\AlphaVec),
\end{equation}
where
\begin{multline}\label{MT2momto4mom}
\MT_{2,4}(\AlphaVec)=
\ZF(\alpha_1,\alpha_2,\alpha_3,\alpha_4)\Co(\alpha_1,\alpha_2,\alpha_3,\alpha_4)+
\ZF(-\alpha_1,-\alpha_2,\alpha_3,\alpha_4)\Co(-\alpha_1,-\alpha_2,\alpha_3,\alpha_4)\\+
\ZF(\alpha_1,-\alpha_2,\alpha_3,\alpha_4)\Co(\alpha_1,-\alpha_2,\alpha_3,\alpha_4)+
\ZF(-\alpha_1,\alpha_2,\alpha_3,\alpha_4)\Co(-\alpha_1,\alpha_2,\alpha_3,\alpha_4).
\end{multline}
\end{lem}
\begin{proof}
Applying  \eqref{Lprod} for two out of four $L$-functions in \eqref{4moment def} and using \eqref{eq:secondmoment}, we obtain
\begin{equation}
\M_4(\AlphaVec)=
\zeta(1+\alpha_3+\alpha_4)\sum_{l=1}^{\infty}
\frac{\sigma_{\alpha_3-\alpha_4}(l)}{l^{1/2+\alpha_3}}\MT_2(l;\alpha_1,\alpha_2)+
\sum_{i=1}^3\TD^{(i)}(\AlphaVec),
\end{equation}
where
\begin{equation}
\TD^{(i)}(\AlphaVec)=
\zeta(1+\alpha_3+\alpha_4)\sum_{l=1}^{\infty}
\frac{\sigma_{\alpha_3-\alpha_4}(l)}{l^{1/2+\alpha_3}}\ET^{(i)}_2(l;\alpha_1,\alpha_2).
\end{equation}
Using \eqref{MT2 def} and \eqref{RamId}, and taking into account the notation \eqref{7zeta def} and \eqref{Co def}, one has
\begin{multline}\label{MT2 to MT4}
\zeta(1+\alpha_3+\alpha_4)\sum_{l=1}^{\infty}
\frac{\sigma_{\alpha_3-\alpha_4}(l)}{l^{1/2+\alpha_3}}\MT_2(l;\alpha_1,\alpha_2)=\\=
\ZF(\alpha_1,\alpha_2,\alpha_3,\alpha_4)\Co(\alpha_1,\alpha_2,\alpha_3,\alpha_4)+
\ZF(-\alpha_1,-\alpha_2,\alpha_3,\alpha_4)\Co(-\alpha_1,-\alpha_2,\alpha_3,\alpha_4)\\+
\ZF(\alpha_1,-\alpha_2,\alpha_3,\alpha_4)\Co(\alpha_1,-\alpha_2,\alpha_3,\alpha_4)+
\ZF(-\alpha_1,\alpha_2,\alpha_3,\alpha_4)\Co(-\alpha_1,\alpha_2,\alpha_3,\alpha_4).
\end{multline}
Substituting \eqref{ET2 1}, \eqref{ET2 2}, \eqref{ET2 3} to \eqref{4mom to 2mom}, we obtain \eqref{TD1}, \eqref{TD2}, \eqref{TD3}. Note that Lemma \ref{lem:phiPhi approx} and the conditions   $\min_{j=3,4}\Re{\alpha_j}>1+\max_{j=1,2}|\Re{\alpha_j}|$ guarantee the absolute convergence of the double series \eqref{TD1}, \eqref{TD2}, \eqref{TD3}.
\end{proof}
Let $\DU:[0,+\infty)\rightarrow[0;1]$ be a smooth  infinitely differentiable  function such that
\begin{equation}\label{DU def}
\DU(x)+\DU(1/x)=1, \DU(x)=1\,\hbox{for}\, 0\le x\le x_0,\,\quad
\DU(x)=0\,\hbox{for}\,  x>1/x_0.
\end{equation}
We insert in \eqref{TD1}, \eqref{TD2} the following decomposition of unity:
\begin{equation}
\DU\left(\frac{m}{n}\right)+\DU\left(\frac{n}{m}\right)=1,
\end{equation}
in order to decompose $\TD^{(i)}(\AlphaVec)$ for $i=1,2$ into two sums:
\begin{equation}\label{TD1=TD11+TD12}
\TD^{(i)}(\AlphaVec)=\TD^{(i,1)}(\AlphaVec)+\TD^{(i,2)}(\AlphaVec),
\end{equation}
where
\begin{equation}\label{TD11}
\TD^{(1,1)}(\AlphaVec)=(-1)^k
\zeta(1+\alpha_3+\alpha_4)\sum_{n=1}^{\infty}
\frac{\sigma_{\alpha_1-\alpha_2}(n)}{n^{1+\alpha_1+\alpha_3}}
\sum_{m=1}^{\infty}
\sigma_{\alpha_1+\alpha_2}(m)\sigma_{\alpha_3-\alpha_4}(m+n)W_{1,1}\left(\frac{m}{n}\right),
\end{equation}
\begin{equation}\label{W11def}
W_{1,1}(x)=\frac{\DU(1/x)}
{x^{\alpha_1/2+\alpha_2/2}(1+x)^{1+\alpha_3}}\phi_k\left(\frac{1}{1+x};\alpha_1,\alpha_2\right),
\end{equation}
\begin{equation}\label{TD12}
\TD^{(1,2)}(\AlphaVec)=(-1)^k
\zeta(1+\alpha_3+\alpha_4)\sum_{m=1}^{\infty}
\frac{\sigma_{\alpha_1+\alpha_2}(m)}{m^{1+\alpha_1+\alpha_3}}
\sum_{n=1}^{\infty}
\sigma_{\alpha_1-\alpha_2}(n)\sigma_{\alpha_3-\alpha_4}(n+m)W_{1,2}\left(\frac{n}{m}\right),
\end{equation}
\begin{equation}\label{W12def}
W_{1,2}(x)=\frac{\DU(1/x)}
{x^{\alpha_1/2-\alpha_2/2}(1+x)^{1+\alpha_3}}\phi_k\left(\frac{x}{1+x};\alpha_1,\alpha_2\right),
\end{equation}

\begin{equation}\label{TD21}
\TD^{(2,1)}(\AlphaVec)=
\zeta(1+\alpha_3+\alpha_4)\sum_{n=1}^{\infty}
\frac{\sigma_{\alpha_1+\alpha_2}(n)}{n^{1+\alpha_1+\alpha_3}}
\sum_{l=1}^{\infty}
\sigma_{\alpha_3-\alpha_4}(l)\sigma_{\alpha_1-\alpha_2}(l+n)W_{2,1}\left(\frac{l}{n}\right),
\end{equation}
\begin{equation}\label{W21def}
W_{2,1}(x)=\frac{\DU(1/x)}
{(1+x)^{\alpha_1/2-\alpha_2/2}x^{1+\alpha_3}}\Phi_k\left(\frac{x}{1+x};\alpha_1,\alpha_2\right),
\end{equation}
\begin{equation}\label{TD22}
\TD^{(2,2)}(\AlphaVec)=
\zeta(1+\alpha_3+\alpha_4)\sum_{l=1}^{\infty}
\frac{\sigma_{\alpha_3-\alpha_4}(l)}{l^{1+\alpha_1+\alpha_3}}
\sum_{n=1}^{\infty}
\sigma_{\alpha_1+\alpha_2}(n)\sigma_{\alpha_1-\alpha_2}(n+l)W_{2,2}\left(\frac{n}{l}\right),
\end{equation}
\begin{equation}\label{W22def}
W_{2,2}(x)=\frac{\DU(1/x)}{(1+x)^{\alpha_1/2-\alpha_2/2}x^{\alpha_1/2+\alpha_2/2}}\Phi_k\left(\frac{1}{1+x};\alpha_1,\alpha_2\right).
\end{equation}
To simplify estimates of some error terms in the forthcoming sections, it is convenient to assume that $|\Re{\alpha_1}|, |\Re{\alpha_2}|\ll\epsilon$.
\section{Analysis of $\TD^{(1,1)}$}\label{sec:TD11}
In this section, we show that the term $\TD^{(1,1)}$, defined by \eqref{TD11}, contains the following three weighted moments. The discrete  contribution is given by the fourth moment of Maass form L-functions:
\begin{multline}\label{Ed11}
\term_{d}^{1}(\AlphaVec):=
2(-1)^k(2\pi)^{\alpha_3-\alpha_4-1}
\sum_{j} \omega_jL_j\left(\frac{1+\alpha_1+\alpha_2-\alpha_3+\alpha_4}{2}\right)\\\times
L_j\left(\frac{1-\alpha_1-\alpha_2-\alpha_3+\alpha_4}{2}\right)
L_j\left(\frac{1+\alpha_1-\alpha_2+\alpha_3+\alpha_4}{2}\right)
L_j\left(\frac{1-\alpha_1+\alpha_2+\alpha_3+\alpha_4}{2}\right)\\\times
\left(\widehat{W}_{1,1}^{+}+\varepsilon_j\widehat{W}_{1,1}^{-}\right)(\alpha_1+\alpha_2,\alpha_3-\alpha_4;t_j).
\end{multline}
Its holomorphic analogue is given by:
\begin{multline}\label{Eh11}
\term_{h}^{1}(\AlphaVec):=2(-1)^k(2\pi)^{\alpha_3-\alpha_4-1}
\sum_{m=6}^{\infty}\frac{2m-1}{\pi^2}\sum_{f\in H_{2m}} \omega_f
L_f\left(\frac{1+\alpha_1+\alpha_2-\alpha_3+\alpha_4}{2}\right)\\\times
L_f\left(\frac{1-\alpha_1-\alpha_2-\alpha_3+\alpha_4}{2}\right)
L_f\left(\frac{1+\alpha_1-\alpha_2+\alpha_3+\alpha_4}{2}\right)
\\\times
L_f\left(\frac{1-\alpha_1+\alpha_2+\alpha_3+\alpha_4}{2}\right)
\widehat{W}_{1,1}(\alpha_1+\alpha_2,\alpha_3-\alpha_4;m).
\end{multline}
Next, we define the continuous contribution as the weighted eighth moment of the Riemann zeta function:
\begin{multline}\label{Ec011def}
\term_{c}^{1}(\AlphaVec):=
\frac{4(-1)^k(2\pi)^{\alpha_3-\alpha_4-1}}{2\pi i}\int_{(0)}
\ZE_{1,1}(\AlphaVec;\xi)\\\times
\left(\widehat{W}_{1,1}^{+}(\alpha_1+\alpha_2,\alpha_3-\alpha_4;-i\xi)+\widehat{W}_{1,1}^{-}(\alpha_1+\alpha_2,\alpha_3-\alpha_4;-i\xi)\right)d\xi,
\end{multline}
where 
\begin{multline}\label{ZE11def}
\ZE_{1,1}(\AlphaVec;\xi)=  
\zeta\left(\frac{1+\alpha_1-\alpha_2+\alpha_3+\alpha_4}{2}+\xi\right)
\zeta\left(\frac{1-\alpha_1+\alpha_2+\alpha_3+\alpha_4}{2}+\xi\right)\\\times
\zeta\left(\frac{1+\alpha_1-\alpha_2+\alpha_3+\alpha_4}{2}-\xi\right)
\zeta\left(\frac{1-\alpha_1+\alpha_2+\alpha_3+\alpha_4}{2}-\xi\right)
\frac{\ZC(\alpha_1+\alpha_2,\alpha_3-\alpha_4,\xi)}{\zeta(1+2\xi)\zeta(1-2\xi)}
\end{multline}
and $\ZC(\alpha,\beta,\xi)$ is the product of four Riemann zeta functions defined by the formula \eqref{Z4 def}.
Recall that the function $W_{1,1}(y)$ is defined by \eqref{W11def} and its transforms used in the definitions of \eqref{Ed11}, \eqref{Eh11} and \eqref{Ec011def} are given in Theorem \ref{Th:BADP}.

Finally, we define the following main terms:

\begin{multline}\label{UW11 final}
\mainU_{1,1}(\AlphaVec):=
\Co(\alpha_1,\alpha_2,\alpha_3,-\alpha_4)\ZF(\alpha_1,\alpha_2,\alpha_3,-\alpha_4)
\frac{\cos\pi\frac{\alpha_1+\alpha_2-2\alpha_4}{2}}{2\cos\pi\frac{\alpha_4-\alpha_2}{2}\cos\pi\frac{\alpha_4-\alpha_1}{2}}\\+
\Co(-\alpha_1,-\alpha_2,\alpha_3,-\alpha_4)\ZF(-\alpha_1,-\alpha_2,\alpha_3,-\alpha_4)
\frac{\cos\pi\frac{\alpha_1+\alpha_2+2\alpha_4}{2}}{2\cos\pi\frac{\alpha_4+\alpha_2}{2}\cos\pi\frac{\alpha_4+\alpha_1}{2}}\\+
\Co(\alpha_1,\alpha_2,-\alpha_3,\alpha_4)\ZF(\alpha_1,\alpha_2,-\alpha_3,\alpha_4)
\frac{\cos\pi\frac{\alpha_1+\alpha_2-2\alpha_3}{2}}{2\cos\pi\frac{\alpha_3-\alpha_2}{2}\cos\pi\frac{\alpha_3-\alpha_1}{2}}\\+
\Co(-\alpha_1,-\alpha_2,-\alpha_3,\alpha_4)\ZF(-\alpha_1,-\alpha_2,-\alpha_3,\alpha_4)
\frac{\cos\pi\frac{\alpha_1+\alpha_2+2\alpha_3}{2}}{2\cos\pi\frac{\alpha_3+\alpha_2}{2}\cos\pi\frac{\alpha_3+\alpha_1}{2}}
\end{multline}
and
\begin{multline}\label{S11}
\Sok_{1,1}(\AlphaVec):=
\Co(-\alpha_1,\alpha_2,-\alpha_3,-\alpha_4)\ZF(-\alpha_1,\alpha_2,-\alpha_3,-\alpha_4)
\frac{\cos\pi\frac{\alpha_1-\alpha_2+2\alpha_3+2\alpha_4}{2}}{2\cos\pi\frac{\alpha_4+\alpha_3}{2}\cos\pi\frac{\alpha_1-\alpha_2+\alpha_3+\alpha_4}{2}}\\+
\Co(\alpha_1,-\alpha_2,-\alpha_3,-\alpha_4)\ZF(\alpha_1,-\alpha_2,-\alpha_3,-\alpha_4)
\frac{\cos\pi\frac{\alpha_2-\alpha_1+2\alpha_3+2\alpha_4}{2}}{2\cos\pi\frac{\alpha_4+\alpha_3}{2}\cos\pi\frac{\alpha_2-\alpha_1+\alpha_3+\alpha_4}{2}},
\end{multline}
where $\ZF(\epsilon_1\alpha_1,\epsilon_2\alpha_2,\epsilon_3\alpha_3,\epsilon_4\alpha_4)$ and $\Co(\epsilon_1\alpha_1,\epsilon_2\alpha_2,\epsilon_3\alpha_3,\epsilon_4\alpha_4)$ are given by \eqref{7zeta def} and \eqref{Co def}, respectively.

Now we are ready to state the main result of this section.
\begin{lem}\label{lem:td11}
For $|\alpha_j|\ll \epsilon$ one has
\begin{equation}
\TD^{(1,1)}(\AlphaVec)=\mainU_{1,1}(\AlphaVec)+\Sok_{1,1}(\AlphaVec)+\term_{d}^{1}(\AlphaVec)+\term_{h}^{1}(\AlphaVec)+\term_{c}^{1}(\AlphaVec).
\end{equation}
\end{lem}

In order to prove Lemma \ref{lem:td11} we apply Theorem \ref{Th:BADP} for studying \eqref{TD11}. As a result, we conclude that for
\begin{equation}\label{TD11 alpha conditions0}
\min_{j=3,4}\Re{\alpha_j}>1+\max_{j=1,2}|\Re{\alpha_j}|,\quad \Re{\alpha_3}-\Re{\alpha_4}>-1+|\Re{\alpha_1}+\Re{\alpha_2}|
\end{equation}
one has
\begin{equation}\label{TD11 to UEhEdEd}
\TD^{(1,1)}(\AlphaVec)=\U(\AlphaVec;W_{1,1})+\E_h(\AlphaVec;W_{1,1})+\E_d(\AlphaVec;W_{1,1})+\E_c(\AlphaVec;W_{1,1}),
\end{equation}
where
\begin{equation}\label{UW11 def}
\U(\AlphaVec;W_{1,1})=(-1)^k
\zeta(1+\alpha_3+\alpha_4)\sum_{n=1}^{\infty}
\frac{\sigma_{\alpha_1-\alpha_2}(n)}{n^{1+\alpha_1+\alpha_3}}U(n,\alpha_1+\alpha_2,\alpha_3-\alpha_4;W_{1,1}),
\end{equation}
and for $\ast\in\{h,d,c\}$
\begin{equation}\label{EhdcW11 def}
\E_{\ast}(\AlphaVec;W_{1,1})=(-1)^k
\zeta(1+\alpha_3+\alpha_4)\sum_{n=1}^{\infty}\frac{\sigma_{\alpha_1-\alpha_2}(n)}{n^{1+\alpha_1+\alpha_3}}
E_{\ast}(n,\alpha_1+\alpha_2,\alpha_3-\alpha_4;W_{1,1}).
\end{equation}
First, we show that $\E_{\ast}(\AlphaVec;W_{1,1})$ converges in a wider range than \eqref{TD11 alpha conditions0}.
Substituting \eqref{Eh def} to \eqref{EhdcW11 def} and using \eqref{Lprod}, we prove that
\begin{equation}
\E_{h}(\AlphaVec;W_{1,1})=\term_{h}^{1}(\AlphaVec),
\end{equation}
where $\term_{h}^{1}(\AlphaVec)$ is defined by \eqref{Eh11}.
The term $\term_{h}^{1}(\AlphaVec)$ converges in the range $|\alpha_j|\ll A$. This follows from the integral representation   \eqref{hatWmtoPhi}, the estimate
\eqref{Phik LGest} and the fact that  $W_{1,1}(y)=0$ unless $y>x_0$.

 In the next lemma,  we investigate the term $\E_d(\AlphaVec;W_{1,1})$. The term $\E_c(\AlphaVec;W_{1,1})$ converges in the same range of $\alpha_j$ since the weight functions in \eqref{Ed def} and \eqref{Ec def} coincide.

\begin{lem}\label{Ed11-lem}
For
\begin{equation}\label{Eh11 alpha conditions0}
\Re{\alpha_3}+\Re{\alpha_4}>-1-|\Re{\alpha_1}-\Re{\alpha_2}|
\end{equation}
one has
\begin{equation}\label{eq:ed1}
\E_d(\AlphaVec;W_{1,1})=\term_{d}^{1}(\AlphaVec), 
\end{equation}
where $\term_{d}^{1}(\AlphaVec)$  is defined by \eqref{Ed11}.
\end{lem}
\begin{proof}
Substituting \eqref{Ed def} to \eqref{EhdcW11 def} and using \eqref{Lprod}, we show that the identity \eqref{eq:ed1} holds. It remains to prove that \eqref{Ed11} converges absolutely. To do this, we are going to show that the functions $\widehat{W}^{\pm}(t_j)$ decay faster than any fixed power of $t_j$.
Due to \eqref{hatW+}, \eqref{hatW-}, \eqref{hatW0}, it  is enough to prove that
\begin{multline}
\widehat{W}_{1,1}^{0}(\alpha,\beta;it_j)=
\int_{0}^{\infty}W_{1,1}(y)y^{\frac{\alpha+\beta-1}{2}-it_j}
\HyGI\left(\frac{1+\alpha-\beta}{2}+it_j,\frac{1-\alpha-\beta}{2}+it_j,1+2it_j;\frac{-1}{y}\right)dy\\\ll\frac{1}{t_j^A},
\end{multline}
where $\alpha=\alpha_1+\alpha_2$ and $\beta=\alpha_3-\alpha_4.$ With this goal, we apply \eqref{W02F1 asympt}. It is sufficient to consider only the contribution of the main term, since all other terms are of the same form and are smaller in absolute value, while the error term yields the bound $t_j^{-A}$ immediately.
Therefore, using \eqref{Stirling0} and \eqref{W11def}, we are left to prove that
\begin{equation}\label{hatW011est2}
\int_{0}^{\infty}
\DU(1/y)\phi_k\left(\frac{1}{1+y};\alpha_1,\alpha_2\right)
\frac{y^{-1/4-\alpha_1-\alpha_2+\alpha_3-\alpha_4}(\sqrt{1+y}+\sqrt{y})^{-it_j}}{(1+y)^{5/4-\alpha_1-\alpha_2+3\alpha_3/2-\alpha_4/2}}dy\ll\frac{1}{t_j^A}.
\end{equation}
But this estimate follows directly from multiple integration by parts.
\end{proof}
Now we start to evaluate the main term.
\begin{lem}\label{U11-lem}
For
\begin{equation}
\min(\Re{\alpha_3},\Re{\alpha_4})>\max(\Re{\alpha_1},\Re{\alpha_2}),\quad
\min(\Re{\alpha_3},\Re{\alpha_4})>-\min(\Re{\alpha_1},\Re{\alpha_2})
\end{equation}
one has
\begin{multline}\label{UW11}
\U(\AlphaVec;W_{1,1})=\\
(-1)^k\RF(\alpha_4-\alpha_1)\RF(\alpha_4-\alpha_2)\ZF(\alpha_1,\alpha_2,\alpha_3,-\alpha_4)
\int_0^{\infty}W_{1,1}(x)x^{\alpha_1+\alpha_2}(1+x)^{\alpha_3-\alpha_4}dx\\+
(-1)^k\RF(\alpha_4+\alpha_1)\RF(\alpha_4+\alpha_2)\ZF(-\alpha_1,-\alpha_2,\alpha_3,-\alpha_4)\int_0^{\infty}W_{1,1}(x)(1+x)^{\alpha_3-\alpha_4}dx\\+
(-1)^k\RF(\alpha_3-\alpha_1)\RF(\alpha_3-\alpha_2)\ZF(\alpha_1,\alpha_2,-\alpha_3,\alpha_4)\int_0^{\infty}W_{1,1}(x)x^{\alpha_1+\alpha_2}dx\\+
(-1)^k\RF(\alpha_3+\alpha_1)\RF(\alpha_3+\alpha_2)\ZF(-\alpha_1,-\alpha_2,-\alpha_3,\alpha_4)\int_0^{\infty}W_{1,1}(x)dx.
\end{multline}
\end{lem}
\begin{proof}
Substituting \eqref{Udef} and \eqref{mul def} to \eqref{UW11 def}, we obtain four summands. Let us consider the first one (other summands can be treated in the same way):
\begin{multline}\label{UW11 eq1}
(-1)^k\zeta(1+\alpha_3+\alpha_4)\frac{\zeta(1+\alpha_1+\alpha_2)\zeta(1+\alpha_3-\alpha_4)}{\zeta(2+\alpha_1+\alpha_2+\alpha_3-\alpha_4)}
\sum_{n=1}^{\infty}\frac{\sigma_{\alpha_1-\alpha_2}(n)}{n^{1+\alpha_1+\alpha_3}}
\sigma_{1+\alpha_1+\alpha_2+\alpha_3-\alpha_4}(n)\\\times
\int_0^{\infty}W_{1,1}(x)x^{\alpha_1+\alpha_2}(1+x)^{\alpha_3-\alpha_4}dx.
\end{multline}
Using \eqref{RamId} one has
\begin{equation}\label{Ram111}
\sum_{n=1}^{\infty}\frac{\sigma_{\alpha_1-\alpha_2}(n)}{n^{1+\alpha_1+\alpha_3}}
\sigma_{1+\alpha_1+\alpha_2+\alpha_3-\alpha_4}(n)=\frac{\zeta(1+\alpha_1+\alpha_3)}{\zeta(1+\alpha_3+\alpha_4)}\zeta(1+\alpha_3+\alpha_2)
\zeta(\alpha_4-\alpha_1)\zeta(\alpha_4-\alpha_2).
\end{equation}
Substituting \eqref{Ram111} to  \eqref{UW11 eq1}, applying  \eqref{zeta functional} to $\zeta(\alpha_4-\alpha_1)\zeta(\alpha_4-\alpha_2)$ and using the notation \eqref{7zeta def}, we obtain the first summand in \eqref{UW11}.
\end{proof}
We are left to evaluate the integrals in \eqref{UW11}.  Unfortunately, it is impossible to express these integrals only in terms of some special functions due to the presence of the function $\DU(1/x)$ in $W_{1,1}(x)$. Nevertheless, we can approximate these integrals by the ratio of Gamma factors  with a good precision.
\begin{lem}\label{W11-1lem}
The following asymptotic formula holds:
\begin{multline}\label{W11 int1}
\int_0^{\infty}W_{1,1}(x)x^{\alpha_1+\alpha_2}(1+x)^{\alpha_3-\alpha_4}dx\\=
2(2\pi)^{\alpha_1+\alpha_2-2\alpha_4}X_k(\alpha_4)
\Gamma(\alpha_4-\alpha_2)\Gamma(\alpha_4-\alpha_1)\cos\pi\frac{\alpha_1+\alpha_2-2\alpha_4}{2}+O(k^{\epsilon-1}).
\end{multline}
\end{lem}
\begin{proof}
Using \eqref{DU def} and \eqref{W11def}, one has
\begin{multline}\label{W11 int1-1}
\int_0^{\infty}W_{1,1}(x)x^{\alpha_1+\alpha_2}(1+x)^{\alpha_3-\alpha_4}dx=
\int_0^{\infty}\phi_k\left(\frac{1}{1+x};\alpha_1,\alpha_2\right)x^{\alpha_1/2+\alpha_2/2}(1+x)^{-1-\alpha_4}dx\\-
\int_0^{\infty}\DU(x)\phi_k\left(\frac{1}{1+x};\alpha_1,\alpha_2\right)x^{\alpha_1/2+\alpha_2/2}(1+x)^{-1-\alpha_4}dx.
\end{multline}
To evaluate the first integral, we apply \eqref{phik Mellin}, \eqref{J def}:
\begin{multline}\label{W11 int1-2}
\int_0^{\infty}\phi_k\left(\frac{1}{1+x};\alpha_1,\alpha_2\right)x^{\alpha_1/2+\alpha_2/2}(1+x)^{-1-\alpha_4}dx=
(2\pi)^{\alpha_1+\alpha_2}
\frac{1}{2\pi i}\int_{(\sigma)}\Gamma(k,\alpha_1,\alpha_2;s)\\\times\sin\left(\pi\frac{s+\alpha_1+\alpha_2}{2}\right)
\int_0^{\infty}(1+x)^{-1/2-\alpha_4-s/2}dxds,
\end{multline}
where $\max(1-2k,1-2\Re{\alpha_4})<\sigma<1-2\max(|\Re{\alpha_1}|,|\Re{\alpha_2}|)$.
Evaluating the integral over $x$ and substituting \eqref{Gamma(k alpha s) def}, we infer
\begin{multline}\label{W11 int1-3}
\int_0^{\infty}\phi_k\left(\frac{1}{1+x};\alpha_1,\alpha_2\right)x^{\alpha_1/2+\alpha_2/2}(1+x)^{-1-\alpha_4}dx=\\=
(2\pi)^{\alpha_1+\alpha_2}
\frac{1}{2\pi i}\int_{(\sigma)}
\frac{\Gamma(k-1/2+s/2)}{\Gamma(k+1/2-s/2)}\Gamma(1/2-\alpha_1-s/2)\Gamma(1/2-\alpha_2-s/2)
\frac{\sin\left(\pi\frac{s+\alpha_1+\alpha_2}{2}\right)}{s/2+\alpha_4-1/2}ds.
\end{multline}
Let $s=\sigma+iy$. Using \eqref{Stirling0} one has
\begin{multline}\label{W11 int1-4}
\frac{\Gamma(k-1/2+s/2)}{\Gamma(k+1/2-s/2)}\Gamma(1/2-\alpha_1-s/2)\Gamma(1/2-\alpha_2-s/2)
\frac{\sin\left(\pi\frac{s+\alpha_1+\alpha_2}{2}\right)}{s/2+\alpha_4-1/2}\\ \ll
\frac{(k+|y|)^{-1+\sigma}}{(1+|y|)^{1+\Re{\alpha_1}+\Re{\alpha_2}+\sigma}} .
\end{multline}
Moving the line of integration to the left on $\Re{s}=2\sigma_1<1-2\Re{\alpha_4}$, crossing the pole at $s=1-2\alpha_4$ and using the bound \eqref{W11 int1-4} to estimate the remaining integral, we conclude that
\begin{multline}\label{W11 int1-5}
\int_0^{\infty}\phi_k\left(\frac{1}{1+x};\alpha_1,\alpha_2\right)x^{\alpha_1/2+\alpha_2/2}(1+x)^{-1-\alpha_4}dx=
2(2\pi)^{\alpha_1+\alpha_2}\frac{\Gamma(k-\alpha_4)}{\Gamma(k+\alpha_4)}\\\times
\Gamma(\alpha_4-\alpha_2)\Gamma(\alpha_4-\alpha_1)\cos\pi\frac{\alpha_1+\alpha_2-2\alpha_4}{2}+O(k^{-1-\Re{\alpha_1}-\Re{\alpha_2}}).
\end{multline}
To estimate the second integral in \eqref{W11 int1-1} we again use \eqref{phik Mellin}, \eqref{J def} showing that
\begin{multline}\label{W11 int1-6}
\int_0^{\infty}\DU(x)\phi_k\left(\frac{1}{1+x};\alpha_1,\alpha_2\right)x^{\alpha_1/2+\alpha_2/2}(1+x)^{-1-\alpha_4}dx\\=
(2\pi)^{\alpha_1+\alpha_2}
\frac{1}{2\pi i}\int_{(\sigma)}\Gamma(k,\alpha_1,\alpha_2;s)\sin\left(\pi\frac{s+\alpha_1+\alpha_2}{2}\right)
\int_0^{\infty}\DU(x)(1+x)^{-1/2-\alpha_4-s/2}dxds.
\end{multline}
Note that the integral
\begin{equation}\label{W11 int1-7}
\int_0^{\infty}\DU(x)(1+x)^{-1/2-\alpha_4-s/2}dx\ll\frac{1}{1+|\Im{s}|},
\end{equation}
considered as a function of $s$, is an entire function. Therefore, moving the line of integration in \eqref{W11 int1-6}  to the left (say on $\Re{s}=0$) and applying \eqref{W11 int1-4}, \eqref{W11 int1-7}, we obtain
\begin{equation}\label{W11 int1-8}
\int_0^{\infty}\DU(x)\phi_k\left(\frac{1}{1+x};\alpha_1,\alpha_2\right)x^{\alpha_1/2+\alpha_2/2}(1+x)^{-1-\alpha_4}dx\ll k^{-1+\epsilon}.
\end{equation}
Therefore, \eqref{W11 int1} follows from  \eqref{W11 int1-1}, \eqref{W11 int1-5}, \eqref{W11 int1-8}.
\end{proof}

\begin{lem}\label{W11-2lem}
The following asymptotic formula holds:
\begin{multline}\label{W11 int2}
\int_0^{\infty}W_{1,1}(x)(1+x)^{\alpha_3-\alpha_4}dx=
(2\pi)^{-\alpha_1-\alpha_2-2\alpha_4}
X_{k}(\alpha_1)X_{k}(\alpha_2)X_{k}(\alpha_4)
\\\times
\Gamma(\alpha_4+\alpha_2)\Gamma(\alpha_4+\alpha_1)\cos\pi\frac{\alpha_1+\alpha_2+2\alpha_4}{2}+O(k^{\epsilon-1}).
\end{multline}
\end{lem}
\begin{proof}
Using \eqref{DU def} and \eqref{W11def} one has
\begin{multline}\label{W11 int2-1}
\int_0^{\infty}W_{1,1}(x)(1+x)^{\alpha_3-\alpha_4}dx=
\int_0^{\infty}\phi_k\left(\frac{1}{1+x};\alpha_1,\alpha_2\right)x^{-\alpha_1/2-\alpha_2/2}(1+x)^{-1-\alpha_4}dx\\-
\int_0^{\infty}\DU(x)\phi_k\left(\frac{1}{1+x};\alpha_1,\alpha_2\right)x^{-\alpha_1/2-\alpha_2/2}(1+x)^{-1-\alpha_4}dx.
\end{multline}
We estimate the second integral first.  Applying \eqref{phik Mellin} and \eqref{J def}, the second integral can be rewritten as
\begin{multline}\label{W11 int2-2}
\int_0^{\infty}\DU(x)\phi_k\left(\frac{1}{1+x};\alpha_1,\alpha_2\right)x^{-\alpha_1/2-\alpha_2/2}(1+x)^{-1-\alpha_4}dx\\=
(2\pi)^{\alpha_1+\alpha_2}
\frac{1}{2\pi i}\int_{(\sigma)}\Gamma(k,\alpha_1,\alpha_2;s)\sin\left(\pi\frac{s+\alpha_1+\alpha_2}{2}\right)
\int_0^{\infty}\DU(x)x^{-\alpha_1-\alpha_2}(1+x)^{-1/2-\alpha_4-s/2}dxds.
\end{multline}
To bound the integral over $x$, we split it  into two parts using $\DU(x)=\DU_0(x)+\DU_1(x)$, where $\DU_0(x)$ is a smooth function such that
\begin{equation}\label{DU0 def}
\DU_0(x)=1\,\hbox{ for } 0\le x\le \epsilon_0,\quad
\DU_0(x)=0\,\hbox{ for }  x> 2\epsilon_0.
\end{equation}
One has
\begin{equation}\label{DU0 int}
\int_0^{\infty}\DU_0(x)x^{-\alpha_1-\alpha_2}(1+x)^{-1/2-\alpha_4-s/2}dx\ll\epsilon_0^{1-\Re{\alpha_1}-\Re{\alpha_2}}.
\end{equation}
To estimate the integral with $\DU_1(x)$, we apply the first derivative test (or just integrate by parts once):
\begin{equation}\label{DU1 int}
\int_0^{\infty}\DU_1(x)x^{-\alpha_1-\alpha_2}(1+x)^{-1/2-\alpha_4-s/2}dx\ll\frac{\epsilon_0^{-\Re{\alpha_1}-\Re{\alpha_2}}}{1+|\Im{s}+\Im{\alpha_4}|}.
\end{equation}
Choosing $\epsilon_0=(1+|\Im{s}+\Im{\alpha_4}|)^{-1}$,  we obtain
\begin{equation}\label{DU int}
\int_0^{\infty}\DU(x)x^{-\alpha_1-\alpha_2}(1+x)^{-1/2-\alpha_4-s/2}dx\ll(1+|\Im{s}+\Im{\alpha_4}|)^{-1+\Re{\alpha_1}+\Re{\alpha_2}}.
\end{equation}
Substituting \eqref{DU int} into  \eqref{W11 int2-2} and estimating the resulting integral using  \eqref{W11 int1-4}, we infer
\begin{equation}\label{W11 int2-3}
\int_0^{\infty}\DU(x)\phi_k\left(\frac{1}{1+x};\alpha_1,\alpha_2\right)x^{-\alpha_1/2-\alpha_2/2}(1+x)^{-1-\alpha_4}dx\ll\frac{k^{\epsilon}}{k}.
\end{equation}
To evaluate the first integral on the right-hand side of \eqref{W11 int2-1}, we apply \eqref{phik Mellin}, \eqref{J def}:
\begin{multline}\label{W11 int2-4}
\int_0^{\infty}\phi_k\left(\frac{1}{1+x};\alpha_1,\alpha_2\right)x^{-\alpha_1/2-\alpha_2/2}(1+x)^{-1-\alpha_4}dx=
(2\pi)^{\alpha_1+\alpha_2}
\frac{1}{2\pi i}\int_{(\sigma)}\Gamma(k,\alpha_1,\alpha_2;s)\\\times\sin\left(\pi\frac{s+\alpha_1+\alpha_2}{2}\right)
\int_0^{\infty}x^{-\alpha_1-\alpha_2}(1+x)^{-1/2-\alpha_4-s/2}dxds,
\end{multline}
where $\max(1-2\Re{\alpha_4}-2\Re{\alpha_1}-2\Re{\alpha_2},-1-2\Re{\alpha_4})<\sigma<1-2\max(|\Re{\alpha_1}|,|\Re{\alpha_2}|)$. To evaluate the integral over $x$, we use Lemma \ref{Lemma beta integrals}, which yields
\begin{multline}\label{W11 int2-5}
\int_0^{\infty}\phi_k\left(\frac{1}{1+x};\alpha_1,\alpha_2\right)x^{-\alpha_1/2-\alpha_2/2}(1+x)^{-1-\alpha_4}dx=
(2\pi)^{\alpha_1+\alpha_2}\Gamma(1-\alpha_1-\alpha_2)\\\times
\frac{1}{2\pi i}\int_{(\sigma)}\Gamma(k,\alpha_1,\alpha_2;s)\sin\left(\pi\frac{s+\alpha_1+\alpha_2}{2}\right)
\frac{\Gamma(-1/2+\alpha_1+\alpha_2+\alpha_4+s/2)}{\Gamma(1/2+\alpha_4+s/2)}ds.
\end{multline}
Moving the line of integration, crossing the pole at $s=1-2\alpha_1-2\alpha_2-2\alpha_4$ and estimating the remaining integral using  \eqref{W11 int1-4}, we obtain
\begin{multline}\label{W11 int2-6}
\int_0^{\infty}\phi_k\left(\frac{1}{1+x};\alpha_1,\alpha_2\right)x^{-\alpha_1/2-\alpha_2/2}(1+x)^{-1-\alpha_4}dx=
2(2\pi)^{\alpha_1+\alpha_2}\frac{\Gamma(k-\alpha_1-\alpha_2-\alpha_4)}{\Gamma(k+\alpha_1+\alpha_2+\alpha_4)}
\\\times
\Gamma(\alpha_4+\alpha_2)\Gamma(\alpha_4+\alpha_1)\cos\pi\frac{\alpha_1+\alpha_2+2\alpha_4}{2}+O(k^{\epsilon-1}).
\end{multline}
Applying \eqref{Xdef} and \eqref{Xapprox}, one has
\begin{equation}\label{XtoXXX}
\frac{\Gamma(k-\alpha_1-\alpha_2-\alpha_4)}{\Gamma(k+\alpha_1+\alpha_2+\alpha_4)}=(2\pi)^{-2\alpha_1-2\alpha_2-2\alpha_4}
X_{k}(\alpha_1)X_{k}(\alpha_2)X_{k}(\alpha_4)\left(1+O(k^{-1})\right).
\end{equation}
Now \eqref{W11 int2} follows from \eqref{W11 int2-1}, \eqref{W11 int2-3}, \eqref{W11 int2-6} and \eqref{XtoXXX}.
\end{proof}

\begin{lem}\label{W11-3lem}
The following asymptotic formula holds:
\begin{multline}\label{W11 int3}
\int_0^{\infty}W_{1,1}(x)x^{\alpha_1+\alpha_2}dx=\\
2(2\pi)^{\alpha_1+\alpha_2-2\alpha_3}X_k(\alpha_3)
\Gamma(\alpha_3-\alpha_2)\Gamma(\alpha_3-\alpha_1)\cos\pi\frac{\alpha_1+\alpha_2-2\alpha_3}{2}+O(k^{\epsilon-1}).
\end{multline}
\end{lem}
\begin{proof}
According to  \eqref{DU def} and \eqref{W11def} the integral \eqref{W11 int3} coincides with the right-hand side of \eqref{W11 int1-1} after changing $\alpha_4$ in it by $\alpha_3$. Therefore, the result follows immediately by making this change on the right-hand side of \eqref{W11 int1}.
\end{proof}
Similarly, after changing $\alpha_4$ on the right-hand side of
\eqref{W11 int2} by $\alpha_3$ we prove the following result.

\begin{lem}\label{W11-4lem}
One has
\begin{multline}\label{W11 int4}
\int_0^{\infty}W_{1,1}(x)dx=\\
(2\pi)^{-\alpha_1-\alpha_2-2\alpha_3}
X_{k}(\alpha_1)X_{k}(\alpha_2)X_{k}(\alpha_3)
\Gamma(\alpha_3+\alpha_2)\Gamma(\alpha_3+\alpha_1)\cos\pi\frac{\alpha_1+\alpha_2+2\alpha_3}{2}+O(k^{\epsilon-1}).
\end{multline}
\end{lem}
Combining results of Lemmas \ref{U11-lem}-\ref{W11-4lem}, we finally prove the next lemma.

\begin{lem}\label{U11-lem final}
One has
\begin{equation}
\U(\AlphaVec;W_{1,1})=\mainU_{1,1}(\AlphaVec),
\end{equation}
where $\mainU_{1,1}(\AlphaVec)$ is defined by \eqref{UW11 final}.
\end{lem}
\begin{proof}
Substituting \eqref{W11 int1}, \eqref{W11 int2}, \eqref{W11 int3}, \eqref{W11 int4} into \eqref{UW11} and using \eqref{zeta functional}, we obtain
\eqref{UW11 final}.
\end{proof}

The region \eqref{TD11 alpha conditions0} does not contain the most interesting points  $\Re{\alpha_j}=0$. The only problem in getting an analytic continuation for \eqref{TD11 to UEhEdEd} is the term $\E_c(\AlphaVec;W_{1,1})$, see \eqref{Ec eq1}. In order to treat this term one needs to compute  the sum:
\begin{equation}
\widehat{W}_{1,1}^{+}(\alpha_1+\alpha_2,\alpha_3-\alpha_4;-i\xi)+\widehat{W}_{1,1}^{-}(\alpha_1+\alpha_2,\alpha_3-\alpha_4;-i\xi)
\end{equation}
at several special points $\xi.$
\begin{lem}\label{Lem:hatW11-xi1}
For $|\alpha_j|\ll\epsilon$ and $\xi_1=\frac{1-\alpha_1+\alpha_2-\alpha_3-\alpha_4}{2}$ one has
\begin{multline}\label{hatW11 xi1}
\widehat{W}_{1,1}^{+}(\alpha_1+\alpha_2,\alpha_3-\alpha_4;-i\xi_1)+\widehat{W}_{1,1}^{-}(\alpha_1+\alpha_2,\alpha_3-\alpha_4;-i\xi_1)=\\=
\frac{\pi}{2}(2\pi)^{\alpha_1+\alpha_2}\Gamma(\alpha_4+\alpha_1)\Gamma(\alpha_4-\alpha_2)\Gamma(\alpha_4+\alpha_3)
\frac{\Gamma(k-\alpha_1-\alpha_3-\alpha_4)}{\Gamma(k+\alpha_1+\alpha_3+\alpha_4)}\\\times
\frac{\cos\pi\frac{\alpha_1-\alpha_2+2\alpha_4}{2}+\cos\pi\frac{\alpha_1+\alpha_2}{2}}{\cos\pi\frac{\alpha_1-\alpha_2+\alpha_3+\alpha_4}{2}}
\cos\pi\frac{\alpha_1-\alpha_2+2\alpha_3+2\alpha_4}{2}+O\left(\frac{k^{\epsilon}}{k}\right).
\end{multline}
\end{lem}
\begin{proof}
It follows from \eqref{hatW+ + hatW-}, \eqref{hatW0} and \eqref{W11def} that
\begin{multline}\label{W11hat+-xi1 eq1}
\widehat{W}_{1,1}^{+}(\alpha_1+\alpha_2,\alpha_3-\alpha_4;-i\xi_1)+\widehat{W}_{1,1}^{-}(\alpha_1+\alpha_2,\alpha_3-\alpha_4;-i\xi_1)=\\=
\frac{\pi}{4}\frac{\cos\pi\frac{\alpha_1-\alpha_2+2\alpha_3}{2}-\cos\pi\frac{\alpha_1+\alpha_2}{2}}{\cos\pi\frac{\alpha_1-\alpha_2+\alpha_3+\alpha_4}{2}}
\int_{0}^{\infty}
\DU(1/y)\phi_k\left(\frac{1}{1+y};\alpha_1,\alpha_2\right)
\frac{y^{-1+\alpha_1/2-\alpha_2/2+\alpha_3}}{(1+y)^{1+\alpha_3}}\\\times
\HyGI\left(1+\alpha_2-\alpha_3,1-\alpha_1-\alpha_3,2-\alpha_1+\alpha_2-\alpha_3-\alpha_4;\frac{-1}{y}\right)dy\\+
\frac{\pi}{4}\frac{\cos\pi\frac{\alpha_1-\alpha_2+2\alpha_4}{2}+\cos\pi\frac{\alpha_1+\alpha_2}{2}}{\cos\pi\frac{\alpha_1-\alpha_2+\alpha_3+\alpha_4}{2}}
\int_{0}^{\infty}
\DU(1/y)\phi_k\left(\frac{1}{1+y};\alpha_1,\alpha_2\right)
\frac{y^{-\alpha_1/2+\alpha_2/2-\alpha_4}}{(1+y)^{1+\alpha_3}}\\\times
\HyGI\left(\alpha_1+\alpha_4,\alpha_4-\alpha_2,\alpha_1-\alpha_2+\alpha_3+\alpha_4;\frac{-1}{y}\right)dy.
\end{multline}
Let us first consider the second integral in \eqref{W11hat+-xi1 eq1}.
Using \eqref{2I1 to 3I2} it can be further broken down into the sum of two integrals:
\begin{multline}\label{W11hat+-xi1 int2eq1}
\int_{0}^{\infty}
\DU(1/y)\phi_k\left(\frac{1}{1+y};\alpha_1,\alpha_2\right)
\frac{y^{-\alpha_1/2+\alpha_2/2-\alpha_4}}{(1+y)^{1+\alpha_3}}
\HyGI\left(\alpha_1+\alpha_4,\alpha_4-\alpha_2,\alpha_1-\alpha_2+\alpha_3+\alpha_4;\frac{-1}{y}\right)dy\\=
\frac{\Gamma(\alpha_1+\alpha_4)\Gamma(\alpha_4-\alpha_2)}{\Gamma(\alpha_1-\alpha_2+\alpha_3+\alpha_4)}
\int_{0}^{\infty}
\DU(1/y)\phi_k\left(\frac{1}{1+y};\alpha_1,\alpha_2\right)
\frac{y^{-\alpha_1/2+\alpha_2/2-\alpha_4}}{(1+y)^{1+\alpha_3}}dy\\-
\int_{0}^{\infty}
\DU(1/y)\phi_k\left(\frac{1}{1+y};\alpha_1,\alpha_2\right)
\frac{y^{-1-\alpha_1/2+\alpha_2/2-\alpha_4}}{(1+y)^{1+\alpha_3}}
\GenHyGI{3}{2}{1+\alpha_1+\alpha_4,1+\alpha_4-\alpha_2,1}{2,1+\alpha_1-\alpha_2+\alpha_3+\alpha_4}{\frac{-1}{y}}dy.
\end{multline}

Let us compute the first integral in \eqref{W11hat+-xi1 int2eq1}. Arguing in the same way as in Lemma \ref{W11-2lem}, we show that
\begin{multline}\label{W11hat+-xi1 int2eq2}
\frac{\Gamma(\alpha_1+\alpha_4)\Gamma(\alpha_4-\alpha_2)}{\Gamma(\alpha_1-\alpha_2+\alpha_3+\alpha_4)}
\int_{0}^{\infty}
\DU(1/y)\phi_k\left(\frac{1}{1+y};\alpha_1,\alpha_2\right)
\frac{y^{-\alpha_1/2+\alpha_2/2-\alpha_4}}{(1+y)^{1+\alpha_3}}dy=\\=
2(2\pi)^{\alpha_1+\alpha_2}\cos\pi\frac{\alpha_1-\alpha_2+2\alpha_3+2\alpha_4}{2}
\Gamma(\alpha_4+\alpha_1)\Gamma(\alpha_4-\alpha_2)\Gamma(\alpha_4+\alpha_3)\\\times
\frac{\Gamma(k-\alpha_1-\alpha_3-\alpha_4)}{\Gamma(k+\alpha_1+\alpha_3+\alpha_4)}+
O\left(\frac{1}{k^{1+\Re{\alpha_2}-\Re{\alpha_4}}}\right).
\end{multline}

Consider the  second  integral in \eqref{W11hat+-xi1 int2eq1}.  Making
the change of variable $x=(1+y)^{-1}$ and using
the definition \eqref{Y tildephik},  we obtain the integral
\begin{equation}\label{W11hat+-xi1 int2eq3}
\int_{0}^{1}
Y(x;\alpha_1,\alpha_2)\frac{g(x)}{x^2(1-x)^2}dx,
\end{equation}
where
\begin{equation}\label{g def}
g(x)=x^{3/2+\alpha_1/2-\alpha_2/2+\alpha_3+\alpha_4}(1-x)^{1/2-\alpha_1/2+\alpha_2/2-\alpha_4}g_1(1-x),
\end{equation}
\begin{equation}\label{g1 def}
g_1(1-x)=\DU\left(\frac{1}{1-x}-1\right)
\GenHyGI{3}{2}{1+\alpha_1+\alpha_4,1+\alpha_4-\alpha_2,1}{2,1+\alpha_1-\alpha_2+\alpha_3+\alpha_4}{\frac{1}{x-1}+1}.
\end{equation}
Integrating by parts with the help of \eqref{Ydifeq}, one has to estimate
\begin{equation}\label{W11hat+-xi1 int2eq4}
\int_{0}^{1}
Y(x;\alpha_1,\alpha_2)\frac{d^2}{dx^2}\left(
\frac{g(x)}{4x^2(1-x)^2\De(x;\alpha_1,\alpha_2)}\right)dx.
\end{equation}
It follows from \eqref{Ydifeq r-def} that
\begin{equation}
4x^2(1-x)^2\De(x;\alpha_1,\alpha_2)=(2k-1)^2x(1-x)+p_2(x),
\end{equation}
where $p_2(x)=x^2-(1+4\alpha_1\alpha_2)x+1-(\alpha_2-\alpha_1)^2$. Evaluating the second derivative in \eqref{W11hat+-xi1 int2eq4}, we show that
\begin{equation}\label{W11hat+-xi1 int2eq5}
\frac{d^2}{dx^2}\left(
\frac{g(x)}{4x^2(1-x)^2\De(x;\alpha_1,\alpha_2)}\right)\ll\frac{x^{-1/2+\Re{(\alpha_1/2-\alpha_2/2+\alpha_3+\alpha_4)}}}{(2k-1)^2x(1-x)+p_2(x)}.
\end{equation}
Using \eqref{W11hat+-xi1 int2eq5}, \eqref{Y tildephik} and Lemma \ref{lem:phikest},  we prove that the integral  \eqref{W11hat+-xi1 int2eq4}  is bounded by
\begin{equation}\label{W11hat+-xi1 int2eq6}
\int_{0}^{1}\DU\left(\frac{x}{1-x}\right)\phi_k\left(x;\alpha_1,\alpha_2\right)
\frac{x^{-1/2+\Re{(\alpha_1/2-\alpha_2/2+\alpha_3+\alpha_4)}}}{k^2x+1}dx\ll\frac{1}{k^{2+2\Re{(\alpha_1+\alpha_3+\alpha_4)}}}.
\end{equation}

We are left to estimate the  first integral in \eqref{W11hat+-xi1 eq1}. Comparing it with the second integral in \eqref{W11hat+-xi1 int2eq1}, we see that
they are quite similar, and thus an estimate like \eqref{W11hat+-xi1 int2eq6} can also be obtained for the  first integral in \eqref{W11hat+-xi1 eq1}.
\end{proof}
\begin{lem}\label{Lem:hatW11-xi2}
For $|\alpha_j|\ll\epsilon$ and $\xi_2=\frac{1+\alpha_1-\alpha_2-\alpha_3-\alpha_4}{2}$ one has
\begin{multline}\label{hatW11 xi2}
\widehat{W}_{1,1}^{+}(\alpha_1+\alpha_2,\alpha_3-\alpha_4;-i\xi_2)+\widehat{W}_{1,1}^{-}(\alpha_1+\alpha_2,\alpha_3-\alpha_4;-i\xi_2)=\\=
\frac{\pi}{2}(2\pi)^{\alpha_1+\alpha_2}\Gamma(\alpha_4+\alpha_2)\Gamma(\alpha_4-\alpha_1)\Gamma(\alpha_4+\alpha_3)
\frac{\Gamma(k-\alpha_2-\alpha_3-\alpha_4)}{\Gamma(k+\alpha_2+\alpha_3+\alpha_4)}\\\times
\frac{\cos\pi\frac{\alpha_2-\alpha_1+2\alpha_4}{2}+\cos\pi\frac{\alpha_1+\alpha_2}{2}}{\cos\pi\frac{\alpha_2-\alpha_1+\alpha_3+\alpha_4}{2}}
\cos\pi\frac{\alpha_2-\alpha_1+2\alpha_3+2\alpha_4}{2}+O\left(\frac{k^{\epsilon}}{k}\right).
\end{multline}
\end{lem}
\begin{proof}
Using \eqref{W11def} and \eqref{phik}, we show that
\begin{equation}
W_{1,1}(\alpha_1,\alpha_2,\alpha_3,\alpha_4;x)=W_{1,1}(\alpha_2,\alpha_1,\alpha_3,\alpha_4;x).
\end{equation}
Therefore, \eqref{hatW11 xi2} follows from \eqref{hatW11 xi1} by making the change of variable $\alpha_1\leftrightarrow\alpha_2$.
\end{proof}

Now we are ready to provide an analytic continuation for
$\E_c(\AlphaVec;W_{1,1})$, which is given in the following lemma.

\begin{lem}\label{Ec11-lem1}
For $|\alpha_j|\ll\epsilon$ one has
\begin{equation}\label{Ec11}
\E_{c}(\AlphaVec;W_{1,1})=\Sok_{1,1}(\AlphaVec)+\term_{c}^{1}(\AlphaVec),
\end{equation}
where $\Sok_{1,1}(\AlphaVec)$ and $\term_{c}^{1}(\AlphaVec)$ is defined by \eqref{S11} and \eqref{Ec011def}, respectively.
\end{lem}
\begin{proof}
Substituting \eqref{Ec def} to \eqref{EhdcW11 def}, evaluating the resulting sum over $n$ with the use of \eqref{RamId}, we show that for
\begin{equation}\label{TD11 alpha conditions10}
\Re{\alpha_3}+\Re{\alpha_4}>1+|\Re{\alpha_1}-\Re{\alpha_2}|,\quad
\Re{\alpha_3}-\Re{\alpha_4}>-1+|\Re{\alpha_1}+\Re{\alpha_2}|
\end{equation}
the following identity holds:
\begin{multline}\label{Ec11-1}
\E_{c}(\AlphaVec;W_{1,1})=\frac{4(-1)^k(2\pi)^{\alpha_3-\alpha_4-1}}{2\pi i}\int_{(0)}
\ZE_{1,1}(\AlphaVec;\xi)\\\times
\left(\widehat{W}_{1,1}^{+}(\alpha_1+\alpha_2,\alpha_3-\alpha_4;-i\xi)+\widehat{W}_{1,1}^{-}(\alpha_1+\alpha_2,\alpha_3-\alpha_4;-i\xi)\right)d\xi.
\end{multline}
The function $\ZE_{1,1}(\AlphaVec;\xi)$ has  eight poles. In both cases when the parameters $\alpha_j$ belongs to the region \eqref{TD11 alpha conditions10} or to the region
\begin{equation}\label{TD11 alpha conditions1}
\Re{\alpha_3}+\Re{\alpha_4}<1-|\Re{\alpha_1}-\Re{\alpha_2}|,\quad
\Re{\alpha_3}-\Re{\alpha_4}>-1+|\Re{\alpha_1}+\Re{\alpha_2}|,
\end{equation}
the poles of $\ZC(\alpha_1+\alpha_2,\alpha_3-\alpha_4,\xi)$ are located in the same half-plane $\Re{\xi}\lessgtr0.$ While the remaining four poles:
\begin{equation}\label{Ec11-poles}
\xi_1=\frac{1-\alpha_1+\alpha_2-\alpha_3-\alpha_4}{2},\,
\xi_2=\frac{1+\alpha_1-\alpha_2-\alpha_3-\alpha_4}{2},\,
\xi_3=-\xi_1,\,\xi_4=-\xi_2,
\end{equation}
change their location from one half-plane to another when the parameters $\alpha_j$ pass from the region \eqref{TD11 alpha conditions10} to the region
\eqref{TD11 alpha conditions1}. To obtain an analytic continuation for $\E_{c}(\AlphaVec;W_{1,1})$ to the region \eqref{TD11 alpha conditions1}, we  use a  variant of the Sokhotski-Plemelj theorem. Arguing in the same way as in \cite[Theorem 6.2]{BFIbero}, we show that in the region \eqref{TD11 alpha conditions1} the following formula holds:
\begin{multline}\label{Ec11-2}
\E_{c}(\AlphaVec;W_{1,1})=
4(-1)^k(2\pi)^{\alpha_3-\alpha_4-1}\left(R_{1,1}(\xi_1)-R_{1,1}(\xi_3)+R_{1,1}(\xi_2)-R_{1,1}(\xi_4)\right)+4(-1)^k\\\times
\frac{(2\pi)^{\alpha_3-\alpha_4-1}}{2\pi i}\int_{(0)}
\ZE_{1,1}(\AlphaVec;\xi)
\left(\widehat{W}_{1,1}^{+}(\alpha_1+\alpha_2,\alpha_3-\alpha_4;-i\xi)+\widehat{W}_{1,1}^{-}(\alpha_1+\alpha_2,\alpha_3-\alpha_4;-i\xi)\right)d\xi,
\end{multline}
where
\begin{equation}\label{res R11k def}
R_{1,1}(\xi_k)=
\res_{\xi_k}\ZE_{1,1}(\AlphaVec;\xi)\left(\widehat{W}_{1,1}^{+}(\alpha_1+\alpha_2,\alpha_3-\alpha_4;-i\xi)+
\widehat{W}_{1,1}^{-}(\alpha_1+\alpha_2,\alpha_3-\alpha_4;-i\xi)\right).
\end{equation}
Using  \eqref{zeta functional} and \eqref{7zeta def}, one has
\begin{multline}\label{R11xi1}
R_{1,1}(\xi_1)=
\ZF(-\alpha_1,\alpha_2,-\alpha_3,-\alpha_4)\RF(\alpha_4+\alpha_1)\RF(\alpha_4-\alpha_2)\RF(\alpha_4+\alpha_3)\\\times
\left(\widehat{W}_{1,1}^{+}(\alpha_1+\alpha_2,\alpha_3-\alpha_4;-i\xi_1)+\widehat{W}_{1,1}^{-}(\alpha_1+\alpha_2,\alpha_3-\alpha_4;-i\xi_1)\right),
\end{multline}
\begin{multline}\label{R11xi2}
R_{1,1}(\xi_2)=
\ZF(\alpha_1,-\alpha_2,-\alpha_3,-\alpha_4)\RF(\alpha_4-\alpha_1)\RF(\alpha_4+\alpha_2)\RF(\alpha_4+\alpha_3)\\\times
\left(\widehat{W}_{1,1}^{+}(\alpha_1+\alpha_2,\alpha_3-\alpha_4;-i\xi_2)+\widehat{W}_{1,1}^{-}(\alpha_1+\alpha_2,\alpha_3-\alpha_4;-i\xi_2)\right),
\end{multline}
\begin{equation}\label{R11xi34}
R_{1,1}(\xi_3)=-R_{1,1}(\xi_1),\quad
R_{1,1}(\xi_4)=-R_{1,1}(\xi_2).
\end{equation}
In order to simplify evaluation of $R_{1,1}(\xi_1)$ and $R_{1,1}(\xi_2)$, we restrict the  parameters $\alpha_j$ to the range $|\alpha_j|\ll\epsilon$ instead of the wider range \eqref{TD11 alpha conditions1}.  Using \eqref{hatW11 xi1}, \eqref{zeta functional}, \eqref{XtoXXX}  and the relation
\begin{equation}
\cos\pi\frac{\alpha_1-\alpha_2+2\alpha_4}{2}+\cos\pi\frac{\alpha_1+\alpha_2}{2}=2\cos\pi\frac{\alpha_1+\alpha_4}{2}\cos\pi\frac{\alpha_4-\alpha_2}{2},
\end{equation}
we infer
\begin{multline}\label{R11xi1-xi3}
4(-1)^k(2\pi)^{\alpha_3-\alpha_4-1}\left(R_{1,1}(\xi_1)-R_{1,1}(\xi_3)\right)=
\frac{\cos\pi\frac{\alpha_1-\alpha_2+2\alpha_3+2\alpha_4}{2}}{2\cos\pi\frac{\alpha_4+\alpha_3}{2}\cos\pi\frac{\alpha_1-\alpha_2+\alpha_3+\alpha_4}{2}}\\\times
\Co(-\alpha_1,\alpha_2,-\alpha_3,-\alpha_4)\ZF(-\alpha_1,\alpha_2,-\alpha_3,-\alpha_4)+O\left(\frac{k^{\epsilon}}{k}\right).
\end{multline}
Using \eqref{hatW11 xi2}, \eqref{zeta functional}, \eqref{XtoXXX} (with $\alpha_1$ replaced by $\alpha_2$)  and the relation
\begin{equation}
\cos\pi\frac{\alpha_2-\alpha_1+2\alpha_4}{2}+\cos\pi\frac{\alpha_1+\alpha_2}{2}=2\cos\pi\frac{\alpha_2+\alpha_4}{2}\cos\pi\frac{\alpha_4-\alpha_1}{2},
\end{equation}
we prove
\begin{multline}\label{R11xi2-xi4}
4(-1)^k(2\pi)^{\alpha_3-\alpha_4-1}\left(R_{1,1}(\xi_2)-R_{1,1}(\xi_4)\right)=
\frac{\cos\pi\frac{\alpha_2-\alpha_1+2\alpha_3+2\alpha_4}{2}}{2\cos\pi\frac{\alpha_4+\alpha_3}{2}\cos\pi\frac{\alpha_2-\alpha_1+\alpha_3+\alpha_4}{2}}
\\\times
\Co(\alpha_1,-\alpha_2,-\alpha_3,-\alpha_4)\ZF(\alpha_1,-\alpha_2,-\alpha_3,-\alpha_4)
+O\left(\frac{k^{\epsilon}}{k}\right).
\end{multline}
Finally, \eqref{Ec11} follows from \eqref{Ec11-2}, \eqref{R11xi1-xi3} and \eqref{R11xi2-xi4}.
\end{proof}

\section{Analysis of $\TD^{(1,2)}$}\label{sec:TD12}

It follows from \eqref{phik 1-x to x}, \eqref{W11def} and \eqref{W12def} that
\begin{equation}\label{W12 to W11}
W_{1,2}(\alpha_1,\alpha_2,\alpha_3,\alpha_4;x)=(-1)^kX_k(\alpha_2)W_{1,1}(\alpha_1,-\alpha_2,\alpha_3,\alpha_4;x).
\end{equation}
Then one has (see \eqref{TD11} and \eqref{TD12})
\begin{equation}\label{TD12 to TD11}
\TD^{(1,2)}(\alpha_1,\alpha_2,\alpha_3,\alpha_4)=(-1)^kX_k(\alpha_2)\TD^{(1,1)}(\alpha_1,-\alpha_2,\alpha_3,\alpha_4).
\end{equation}
We are left to understand how the coefficients  $\Co(\epsilon_1\alpha_1,\epsilon_2\alpha_2,\epsilon_3\alpha_3,\epsilon_4\alpha_4)$, defined by \eqref{Co def}, change
if one changes the sign  $\epsilon_2$ and multiplies by $(-1)^kX_k(\alpha_2)$.
The coefficient $\Co(\ast,\alpha_2,\ast,\ast)$ does not contain $X_k(\alpha_2)$, so it does not change after the sign change of $\alpha_2$, and being multiplied by $(-1)^kX_k(\alpha_2)$ it becomes $\Co(\ast,-\alpha_2,\ast,\ast)$.
The coefficient $\Co(\ast,-\alpha_2,\ast,\ast)$ contains $(-1)^kX_k(\alpha_2)$, so after the sign change of $\alpha_2$ and multiplication by $(-1)^kX_k(\alpha_2)$ we obtain $\Co(\ast,\alpha_2,\ast,\ast)$ since
$$(-1)^kX_k(-\alpha_2)(-1)^kX_k(\alpha_2)=1.$$
This means that we just need to change $\alpha_2$ by $-\alpha_2$ in \eqref{UW11 final} and \eqref{S11}.
Keeping this in mind and using \eqref{TD11 to UEhEdEd}), we finally derive the following decomposition.

\begin{lem}\label{lem:td12} For $|\alpha_j|\ll \epsilon$ one has
\begin{multline}\label{TD12 to to UEhEdEd}
\TD^{(1,2)}(\alpha_1,\alpha_2,\alpha_3,\alpha_4)=
\mainU_{1,2}(\AlphaVec)+\Sok_{1,2}(\AlphaVec)+
(-1)^kX_k(\alpha_2)
\left(\term_{h}^{1}+\term_{d}^{1}+\term_{c}^{1}\right)(\alpha_1,-\alpha_2,\alpha_3,\alpha_4)\\+O(k^{\epsilon-1}),
\end{multline}
where $\term_{h}^{1}(\AlphaVec)$, $\term_{d}^{1}(\AlphaVec)$ and $\term_{c}^{1}(\AlphaVec)$ are given by \eqref{Eh11},\eqref{Ed11} and \eqref{Ec011def}, respectively, and
\begin{multline}\label{UW12 final}
\mainU_{1,2}(\AlphaVec):=
\Co(\alpha_1,-\alpha_2,\alpha_3,-\alpha_4)\ZF(\alpha_1,-\alpha_2,\alpha_3,-\alpha_4)
\frac{\cos\pi\frac{\alpha_2-\alpha_1+2\alpha_4}{2}}{2\cos\pi\frac{\alpha_4-\alpha_1}{2}\cos\pi\frac{\alpha_4+\alpha_2}{2}}\\+
\Co(-\alpha_1,\alpha_2,\alpha_3,-\alpha_4)\ZF(-\alpha_1,\alpha_2,\alpha_3,-\alpha_4)
\frac{\cos\pi\frac{\alpha_1-\alpha_2+2\alpha_4}{2}}{2\cos\pi\frac{\alpha_4+\alpha_1}{2}\cos\pi\frac{\alpha_4-\alpha_2}{2}}\\+
\Co(\alpha_1,-\alpha_2,-\alpha_3,\alpha_4)\ZF(\alpha_1,-\alpha_2,-\alpha_3,\alpha_4)
\frac{\cos\pi\frac{\alpha_2-\alpha_1+2\alpha_3}{2}}{2\cos\pi\frac{\alpha_3-\alpha_1}{2}\cos\pi\frac{\alpha_3+\alpha_2}{2}}\\+
\Co(-\alpha_1,\alpha_2,-\alpha_3,\alpha_4)\ZF(-\alpha_1,\alpha_2,-\alpha_3,\alpha_4)
\frac{\cos\pi\frac{\alpha_1-\alpha_2+2\alpha_3}{2}}{2\cos\pi\frac{\alpha_3+\alpha_1}{2}\cos\pi\frac{\alpha_3-\alpha_2}{2}},
\end{multline}
\begin{multline}\label{S12}
\Sok_{1,2}(\AlphaVec):=
\Co(-\alpha_1,-\alpha_2,-\alpha_3,-\alpha_4)\ZF(-\alpha_1,-\alpha_2,-\alpha_3,-\alpha_4)
\frac{\cos\pi\frac{\alpha_1+\alpha_2+2\alpha_3+2\alpha_4}{2}}{2\cos\pi\frac{\alpha_4+\alpha_3}{2}\cos\pi\frac{\alpha_1+\alpha_2+\alpha_3+\alpha_4}{2}}\\+
\Co(\alpha_1,\alpha_2,-\alpha_3,-\alpha_4)\ZF(\alpha_1,\alpha_2,-\alpha_3,-\alpha_4)
\frac{\cos\pi\frac{-\alpha_2-\alpha_1+2\alpha_3+2\alpha_4}{2}}{2\cos\pi\frac{\alpha_4+\alpha_3}{2}\cos\pi\frac{-\alpha_2-\alpha_1+\alpha_3+\alpha_4}{2}}.
\end{multline}
\end{lem}

\section{Analysis of $\TD^{(2,1)}$}\label{sec:TD21}

In order to describe discrete and holomorphic contributions coming from the term   \eqref{TD21}, we introduce the following weighted fourth moments of $L$-functions associated to Maass and holomorphic cusp forms:
\begin{multline}\label{Ed21}
\term_{d}^{2}(\AlphaVec):=
2(2\pi)^{\alpha_1-\alpha_2-1}
\sum_{j} \omega_j
L_j\left(\frac{1-\alpha_1+\alpha_2+\alpha_3-\alpha_4}{2}\right)\\\times
L_j\left(\frac{1-\alpha_1+\alpha_2-\alpha_3+\alpha_4}{2}\right)
L_j\left(\frac{1-\alpha_1-\alpha_2+\alpha_3+\alpha_4}{2}\right)
L_j\left(\frac{1+\alpha_1+\alpha_2+\alpha_3+\alpha_4}{2}\right)\\\times
\left(\widehat{W}_{2,1}^{+}+\varepsilon_j\widehat{W}_{2,1}^{-}\right)(\alpha_3-\alpha_4,\alpha_1-\alpha_2;t_j),
\end{multline}
\begin{multline}\label{Eh21} 
\term_{h}^{2}(\AlphaVec):=2(2\pi)^{\alpha_1-\alpha_2-1}
\sum_{m=6}^{\infty}\frac{2m-1}{\pi^2}\sum_{f\in H_{2m}} \omega_f
L_f\left(\frac{1-\alpha_1+\alpha_2+\alpha_3-\alpha_4}{2}\right)\\\times
L_f\left(\frac{1-\alpha_1+\alpha_2-\alpha_3+\alpha_4}{2}\right)
L_f\left(\frac{1-\alpha_1-\alpha_2+\alpha_3+\alpha_4}{2}\right)
L_f\left(\frac{1+\alpha_1+\alpha_2+\alpha_3+\alpha_4}{2}\right)\\\times
\widehat{W}_{2,1}(\alpha_3-\alpha_4,\alpha_1-\alpha_2;m).
\end{multline}
The continuous contribution is given by the weighted eighth moment of the Riemann zeta function:
\begin{multline}\label{Ec021def}
\term_{c}^{2}(\AlphaVec):=\frac{4(-1)^k(2\pi)^{\alpha_1-\alpha_2-1}}{2\pi i}\int_{(0)}
\ZE_{2,1}(\AlphaVec;\xi)\\\times
\left(\widehat{W}_{2,1}^{+}(\alpha_3-\alpha_4,\alpha_1-\alpha_2;-i\xi)+\widehat{W}_{2,1}^{-}(\alpha_3-\alpha_4,\alpha_1-\alpha_2;-i\xi)\right)d\xi,
\end{multline}
where
\begin{multline}\label{ZE21def}
\ZE_{2,1}(\AlphaVec;\xi)=
\zeta\left(\frac{1+\alpha_1+\alpha_2+\alpha_3+\alpha_4}{2}+\xi\right)
\zeta\left(\frac{1-\alpha_1-\alpha_2+\alpha_3+\alpha_4}{2}+\xi\right)\\\times
\zeta\left(\frac{1+\alpha_1+\alpha_2+\alpha_3+\alpha_4}{2}-\xi\right)
\zeta\left(\frac{1-\alpha_1-\alpha_2+\alpha_3+\alpha_4}{2}-\xi\right)
\frac{\ZC(\alpha_3-\alpha_4,\alpha_1-\alpha_2,\xi)}{\zeta(1+2\xi)\zeta(1-2\xi)},
\end{multline}
and $\ZC(\alpha,\beta,\xi)$ is the product of four Riemann zeta functions defined by \eqref{Z4 def}.

Recall that the function $W_{2,1}(y)$ is defined by \eqref{W21def} and its transforms used in the definitions of \eqref{Ed21}, \eqref{Eh21} and \eqref{Ec021def} are given in Theorem \ref{Th:BADP}.

Finally, in order to state the main result of this section we introduce the following terms:
\begin{multline}\label{UW21 final}
\mainU_{2,1}(\AlphaVec):=
\Co(\alpha_1,-\alpha_2,\alpha_3,-\alpha_4)\ZF(\alpha_1,-\alpha_2,\alpha_3,-\alpha_4)
\frac{\cos\pi\frac{\alpha_1+\alpha_2}{2}}{2\cos\pi\frac{\alpha_4+\alpha_2}{2}\cos\pi\frac{\alpha_4-\alpha_1}{2}}\\+
\Co(\alpha_1,-\alpha_2,-\alpha_3,\alpha_4)\ZF(\alpha_1,-\alpha_2,-\alpha_3,\alpha_4)
\frac{\cos\pi\frac{\alpha_1+\alpha_2}{2}}{2\cos\pi\frac{\alpha_3-\alpha_1}{2}\cos\pi\frac{\alpha_3+\alpha_2}{2}}\\+
\Co(-\alpha_1,\alpha_2,\alpha_3,-\alpha_4)\ZF(-\alpha_1,\alpha_2,\alpha_3,-\alpha_4)
\frac{\cos\pi\frac{\alpha_1+\alpha_2}{2}}{2\cos\pi\frac{\alpha_4+\alpha_1}{2}\cos\pi\frac{\alpha_4-\alpha_2}{2}}\\+
\Co(-\alpha_1,\alpha_2,-\alpha_3,\alpha_4)\ZF(-\alpha_1,\alpha_2,-\alpha_3,\alpha_4)
\frac{\cos\pi\frac{\alpha_1+\alpha_2}{2}}{2\cos\pi\frac{\alpha_3+\alpha_1}{2}\cos\pi\frac{\alpha_3-\alpha_2}{2}}
\end{multline}
and
\begin{multline}\label{S21}
\Sok_{2,1}(\AlphaVec)=
\Co(-\alpha_1,-\alpha_2,-\alpha_3,-\alpha_4)\ZF(-\alpha_1,-\alpha_2,-\alpha_3,-\alpha_4)
\frac{\cos\pi\frac{\alpha_1+\alpha_2}{2}}{2\cos\pi\frac{\alpha_4+\alpha_3}{2}\cos\pi\frac{\alpha_1+\alpha_2+\alpha_3+\alpha_4}{2}}\\+
\Co(\alpha_1,\alpha_2,-\alpha_3,-\alpha_4)\ZF(\alpha_1,\alpha_2,-\alpha_3,-\alpha_4)
\frac{\cos\pi\frac{\alpha_1+\alpha_2}{2}}{2\cos\pi\frac{\alpha_4+\alpha_3}{2}\cos\pi\frac{\alpha_1+\alpha_2-\alpha_3-\alpha_4}{2}},
\end{multline}
where $\ZF(\epsilon_1\alpha_1,\epsilon_2\alpha_2,\epsilon_3\alpha_3,\epsilon_4\alpha_4)$ and $\Co(\epsilon_1\alpha_1,\epsilon_2\alpha_2,\epsilon_3\alpha_3,\epsilon_4\alpha_4)$ are defined by \eqref{7zeta def} and \eqref{Co def}, respectively.

\begin{lem}\label{lem:td21}
For $|\alpha_j|\ll \epsilon$ one has
\begin{equation}
\TD^{(2,1)}(\AlphaVec)=\mainU_{2,1}(\AlphaVec)+\Sok_{2,1}(\AlphaVec)+\term_{d}^{2}(\AlphaVec)+\term_{h}^{2}(\AlphaVec)+\term_{c}^{2}(\AlphaVec).
\end{equation}
\end{lem}

In order to prove Lemma \ref{lem:td21} we  apply Theorem \ref{Th:BADP}. Accordingly, for
\begin{equation}\label{TD21 alpha conditions0}
\min_{j=3,4}\Re{\alpha_j}>1+\frac{1}{2}|\Re{\alpha_1}+\Re{\alpha_2}|+\frac{1}{2}|\Re{\alpha_1}-\Re{\alpha_2}|,\quad
|\Re{\alpha_3}-\Re{\alpha_4}|<1+\Re{\alpha_1}-\Re{\alpha_2}
\end{equation}
one has
\begin{equation}\label{TD21 to UEhEdEd}
\TD^{(2,1)}(\AlphaVec)=\U(\AlphaVec;W_{2,1})+\E_h(\AlphaVec;W_{2,1})+\E_d(\AlphaVec;W_{2,1})+\E_c(\AlphaVec;W_{2,1}),
\end{equation}
where
\begin{equation}\label{UW21 def}
\U(\AlphaVec;W_{2,1})=
\zeta(1+\alpha_3+\alpha_4)\sum_{n=1}^{\infty}
\frac{\sigma_{\alpha_1+\alpha_2}(n)}{n^{1+\alpha_1+\alpha_3}}U(n,\alpha_3-\alpha_4,\alpha_1-\alpha_2;W_{2,1}),
\end{equation}
and for $\ast\in\{h,d,c\}$
\begin{equation}\label{EhdcW21 def}
\E_{\ast}(\AlphaVec;W_{2,1})=(-1)^k
\zeta(1+\alpha_3+\alpha_4)\sum_{n=1}^{\infty}\frac{\sigma_{\alpha_1+\alpha_2}(n)}{n^{1+\alpha_1+\alpha_3}}
E_{\ast}(n,\alpha_3-\alpha_4,\alpha_1-\alpha_2;W_{2,1}).
\end{equation}

Substituting \eqref{Eh def} to \eqref{EhdcW21 def} and using \eqref{Lprod}, we show that
\begin{equation}
\E_{h}(\AlphaVec;W_{2,1})=\term_{h}^2(\AlphaVec),
\end{equation}
where $\term_{h}^2(\AlphaVec)$ is defined by \eqref{Eh21}.

The term  $\term_{h}^2(\AlphaVec)$ converges in the range $|\alpha_j|\ll A$. This follows from the integral representation   \eqref{hatWmtoPhi} and the estimate \eqref{Phik LGest} since  $W_{2,1}(y)=0$ unless $y>x_0$. In the next lemma  we investigate the term $\E_d(\AlphaVec;W_{2,1})$. The term $\E_c(\AlphaVec;W_{2,1})$ converges in the same range of $\alpha_j$ because the weight functions in \eqref{Ed def} and \eqref{Ec def} coincide.


The following result can be proved in the same way as Lemma  \ref{Ed11-lem}.
\begin{lem}\label{Eh21-lem}
For
\begin{equation}\label{Eh21 alpha conditions0}
\Re{\alpha_3}+\Re{\alpha_4}>-1+|\Re{\alpha_1}+\Re{\alpha_2}|
\end{equation}
one has
\begin{equation}
\E_d(\AlphaVec;W_{2,1})=\term_{d}^{2}(\AlphaVec),
\end{equation}
where $\term_{d}^{2}(\AlphaVec)$ is defined by \eqref{Ed21}.
\end{lem}


\begin{lem}\label{U21-lem}
For
\begin{equation}
2\min(\Re{\alpha_3},\Re{\alpha_4})>|\Re{\alpha_1}+\Re{\alpha_2}|+|\Re{\alpha_1}-\Re{\alpha_2}|
\end{equation}
the following identity holds:
\begin{multline}\label{UW21}
\U(\AlphaVec;W_{2,1})=\\
\RF(\alpha_4-\alpha_1)\RF(\alpha_4+\alpha_2)\ZF(\alpha_1,-\alpha_2,\alpha_3,-\alpha_4)
\int_0^{\infty}W_{2,1}(x)x^{\alpha_3-\alpha_4}(1+x)^{\alpha_1-\alpha_2}dx\\+
\RF(\alpha_3-\alpha_1)\RF(\alpha_3+\alpha_2)\ZF(\alpha_1,-\alpha_2,-\alpha_3,\alpha_4)\int_0^{\infty}W_{2,1}(x)(1+x)^{\alpha_1-\alpha_2}dx\\+
\RF(\alpha_4+\alpha_1)\RF(\alpha_4-\alpha_2)\ZF(-\alpha_1,\alpha_2,\alpha_3,-\alpha_4)\int_0^{\infty}W_{2,1}(x)x^{\alpha_3-\alpha_4}dx\\+
\RF(\alpha_3+\alpha_1)\RF(\alpha_3-\alpha_2)\ZF(-\alpha_1,\alpha_2,-\alpha_3,\alpha_4)\int_0^{\infty}W_{2,1}(x)dx.
\end{multline}
\end{lem}
\begin{proof}
Substituting \eqref{Udef} and \eqref{mul def} to \eqref{UW21 def}, we obtain four summands. Let us consider the first one (other summands can be treated in the same way):
\begin{multline}\label{UW21 eq1}
\zeta(1+\alpha_3+\alpha_4)\frac{\zeta(1+\alpha_1-\alpha_2)\zeta(1+\alpha_3-\alpha_4)}{\zeta(2+\alpha_1-\alpha_2+\alpha_3-\alpha_4)}
\sum_{n=1}^{\infty}\frac{\sigma_{\alpha_1+\alpha_2}(n)}{n^{1+\alpha_1+\alpha_3}}
\sigma_{1+\alpha_1-\alpha_2+\alpha_3-\alpha_4}(n)\\\times
\int_0^{\infty}W_{2,1}(x)x^{\alpha_3-\alpha_4}(1+x)^{\alpha_1-\alpha_2}dx.
\end{multline}
Using \eqref{RamId} one has
\begin{equation}\label{Ram211}
\sum_{n=1}^{\infty}\frac{\sigma_{\alpha_1+\alpha_2}(n)}{n^{1+\alpha_1+\alpha_3}}
\sigma_{1+\alpha_1-\alpha_2+\alpha_3-\alpha_4}(n)=\frac{\zeta(1+\alpha_1+\alpha_3)}{\zeta(1+\alpha_3+\alpha_4)}\zeta(1+\alpha_3-\alpha_2)
\zeta(\alpha_4-\alpha_1)\zeta(\alpha_4+\alpha_2).
\end{equation}
Substituting \eqref{Ram211} to  \eqref{UW21 eq1}, applying  \eqref{zeta functional} to the product $\zeta(\alpha_4-\alpha_1)\zeta(\alpha_4+\alpha_2)$ and using the notation \eqref{7zeta def}, we obtain the first summand in \eqref{UW21}.
\end{proof}

\begin{lem}\label{W21-1lem}
The following asymptotic formula holds:
\begin{multline}\label{W21 int1}
\int_0^{\infty}W_{2,1}(x)x^{\alpha_3-\alpha_4}(1+x)^{\alpha_1-\alpha_2}dx=
2(2\pi)^{\alpha_1-\alpha_2-2\alpha_4}X_k(\alpha_2)X_k(\alpha_4)\\\times
\Gamma(\alpha_4-\alpha_1)\Gamma(\alpha_4+\alpha_2)\cos\pi\frac{\alpha_1+\alpha_2}{2}+O(k^{-2}).
\end{multline}
\end{lem}
\begin{proof}
Using \eqref{DU def} and \eqref{W21def} we have
\begin{multline}\label{W21 int1-1}
\int_0^{\infty}W_{2,1}(x)x^{\alpha_3-\alpha_4}(1+x)^{\alpha_1-\alpha_2}dx=
\int_0^{\infty}\Phi_k\left(\frac{x}{1+x};\alpha_1,\alpha_2\right)(1+x)^{\alpha_1/2-\alpha_2/2}x^{-1-\alpha_4}dx\\-
\int_0^{\infty}\DU(x)\Phi_k\left(\frac{x}{1+x};\alpha_1,\alpha_2\right)(1+x)^{\alpha_1/2-\alpha_2/2}x^{-1-\alpha_4}dx.
\end{multline}
To evaluate the first integral, we use \eqref{Phik Mellin}, \eqref{J def}. As a result, one has
\begin{multline}\label{W21 int1-2}
\int_0^{\infty}\Phi_k\left(\frac{x}{1+x};\alpha_1,\alpha_2\right)(1+x)^{\alpha_1/2-\alpha_2/2}x^{-1-\alpha_4}dx=
(-1)^k(2\pi)^{\alpha_1+\alpha_2}
\frac{1}{2\pi i}\int_{(\sigma)}\Gamma(k,\alpha_1,\alpha_2;s)\\\times\sin\left(\pi\frac{s+\alpha_1+\alpha_2}{2}\right)
\int_0^{\infty}x^{-1/2-\alpha_4-s/2}(1+x)^{-1/2+\alpha_1+s/2}dxds,
\end{multline}
where $1-2k<\sigma<1-2\max(|\Re{\alpha_1}|,|\Re{\alpha_2}|,|\Re{\alpha_4}|)$.
Evaluating the integral over $x$ with the use of Lemma \ref{Lemma beta integrals}  and substituting \eqref{Gamma(k alpha s) def}, we prove that
\begin{multline}\label{W21 int1-3}
\int_0^{\infty}\Phi_k\left(\frac{x}{1+x};\alpha_1,\alpha_2\right)(1+x)^{\alpha_1/2-\alpha_2/2}x^{-1-\alpha_4}dx=
(-1)^k(2\pi)^{\alpha_1+\alpha_2}\Gamma(\alpha_4-\alpha_1)\\\times
\frac{1}{2\pi i}\int_{(\sigma)}
\frac{\Gamma(k-1/2+s/2)}{\Gamma(k+1/2-s/2)}\Gamma(1/2-\alpha_2-s/2)\Gamma(1/2-\alpha_4-s/2)
\sin\left(\pi\frac{s+\alpha_1+\alpha_2}{2}\right)ds.
\end{multline}
Moving the line of integration to the left, evaluating the residues at $s_j=1-2k-2j$ and using \cite[15.4.20]{HMF}, one has
\begin{multline}\label{W21 int1-4}
\int_0^{\infty}\Phi_k\left(\frac{x}{1+x};\alpha_1,\alpha_2\right)(1+x)^{\alpha_1/2-\alpha_2/2}x^{-1-\alpha_4}dx=\\
2(2\pi)^{\alpha_1+\alpha_2}\Gamma(\alpha_4-\alpha_1)\cos\pi\frac{\alpha_1+\alpha_2}{2}
\HyGI(k-\alpha_2,k-\alpha_4,2k;1)=\\2(2\pi)^{\alpha_1-\alpha_2-2\alpha_4}X_k(\alpha_2)X_k(\alpha_4)
\Gamma(\alpha_4-\alpha_1)\Gamma(\alpha_4+\alpha_2)\cos\pi\frac{\alpha_1+\alpha_2}{2}.
\end{multline}

To estimate the second integral in \eqref{W21 int1-1}, we again use \eqref{Phik Mellin}, \eqref{J def}. Consequently,
\begin{multline}\label{W21 int1-6}
\int_0^{\infty}\DU(x)\Phi_k\left(\frac{x}{1+x};\alpha_1,\alpha_2\right)(1+x)^{\alpha_1/2-\alpha_2/2}x^{-1-\alpha_4}dx=
(-1)^k(2\pi)^{\alpha_1+\alpha_2}\\\times
\frac{1}{2\pi i}\int_{(\sigma)}\Gamma(k,\alpha_1,\alpha_2;s)\sin\left(\pi\frac{s+\alpha_1+\alpha_2}{2}\right)
\int_0^{\infty}\DU(x)x^{-1/2-\alpha_4-\sigma}(1+x)^{-1/2+\alpha_1+\sigma}(1+1/x)^{iy}dxds,
\end{multline}
where $s/2=\sigma+iy$. Integrating by parts $A$-times the integral over $x$,
one has
\begin{equation}\label{W21 int1-7}
\int_0^{\infty}\DU(x)x^{-1/2-\alpha_4-\sigma}(1+x)^{-1/2+\alpha_1+\sigma}(1+1/x)^{iy}dx\ll\frac{1}{(1+|y|)^A}.
\end{equation}
Moving the line of integration in \eqref{W21 int1-6} to $\Re{s}=-1/2$, applying  \eqref{W21 int1-7} and \eqref{Stirling0} to estimate the Gamma factors, we infer
\begin{equation}\label{W21 int1-8}
\int_0^{\infty}\DU(x)\Phi_k\left(\frac{x}{1+x};\alpha_1,\alpha_2\right)(1+x)^{\alpha_1/2-\alpha_2/2}x^{-1-\alpha_4}dx\ll \frac{1}{k^2}.
\end{equation}
Therefore, \eqref{W21 int1} follows from  \eqref{W21 int1-1}, \eqref{W21 int1-4} and \eqref{W21 int1-8}.
\end{proof}

\begin{lem}\label{W21-2lem}
The following asymptotic formula holds:
\begin{multline}\label{W21 int2}
\int_0^{\infty}W_{2,1}(x)(1+x)^{\alpha_1-\alpha_2}dx=\\
2(2\pi)^{\alpha_1-\alpha_2-2\alpha_3}X_k(\alpha_2)X_k(\alpha_3)
\Gamma(\alpha_3-\alpha_1)\Gamma(\alpha_3+\alpha_2)\cos\pi\frac{\alpha_1+\alpha_2}{2}+O(k^{-2}).
\end{multline}
\end{lem}
\begin{proof}
According to \eqref{W21def} one has
\begin{equation}\label{W21 int2-1}
\int_0^{\infty}W_{2,1}(x)(1+x)^{\alpha_1-\alpha_2}dx=
\int_0^{\infty}\DU(1/x)\Phi_k\left(\frac{x}{1+x};\alpha_1,\alpha_2\right)(1+x)^{\alpha_1/2-\alpha_2/2}x^{-1-\alpha_3}dx.
\end{equation}
Therefore, changing $\alpha_4$ in \eqref{W21 int1-1}  by $\alpha_3$, we recover \eqref{W21 int2-1}. Further, making the same change in
\eqref{W21 int1}, we obtain \eqref{W21 int2}.
\end{proof}

\begin{lem}\label{W21-3lem}
The following asymptotic formula holds:
\begin{multline}\label{W21 int3}
\int_0^{\infty}W_{2,1}(x)x^{\alpha_3-\alpha_4}dx=\\
2(2\pi)^{-\alpha_1+\alpha_2-2\alpha_4}X_k(\alpha_1)X_k(\alpha_4)
\Gamma(\alpha_4+\alpha_1)\Gamma(\alpha_4-\alpha_2)\cos\pi\frac{\alpha_1+\alpha_2}{2}+O(k^{-2}).
\end{multline}
\end{lem}
\begin{proof}
Applying \eqref{W21def} one has
\begin{multline}\label{W21 int3-1}
\int_0^{\infty}W_{2,1}(x)x^{\alpha_3-\alpha_4}dx=
\int_0^{\infty}\DU(1/x)\Phi_k\left(\frac{x}{1+x};\alpha_1,\alpha_2\right)(1+x)^{-\alpha_1/2+\alpha_2/2}x^{-1-\alpha_4}dx.
\end{multline}
  Note that \eqref{Phi_k2} yields the identity $\Phi_k\left(y;\alpha_1,\alpha_2\right)=\Phi_k\left(y;\alpha_2,\alpha_1\right)$.  Therefore, changing $\alpha_1\leftrightarrow\alpha_2$ in \eqref{W21 int1-1},  we obtain \eqref{W21 int3-1}. Further, making the same change in
\eqref{W21 int1}, we recover \eqref{W21 int3}.
\end{proof}

\begin{lem}\label{W21-4lem}
The following asymptotic formula holds:
\begin{multline}\label{W21 int4}
\int_0^{\infty}W_{2,1}(x)dx=
2(2\pi)^{-\alpha_1+\alpha_2-2\alpha_3}X_k(\alpha_1)X_k(\alpha_3)
\Gamma(\alpha_3+\alpha_1)\Gamma(\alpha_3-\alpha_2)\cos\pi\frac{\alpha_1+\alpha_2}{2}\\+O(k^{-2}).
\end{multline}
\end{lem}
\begin{proof}
According to \eqref{W21def} one has
\begin{equation}\label{W21 int4-1}
\int_0^{\infty}W_{2,1}(x)dx=
\int_0^{\infty}\DU(1/x)\Phi_k\left(\frac{x}{1+x};\alpha_1,\alpha_2\right)(1+x)^{-\alpha_1/2+\alpha_2/2}x^{-1-\alpha_3}dx.
\end{equation}
Therefore, changing $\alpha_4$ in \eqref{W21 int3-1}  by $\alpha_3$, we recover \eqref{W21 int4-1}. Making the same change in
\eqref{W21 int3}, we obtain \eqref{W21 int4}.
\end{proof}

Combining Lemmas \ref{U21-lem}-\ref{W21-4lem}, we finally prove the following identity.

\begin{lem}\label{U21-lem final}
One has
\begin{equation}
\U(\AlphaVec;W_{2,1})=\mainU_{2,1}(\AlphaVec),
\end{equation}
where 
$\mainU_{2,1}(\AlphaVec)$ is defined by \eqref{UW21 final}.
\end{lem}
\begin{proof}
Substituting \eqref{W21 int1}, \eqref{W21 int2}, \eqref{W21 int3}, \eqref{W21 int4} into \eqref{UW21} and using \eqref{zeta functional}, we show that
\eqref{UW21 final} holds.
\end{proof}

The region \eqref{TD21 alpha conditions0} does not contain the most interesting points  $\Re{\alpha_j}=0$. In fact, the only obstacle in proving the analytic continuation
for \eqref{TD21 to UEhEdEd} is the term $\E_c(\AlphaVec;W_{2,1})$. In order to deal with this term, it is required to compute the sum
\begin{equation}
\widehat{W}_{2,1}^{+}(\alpha_3-\alpha_4,\alpha_1-\alpha_2;-i\xi)+\widehat{W}_{2,1}^{-}(\alpha_3-\alpha_4,\alpha_1-\alpha_2;-i\xi)
\end{equation}
at several special points $\xi.$
\begin{lem}\label{Lem:hatW21-xi1}
For $|\alpha_j|\ll\epsilon$ and $\xi_1=\frac{1-\alpha_1-\alpha_2-\alpha_3-\alpha_4}{2}$ one has
\begin{multline}\label{hatW21 xi1}
\widehat{W}_{2,1}^{+}(\alpha_3-\alpha_4,\alpha_1-\alpha_2;-i\xi_1)+\widehat{W}_{2,1}^{-}(\alpha_3-\alpha_4,\alpha_1-\alpha_2;-i\xi_1)
=\\=
\frac{\pi}{2}(2\pi)^{-\alpha_1+\alpha_2}X_k(\alpha_1)
\Gamma(\alpha_2+\alpha_3)\Gamma(\alpha_2+\alpha_4)\Gamma(\alpha_3+\alpha_4)
\frac{\Gamma(k-\alpha_2-\alpha_3-\alpha_4)}{\Gamma(k+\alpha_2+\alpha_3+\alpha_4)}\\\times
\frac{\cos\pi\frac{\alpha_3+\alpha_4+2\alpha_2}{2}+\cos\pi\frac{\alpha_3-\alpha_4}{2}}{\cos\pi\frac{\alpha_1+\alpha_2+\alpha_3+\alpha_4}{2}}
\cos\pi\frac{\alpha_1+\alpha_2}{2}+O\left(\frac{1}{k^{2-\epsilon}}\right).
\end{multline}
\end{lem}
\begin{proof}
It follows from \eqref{hatW+ + hatW-}, \eqref{hatW0} and \eqref{W21def} that
\begin{multline}\label{W21hat+-xi1 eq1}
\widehat{W}_{2,1}^{+}(\alpha_3-\alpha_4,\alpha_1-\alpha_2;-i\xi_1)+\widehat{W}_{2,1}^{-}(\alpha_3-\alpha_4,\alpha_1-\alpha_2;-i\xi_1)
=\\=
\frac{\pi}{4}\frac{\cos\pi\frac{2\alpha_1+\alpha_3+\alpha_4}{2}-\cos\pi\frac{\alpha_3-\alpha_4}{2}}{\cos\pi\frac{\alpha_1+\alpha_2+\alpha_3+\alpha_4}{2}}
\int_{0}^{\infty}
\DU(1/y)\Phi_k\left(\frac{y}{1+y};\alpha_1,\alpha_2\right)
\frac{y^{-2+\alpha_1}}{(1+y)^{\alpha_1/2-\alpha_2/2}}\\\times
\HyGI\left(1-\alpha_1-\alpha_4,1-\alpha_1-\alpha_3,2-\alpha_1-\alpha_2-\alpha_3-\alpha_4;\frac{-1}{y}\right)dy\\+
\frac{\pi}{4}\frac{\cos\pi\frac{2\alpha_2+\alpha_3+\alpha_4}{2}+\cos\pi\frac{\alpha_3-\alpha_4}{2}}{\cos\pi\frac{\alpha_1+\alpha_2+\alpha_3+\alpha_4}{2}}
\int_{0}^{\infty}
\DU(1/y)\Phi_k\left(\frac{y}{1+y};\alpha_1,\alpha_2\right)
\frac{y^{-1-\alpha_2-\alpha_3-\alpha_4}}{(1+y)^{\alpha_1/2-\alpha_2/2}}\\\times
\HyGI\left(\alpha_2+\alpha_3,\alpha_2+\alpha_4,\alpha_1+\alpha_2+\alpha_3+\alpha_4;\frac{-1}{y}\right)dy.
\end{multline}
Let us first consider the second integral in \eqref{W21hat+-xi1 eq1}.
Using \eqref{2I1 to 3I2} it can be further splitted into the sum of two integrals:
\begin{multline}\label{W21hat+-xi1 int2eq1}
\int_{0}^{\infty}
\DU(1/y)\Phi_k\left(\frac{y}{1+y};\alpha_1,\alpha_2\right)
\frac{y^{-1-\alpha_2-\alpha_3-\alpha_4}}{(1+y)^{\alpha_1/2-\alpha_2/2}}\\\times
\HyGI\left(\alpha_2+\alpha_3,\alpha_2+\alpha_4,\alpha_1+\alpha_2+\alpha_3+\alpha_4;\frac{-1}{y}\right)dy=\\=
\frac{\Gamma(\alpha_2+\alpha_3)\Gamma(\alpha_2+\alpha_4)}{\Gamma(\alpha_1+\alpha_2+\alpha_3+\alpha_4)}
\int_{0}^{\infty}
\DU(1/y)\Phi_k\left(\frac{y}{1+y};\alpha_1,\alpha_2\right)
\frac{y^{-1-\alpha_2-\alpha_3-\alpha_4}}{(1+y)^{\alpha_1/2-\alpha_2/2}}dy\\-
\int_{0}^{\infty}
\DU(1/y)\Phi_k\left(\frac{y}{1+y};\alpha_1,\alpha_2\right)
\frac{y^{-2-\alpha_2-\alpha_3-\alpha_4}}{(1+y)^{\alpha_1/2-\alpha_2/2}}
\GenHyGI{3}{2}{1+\alpha_2+\alpha_3,1+\alpha_2+\alpha_4,1}{2,1+\alpha_1+\alpha_2+\alpha_3+\alpha_4}{\frac{-1}{y}}dy.
\end{multline}
Now we compute the first integral in \eqref{W21hat+-xi1 int2eq1}. Arguing in the same way as in Lemma \ref{W21-1lem}, we obtain
\begin{multline}\label{W21hat+-xi1 int2eq2}
\frac{\Gamma(\alpha_2+\alpha_3)\Gamma(\alpha_2+\alpha_4)}{\Gamma(\alpha_1+\alpha_2+\alpha_3+\alpha_4)}
\int_{0}^{\infty}
\DU(1/y)\Phi_k\left(\frac{y}{1+y};\alpha_1,\alpha_2\right)
\frac{y^{-1-\alpha_2-\alpha_3-\alpha_4}}{(1+y)^{\alpha_1/2-\alpha_2/2}}dy
=\\=
2(2\pi)^{\alpha_1+\alpha_2}\cos\pi\frac{\alpha_1+\alpha_2}{2}
\Gamma(\alpha_2+\alpha_3)\Gamma(\alpha_2+\alpha_4)\Gamma(\alpha_3+\alpha_4)
\frac{\Gamma(k-\alpha_2-\alpha_3-\alpha_4)\Gamma(k-\alpha_1)}{\Gamma(k+\alpha_2+\alpha_3+\alpha_4)\Gamma(k+\alpha_1)}+
O\left(\frac{1}{k^2}\right).
\end{multline}
To estimate the first integral in \eqref{W21hat+-xi1 eq1} and the second one in \eqref{W21hat+-xi1 int2eq1}, we make the change of variable $y=\sinh^{-2}\frac{\sqrt{x}}{2}$ and apply \eqref{Phik est}. The integral over $x>k^{\epsilon-2}$ is negligible since the $K$-Bessel function in \eqref{Phik est} decays exponentially. Using the estimate $K_{\nu}(x)\ll x^{-|\Re{\nu}|}$ (see \cite[10.27.2, 10.30.2]{HMF}), the remaining integral over $(0,k^{\epsilon-2})$ is bounded by $k^{\epsilon-2}$.
\end{proof}
\begin{lem}\label{Lem:hatW21-xi2}
For $|\alpha_j|\ll\epsilon$ and $\xi_2=\frac{1+\alpha_1+\alpha_2-\alpha_3-\alpha_4}{2}$ one has
\begin{multline}\label{hatW21 xi2}
\widehat{W}_{2,1}^{+}(\alpha_3-\alpha_4,\alpha_1-\alpha_2;-i\xi_2)+\widehat{W}_{2,1}^{-}(\alpha_3-\alpha_4,\alpha_1-\alpha_2;-i\xi_2)
=\\=
\frac{\pi}{2}(2\pi)^{-\alpha_1+\alpha_2}X_k(\alpha_1)
\Gamma(\alpha_3+\alpha_4)\Gamma(\alpha_3-\alpha_1)\Gamma(\alpha_4-\alpha_1)
\frac{\Gamma(k+\alpha_1-\alpha_3-\alpha_4)}{\Gamma(k-\alpha_1+\alpha_3+\alpha_4)}\\\times
\frac{\cos\pi\frac{\alpha_3+\alpha_4-2\alpha_2}{2}+\cos\pi\frac{\alpha_3-\alpha_4}{2}}{\cos\pi\frac{\alpha_1+\alpha_2-\alpha_3-\alpha_4}{2}}
\cos\pi\frac{\alpha_1+\alpha_2}{2}+O\left(\frac{1}{k^{2-\epsilon}}\right).
\end{multline}
\end{lem}
\begin{proof}
It follows from \eqref{hatW+ + hatW-}, \eqref{hatW0} and \eqref{W21def} that
\begin{multline}\label{W21hat+-xi2 eq1}
\widehat{W}_{2,1}^{+}(\alpha_3-\alpha_4,\alpha_1-\alpha_2;-i\xi_2)+\widehat{W}_{2,1}^{-}(\alpha_3-\alpha_4,\alpha_1-\alpha_2;-i\xi_2)
=\\=
\frac{\pi}{4}\frac{\cos\pi\frac{-2\alpha_2+\alpha_3+\alpha_4}{2}-\cos\pi\frac{\alpha_3-\alpha_4}{2}}{\cos\pi\frac{\alpha_1+\alpha_2-\alpha_3-\alpha_4}{2}}
\int_{0}^{\infty}
\DU(1/y)\Phi_k\left(\frac{y}{1+y};\alpha_1,\alpha_2\right)
\frac{y^{-2-\alpha_2}}{(1+y)^{\alpha_1/2-\alpha_2/2}}\\\times
\HyGI\left(1+\alpha_2-\alpha_4,1+\alpha_2-\alpha_3,2+\alpha_1+\alpha_2-\alpha_3-\alpha_4;\frac{-1}{y}\right)dy\\+
\frac{\pi}{4}\frac{\cos\pi\frac{-2\alpha_1+\alpha_3+\alpha_4}{2}+\cos\pi\frac{\alpha_3-\alpha_4}{2}}{\cos\pi\frac{\alpha_1+\alpha_2-\alpha_3-\alpha_4}{2}}
\int_{0}^{\infty}
\DU(1/y)\Phi_k\left(\frac{y}{1+y};\alpha_1,\alpha_2\right)
\frac{y^{-1+\alpha_1-\alpha_3-\alpha_4}}{(1+y)^{\alpha_1/2-\alpha_2/2}}\\\times
\HyGI\left(\alpha_3-\alpha_1,\alpha_4-\alpha_1,-\alpha_1-\alpha_2+\alpha_3+\alpha_4;\frac{-1}{y}\right)dy.
\end{multline}
Now arguing in the same way as in Lemma \ref{Lem:hatW21-xi1} one can prove \eqref{hatW21 xi2}.
Nevertheless, there is a simpler way to obtain \eqref{hatW21 xi2}, as we now show. Applying \cite[15.8.1]{HMF} one has
\begin{equation}\label{Phi(-)toPhi(+)}
X_k(\alpha_1)X_k(\alpha_2)\Phi_k\left(x;-\alpha_2,-\alpha_1\right)=\Phi_k\left(x;\alpha_1,\alpha_2\right).
\end{equation}
Therefore, making the change $\alpha_1\rightarrow-\alpha_2$, $\alpha_2\rightarrow-\alpha_1$ in \eqref{W21hat+-xi1 eq1} and multiplying the result by $X_k(\alpha_1)X_k(\alpha_2)$, we show that \eqref{W21hat+-xi2 eq1} holds. Hence to prove \eqref{hatW21 xi2} it is only required to perform this change of variables in
\eqref{hatW21 xi1} and multiply the result by $X_k(\alpha_1)X_k(\alpha_2)$.
\end{proof}


Now we are ready to provide an analytic continuation for $\E_c(\AlphaVec;W_{2,1})$ in the lemma below.

\begin{lem}\label{Ec21-lem1}
For $|\alpha_j|\ll\epsilon$ one has
\begin{equation}\label{Ec21}
\E_{c}(\AlphaVec;W_{2,1})=\Sok_{2,1}(\AlphaVec)+\term_{c}^{2}(\AlphaVec),
\end{equation}
where $\Sok_{2,1}(\AlphaVec)$ and $\term_{c}^{2}(\AlphaVec)$ are defined by \eqref{S21} and \eqref{Ec021def}, respectively.

\end{lem}
\begin{proof}
Arguing in the same way as in Lemma \ref{Ec11-lem1}, we obtain \eqref{S21} with
\begin{equation}\label{S21eq2}
\Sok_{2,1}(\AlphaVec)=
8(2\pi)^{\alpha_1-\alpha_2-1}\left(R_{2,1}(\xi_1)+R_{2,1}(\xi_2)\right),
\end{equation}
\begin{multline}\label{R21xi1}
R_{2,1}(\xi_1)=
\ZF(-\alpha_1,-\alpha_2,-\alpha_3,-\alpha_4)\RF(\alpha_2+\alpha_3)\RF(\alpha_2+\alpha_4)\RF(\alpha_3+\alpha_4)\\\times
\left(\widehat{W}_{2,1}^{+}(\alpha_3-\alpha_4,\alpha_1-\alpha_2;-i\xi_1)+\widehat{W}_{2,1}^{-}(\alpha_3+\alpha_4,\alpha_1-\alpha_2;-i\xi_1)\right),
\end{multline}
\begin{multline}\label{R21xi2}
R_{2,1}(\xi_2)=
\ZF(\alpha_1,\alpha_2,-\alpha_3,-\alpha_4)\RF(\alpha_4-\alpha_1)\RF(\alpha_3-\alpha_1)\RF(\alpha_3+\alpha_4)\\\times
\left(\widehat{W}_{2,1}^{+}(\alpha_3-\alpha_4,\alpha_1-\alpha_2;-i\xi_2)+\widehat{W}_{2,1}^{-}(\alpha_3-\alpha_4,\alpha_1-\alpha_2;-i\xi_2)\right),
\end{multline}
where
\begin{equation}\label{Ec21-poles}
\xi_1=\frac{1-\alpha_1-\alpha_2-\alpha_3-\alpha_4}{2},\quad
\xi_2=\frac{1+\alpha_1+\alpha_2-\alpha_3-\alpha_4}{2}.
\end{equation}
Using \eqref{Xdef} and \eqref{Xapprox} one has
\begin{equation}\label{XtoXXX-2}
\frac{\Gamma(k-\alpha_2-\alpha_3-\alpha_4)}{\Gamma(k+\alpha_2+\alpha_3+\alpha_4)}=(2\pi)^{-2\alpha_2-2\alpha_3-2\alpha_4}
X_{k}(\alpha_2)X_{k}(\alpha_3)X_{k}(\alpha_4)\left(1+O(k^{-1})\right),
\end{equation}
\begin{equation}\label{XtoXXX-3}
\frac{\Gamma(k+\alpha_1-\alpha_3-\alpha_4)}{\Gamma(k-\alpha_1+\alpha_3+\alpha_4)}=(2\pi)^{2\alpha_1-2\alpha_3-2\alpha_4}
\frac{X_{k}(\alpha_3)X_{k}(\alpha_4)}{X_{k}(\alpha_1)}\left(1+O(k^{-1})\right).
\end{equation}
Substituting \eqref{hatW21 xi1} into \eqref{R21xi1}, using \eqref{XtoXXX-2} and \eqref{zeta functional}, we show that
\begin{multline}\label{R21xi1eq2}
8(2\pi)^{\alpha_1-\alpha_2-1}R_{2,1}(\xi_1)=
\Co(-\alpha_1,-\alpha_2,-\alpha_3,-\alpha_4)\ZF(-\alpha_1,-\alpha_2,-\alpha_3,-\alpha_4)\\\times
\frac{\cos\pi\frac{\alpha_3+\alpha_4+2\alpha_2}{2}+\cos\pi\frac{\alpha_3-\alpha_4}{2}}
{4\cos\pi\frac{\alpha_2+\alpha_3}{2}\cos\pi\frac{\alpha_2+\alpha_4}{2}\cos\pi\frac{\alpha_3+\alpha_4}{2}}
\frac{\cos\pi\frac{\alpha_1+\alpha_2}{2}}{\cos\pi\frac{\alpha_1+\alpha_2+\alpha_3+\alpha_4}{2}}
+O\left(\frac{1}{k^{2-\epsilon}}\right).
\end{multline}
Substituting \eqref{hatW21 xi2} into \eqref{R21xi2}, applying \eqref{XtoXXX-3} and \eqref{zeta functional}, we conclude that
\begin{multline}\label{R21xi2eq2}
8(2\pi)^{\alpha_1-\alpha_2-1}R_{2,1}(\xi_2)=
\Co(\alpha_1,\alpha_2,-\alpha_3,-\alpha_4)\ZF(\alpha_1,\alpha_2,-\alpha_3,-\alpha_4)\\\times
\frac{\cos\pi\frac{\alpha_3+\alpha_4-2\alpha_1}{2}+\cos\pi\frac{\alpha_3-\alpha_4}{2}}
{4\cos\pi\frac{\alpha_3-\alpha_1}{2}\cos\pi\frac{\alpha_4-\alpha_1}{2}\cos\pi\frac{\alpha_3+\alpha_4}{2}}
\frac{\cos\pi\frac{\alpha_1+\alpha_2}{2}}{\cos\pi\frac{\alpha_1+\alpha_2-\alpha_3-\alpha_4}{2}}
+O\left(\frac{1}{k^{2-\epsilon}}\right).
\end{multline}
Substituting  \eqref{R21xi1eq2}, \eqref{R21xi2eq2} into \eqref{S21eq2} and using \cite[4.21.8]{HMF}, we prove the lemma.
\end{proof}
\section{Analysis of $\TD^{(2,2)}$}\label{sec:TD22}
In this section, we show that the term $\TD^{(2,2)}(\AlphaVec)$ does not contain any part of the main term of the fourth moment. More precisely, we prove that this term is negligible. 

Similarly to other cases, applying Theorem \ref{Th:BADP} to study the sum \eqref{TD22} one has the following decomposition:
\begin{equation}\label{TD22 to UEhEdEd}
\TD^{(2,2)}(\AlphaVec)=\U(\AlphaVec;W_{2,2})+\E_h(\AlphaVec;W_{2,2})+\E_d(\AlphaVec;W_{2,2})+\E_c(\AlphaVec;W_{2,2}).
\end{equation}
The main term $\U(\AlphaVec;W_{2,2})$ of $\TD^{(2,2)}(\AlphaVec)$  has the same structure as $\U(\AlphaVec;W_{2,1})$, see \eqref{UW21}.
It follows from \eqref{W22def} and  \eqref{Phik LGest}  that
\begin{equation}\label{W22 est}
W_{2,2}(\AlphaVec;x)\ll
\frac{\DU(1/x)e^{-(2k-1)\log(\sqrt{1+x}+\sqrt{x})}}{(1+x)^{1/4+\Re{\alpha_1}/2-\Re{\alpha_2}/2}x^{1/4+\Re{\alpha_1}/2+\Re{\alpha_2}/2}k^{1/2+\Re{\alpha_1}+\Re{\alpha_2}}}.
\end{equation}
Hence this function is negligible. Therefore, all integrals of the form like in \eqref{UW21} are negligible, and therefore, $\U(\AlphaVec;W_{2,2})$ is negligible too.
The same is true for the residue contribution $\Sok_{2,2}(\AlphaVec)$  arising after analytic continuation of $\E_c(\AlphaVec;W_{2,2})$. Therefore, we are left to estimate
$\E_{\ast}(\AlphaVec;W_{2,2})$ for $\ast \in\{h,d,c\}$.

The term $\E_{h}(\AlphaVec;W_{2,2})$ is also negligible. This follows from the integral representation   \eqref{hatWmtoPhi} and the estimates \eqref{Phik LGest}, \eqref{W22 est} since  $W_{2,2}(y)=0$ unless $y>x_0$.

The case of $\E_{d}(\AlphaVec;W_{2,2})$ can be treated along the line of the proof of Lemma  \ref{Ed11-lem} with the help of \eqref{W22 est}. After integration by parts, we obtain derivatives of the function $\Phi_k\left(\frac{1}{1+x};\alpha_1,\alpha_2\right)$. Being a hypergeometric function it satisfies a second order differential equation  with coefficients depending on $k$ polynomially. Therefore, all derivatives of $\Phi_k$ could be replaced by $\Phi_k$ and $\Phi'_k$ at the cost of the error $k^{A}$. To get rid of $\Phi'_k$, we integrate by parts once again. Finally, using \eqref{W22 est}, which shows the exponential decay in $k$, we prove the estimate
$\E_{d}(\AlphaVec;W_{2,2})\ll k^{-A}.$ The same arguments lead to $\E_{c}(\AlphaVec;W_{2,2})\ll k^{-A}.$

\section{Analysis of $\TD^{(3)}$}\label{sec:TD3}

\begin{lem}\label{lem:td3} For $|\alpha_j|\ll \epsilon$ one has
\begin{equation}
\TD^{(3)}(\AlphaVec)=\mainU_3(\AlphaVec)+\Sok_3(\AlphaVec)\\+
(-1)^kX_k(\alpha_2)\left(\term_{h}^{2}+\term_{d}^2+\term_c^{2}\right)(\alpha_1,-\alpha_2,\alpha_3,\alpha_4)+O(k^{\epsilon-2}),
\end{equation}
where $\term_{h}^{2}(\AlphaVec)$, $\term_{d}^{2}(\AlphaVec)$ and $\term_{c}^{2}(\AlphaVec)$ are given by \eqref{Eh21}, \eqref{Ed21} and \eqref{Ec021def}, respectively, and
\begin{multline}\label{U3 final}
\mainU_3(\AlphaVec):=
\Co(\alpha_1,\alpha_2,\alpha_3,-\alpha_4)\ZF(\alpha_1,\alpha_2,\alpha_3,-\alpha_4)
\frac{\cos\pi\frac{\alpha_1-\alpha_2}{2}}{2\cos\pi\frac{\alpha_4-\alpha_2}{2}\cos\pi\frac{\alpha_4-\alpha_1}{2}}\\+
\Co(\alpha_1,\alpha_2,-\alpha_3,\alpha_4)\ZF(\alpha_1,\alpha_2,-\alpha_3,\alpha_4)
\frac{\cos\pi\frac{\alpha_1-\alpha_2}{2}}{2\cos\pi\frac{\alpha_3-\alpha_1}{2}\cos\pi\frac{\alpha_3-\alpha_2}{2}}\\+
\Co(-\alpha_1,-\alpha_2,\alpha_3,-\alpha_4)\ZF(-\alpha_1,-\alpha_2,\alpha_3,-\alpha_4)
\frac{\cos\pi\frac{\alpha_1-\alpha_2}{2}}{2\cos\pi\frac{\alpha_4+\alpha_1}{2}\cos\pi\frac{\alpha_4+\alpha_2}{2}}\\+
\Co(-\alpha_1,-\alpha_2,-\alpha_3,\alpha_4)\ZF(-\alpha_1,-\alpha_2,-\alpha_3,\alpha_4)
\frac{\cos\pi\frac{\alpha_1-\alpha_2}{2}}{2\cos\pi\frac{\alpha_3+\alpha_1}{2}\cos\pi\frac{\alpha_3+\alpha_2}{2}},
\end{multline}
\begin{multline}\label{S3}
\Sok_{3}(\AlphaVec):=
\Co(-\alpha_1,\alpha_2,-\alpha_3,-\alpha_4)\ZF(-\alpha_1,\alpha_2,-\alpha_3,-\alpha_4)
\frac{\cos\pi\frac{\alpha_1-\alpha_2}{2}}{2\cos\pi\frac{\alpha_4+\alpha_3}{2}\cos\pi\frac{\alpha_1-\alpha_2+\alpha_3+\alpha_4}{2}}\\+
\Co(\alpha_1,-\alpha_2,-\alpha_3,-\alpha_4)\ZF(\alpha_1,-\alpha_2,-\alpha_3,-\alpha_4)
\frac{\cos\pi\frac{\alpha_1-\alpha_2}{2}}{2\cos\pi\frac{\alpha_4+\alpha_3}{2}\cos\pi\frac{\alpha_1-\alpha_2-\alpha_3-\alpha_4}{2}}.
\end{multline}
\end{lem}
\begin{proof}
It follows from \eqref{4mom to 2mom} and \eqref{ET2 3} that
\begin{equation}\label{TD3 to TD2}
\TD^{(3)}(\alpha_1,\alpha_2,\alpha_3,\alpha_4)=
(-1)^kX_k(\alpha_2)\TD^{(2)}(\alpha_1,-\alpha_2,\alpha_3,\alpha_4).
\end{equation}
Therefore, using  \eqref{TD21 to UEhEdEd}, one has
\begin{equation}\label{TD3 to UEhEdEd}
\TD^{(3)}(\AlphaVec)=\U_3(\AlphaVec)+\Sok_3(\AlphaVec)\\+
(-1)^kX_k(\alpha_2)\left(\E_h+\E_d+\E_c^{(0)}\right)(\alpha_1,-\alpha_2,\alpha_3,\alpha_4;W_{2,1})+O(k^{\epsilon-2}),
\end{equation}
where
\begin{equation}\label{U3 to U21}
\U_3(\AlphaVec)=(-1)^kX_k(\alpha_2)\U(\alpha_1,-\alpha_2,\alpha_3,\alpha_4;W_{2,1}),
\end{equation}
and (see \eqref{Ec21})
\begin{equation}\label{S3 to S21}
\Sok_3(\AlphaVec)=(-1)^kX_k(\alpha_2)\Sok_{2,1}(\alpha_1,-\alpha_2,\alpha_3,\alpha_4).
\end{equation}

In the beginning of Section \ref{sec:TD12}, it was shown that changing the sign of $\epsilon_2$ in
$\Co(\epsilon_1\alpha_1,\epsilon_2\alpha_2,\epsilon_3\alpha_3,\epsilon_4\alpha_4)$  and then multiplying by $(-1)^kX_k(\alpha_2)$, we obtain
$\Co(\epsilon_1\alpha_1,-\epsilon_2\alpha_2,\epsilon_3\alpha_3,\epsilon_4\alpha_4)$. Therefore, using \eqref{UW21 final}, one has $\U_3(\AlphaVec)=\mainU_{3}(\AlphaVec)$.
Applying \eqref{S21}, we show that \eqref{S3} holds.

\end{proof}

\section{Proof of Theorem \ref{thm:main}}\label{sec:proofTh}

In order to prove the main theorem, we use the formula \eqref{4mom to 2mom} for the fourth moment together with the decomposition \eqref{TD1=TD11+TD12}
and Lemmas \ref{lem:td11}, \ref{lem:td12}, \ref{lem:td21} and \ref{lem:td3}. Further, note that the contribution of the term $\TD^{(2,2)}(\AlphaVec)$ to the fourth moment is negligible according to Section \ref{sec:TD22}.

The final step is to rearrange slightly all main terms.

Summing \eqref{UW12 final} and \eqref{UW21 final}, and using \cite[4.21.16]{HMF}, we infer
\begin{multline}\label{UW12+UW21}
\mainU_{1,2}(\AlphaVec)+\mainU_{2,1}(\AlphaVec)=
\Co(\alpha_1,-\alpha_2,\alpha_3,-\alpha_4)\ZF(\alpha_1,-\alpha_2,\alpha_3,-\alpha_4)
\\+
\Co(\alpha_1,-\alpha_2,-\alpha_3,\alpha_4)\ZF(\alpha_1,-\alpha_2,-\alpha_3,\alpha_4)+
\Co(-\alpha_1,\alpha_2,\alpha_3,-\alpha_4)\ZF(-\alpha_1,\alpha_2,\alpha_3,-\alpha_4)
\\+
\Co(-\alpha_1,\alpha_2,-\alpha_3,\alpha_4)\ZF(-\alpha_1,\alpha_2,-\alpha_3,\alpha_4).
\end{multline}

Similarly, summing \eqref{S12} and \eqref{S21}, and applying \cite[4.21.16]{HMF}, we prove that
\begin{multline}\label{S12+S21}
\Sok_{1,2}(\AlphaVec)+\Sok_{2,1}(\AlphaVec)=
\Co(-\alpha_1,-\alpha_2,-\alpha_3,-\alpha_4)\ZF(-\alpha_1,-\alpha_2,-\alpha_3,-\alpha_4)\\
+\Co(\alpha_1,\alpha_2,-\alpha_3,-\alpha_4)\ZF(\alpha_1,\alpha_2,-\alpha_3,-\alpha_4).
\end{multline}

Now let us sum  \eqref{UW11 final} and  \eqref{U3 final} with the help of \cite[4.21.16]{HMF}. This gives
\begin{multline}\label{UW11+U3}
\mainU_{1,1}(\AlphaVec)+\mainU_3(\AlphaVec)=
\Co(\alpha_1,\alpha_2,\alpha_3,-\alpha_4)\ZF(\alpha_1,\alpha_2,\alpha_3,-\alpha_4)
\\+
\Co(\alpha_1,\alpha_2,-\alpha_3,\alpha_4)\ZF(\alpha_1,\alpha_2,-\alpha_3,\alpha_4)+
\Co(-\alpha_1,-\alpha_2,\alpha_3,-\alpha_4)\ZF(-\alpha_1,-\alpha_2,\alpha_3,-\alpha_4)
\\+
\Co(-\alpha_1,-\alpha_2,-\alpha_3,\alpha_4)\ZF(-\alpha_1,-\alpha_2,-\alpha_3,\alpha_4).
\end{multline}
Summing \eqref{S11} and \eqref{S3}, and applying \cite[4.21.16]{HMF}, we conclude that
\begin{multline}\label{S11+S3}
\Sok_{11}(\AlphaVec)+\Sok_{3}(\AlphaVec)=
\Co(-\alpha_1,\alpha_2,-\alpha_3,-\alpha_4)\ZF(-\alpha_1,\alpha_2,-\alpha_3,-\alpha_4)\\+
\Co(\alpha_1,-\alpha_2,-\alpha_3,-\alpha_4)\ZF(\alpha_1,-\alpha_2,-\alpha_3,-\alpha_4).
\end{multline}

Finally, it follows from \eqref{MT(a1,a2,a3,a4) conj}, \eqref{MT2momto4mom}, \eqref{UW12+UW21}, \eqref{S12+S21}, \eqref{UW11+U3} and \eqref{S11+S3} that
\begin{multline}\label{sum of all MT}
\MT_{2,4}(\AlphaVec)+\mainU_{1,2}(\AlphaVec)+\mainU_{2,1}(\AlphaVec)+\Sok_{1,2}(\AlphaVec)+\Sok_{2,1}(\AlphaVec)+
\mainU_{1,1}(\AlphaVec)+\mainU_3(\AlphaVec)+\Sok_{1,1}(\AlphaVec)+\Sok_{3}(\AlphaVec)\\=\MT_4(\alpha_1,\alpha_2,\alpha_3,\alpha_4).
\end{multline}


\section{Appendix. Kuznetsov's proof of Theorem \ref{thm:4thmomkuznetsov}}
Let us take $\alpha_1=\alpha_2=\alpha_3=\alpha_4=0$ (consequently, we may delete  dependence on $\alpha$ in all formulas) and try to estimate the right-hand side of
\eqref{4mom result}.
First, we show that the terms  $\E_h(W_{1,1})$ and $\E_h(W_{2,1})$ are negligible. This follows from the result below.

\begin{lem}\label{EhZERO-lem1}
For $m>2$ the following estimate holds:
\begin{equation}\label{hatWm11 12 ZEROest}
\widehat{W}_{1,1}(m)+\widehat{W}_{2,1}(m)\ll\frac{1}{k^2m^A},
\end{equation}
and therefore,
\begin{equation}\label{EhW11W12 est}
\term_{h}^{1}\ll\frac{1}{k^2},\quad \term_{h}^{2}\ll\frac{1}{k^2}.
\end{equation}
\end{lem}
\begin{proof}
Let us start with estimating $\widehat{W}_{2,1}(m)$. It follows from  \eqref{hatWmtoPhi}, \eqref{W21def} and  \eqref{DU def} that
\begin{multline}\label{hatW21m-eq1}
\widehat{W}_{2,1}(m)\ll
\int_{x_0}^{\infty}\frac{\DU(1/y)}{y}\Phi_k\left(\frac{y}{1+y}\right)\Phi_m\left(\frac{1}{1+y}\right)dy\\\ll
\int_{x_0/(1+x_0)}^{1}\frac{\DU(1/x-1)}{x(1-x)}\Phi_k\left(x\right)\Phi_m\left(1-x\right)dx.
\end{multline}
We split the integral into two: the first one over $x_0/(1+x_0)<x<1-k^{-\delta}$ and the second one over
$1-k^{-\delta}<x<1$ with $1<\delta<2$. The first integral is negligible by \eqref{Phik0 LGest0}. To bound the second integral we apply  \eqref{Phik0 LGest0} and \eqref{Phik0 LGest1}. This yields
\begin{equation}\label{hatW21m-eq2}
\widehat{W}_{2,1}(m)\ll\frac{1}{(km)^{A}}+\frac{1}{m^Ak^{\delta(m-1)}}.
\end{equation}
The case of $\widehat{W}_{1,1}(m)$ is slightly more complicated.
It follows from  \eqref{hatWmtoPhi}, \eqref{W11def} and  \eqref{DU def} that
\begin{multline}\label{hatW11m-eq1}
\widehat{W}_{1,1}(m)\ll
\int_{x_0}^{\infty}\frac{\DU(1/y)}{1+y}\phi_k\left(\frac{1}{1+y}\right)\Phi_m\left(\frac{1}{1+y}\right)dy
\int_{0}^{(1+x_0)^{-1}}\DU\left(\frac{x}{1-x}\right)\phi_k\left(x\right)\Phi_m\left(x\right)\frac{dx}{x}.
\end{multline}
We smoothly split the integral into two: the first one over $0<x<k^{-\delta}$ and the second over
$k^{-\delta}<x<1/(1+x_0)$ with some $1<\delta<2$. Using \eqref{Phik0 LGest0} and Lemma \ref{lem:phikest}, we estimate the first one trivially as
\begin{equation}\label{hatW11m-eq2}
\int_{0}^{1}\DU\left(\frac{x}{k^{-\delta}}\right)\phi_k\left(x\right)\Phi_m\left(x\right)\frac{dx}{x}\ll\frac{1}{m^Ak^{\delta(m-1)}}.
\end{equation}
Now let us estimate the integral over $k^{-\delta}<x<1/(1+x_0)$. Define
\begin{equation}\label{Yy def}
\y_k(x)=\sqrt{x(1-x)}\phi_k(x),\quad
\Y_m(x)=\sqrt{1-x}\Phi_m(x).
\end{equation}
These functions satisfy the  differential equations (see \cite[Corollary 5.6, Lemma 5.15]{BFJEMS}):
\begin{equation}\label{Ykdifeq}
\y_k''(x)+\left(\frac{1}{4x^2(1-x)^2}+\frac{k(k-1)}{x(1-x)} \right)\y_k(x)=0,
\end{equation}
\begin{equation}\label{ymdifeq}
\Y_m''(x)-\left(\frac{(m-1/2)^2}{x^2(1-x)}-\frac{1}{4x^2(1-x)^2}+\frac{1}{4x(1-x)} \right)\Y_m(x)=0.
\end{equation}
Using \eqref{Ykdifeq} and integrating by parts twice, one has
\begin{multline}\label{hatW11m-eq3}
\int_{0}^{(1+x_0)^{-1}}\frac{\partial^2}{\partial x^2}\left(
\DU\left(\frac{k^{-\delta}}{x}\right)
\DU\left(\frac{x}{1-x}\right)\frac{\Y_m\left(x\right)x^{1/2}(1-x)}
{\left(\frac{1}{4}+k(k-1)x(1-x) \right)}
\right)\y_k(x)dx.
\end{multline}
After lengthy but straightforward computations, it turns out that the largest contribution comes from the second derivative of $\Y_m(x)$. Applying
\eqref{ymdifeq}, followed by  \eqref{Phik0 LGest0} and Lemma \ref{lem:phikest}, one has
\begin{multline}\label{hatW11m-eq4}
\int_{0}^{(1+x_0)^{-1}}
\DU\left(\frac{k^{-\delta}}{x}\right)
\DU\left(\frac{x}{1-x}\right)\frac{\Y_m(x)\y_k(x)x^{1/2}(1-x)}
{\left(\frac{1}{4}+k(k-1)x(1-x) \right)}\frac{(m-1/2)^2}{x^2(1-x)}dx\\\ll
\frac{m^2}{k^2}\int_{k^{-\delta}}^{(1+x_0)^{-1}}\frac{\phi_k(x)\Phi_m(x)}{x^2}dx\ll\frac{1}{k^2m^A}.
\end{multline}
Using \eqref{hatW11m-eq2} and \eqref{hatW11m-eq4}, we show that
\begin{equation}\label{hatW11m-eq5}
\widehat{W}_{1,1}(m)\ll\frac{1}{k^2m^{A}}+\frac{1}{m^Ak^{\delta(m-1)}}.
\end{equation}
Now \eqref{hatWm11 12 ZEROest} follows from \eqref{hatW21m-eq2} and \eqref{hatW11m-eq5}.
\end{proof}
The next step is to prove that the terms $\term_{d}^{2}$ and $\term_{c}^{2}$ are also small.
\begin{lem}\label{Ed21ZERO-lem1}
One has
\begin{equation}\label{hatW210 est1}
\widehat{W}_{2,1}^{0}(ir)\ll\frac{1}{k^{1-\epsilon}r^{5/2}},
\end{equation}
and therefore,
\begin{equation}\label{EdEC21ZEROest}
\term_{d}^{2}+\term_{c}^{2}\ll\frac{1}{k^{1-\epsilon}}.
\end{equation}
\end{lem}
\begin{proof}
Using \eqref{hatW0}, \eqref{W21def} and making the change of variable $\frac{y}{1+y}=x$, we obtain
\begin{multline}\label{hatW210 eq1}
\widehat{W}_{2,1}^{0}(ir)=
\int_{0}^{1}
\DU\left(\frac{1-x}{x}\right)\Phi_k\left(x\right)
\left(\frac{x}{1-x}\right)^{-1/2-ir}
\HyGI\left(\frac{1}{2}+ir,\frac{1}{2}+ir,1+2ir;\frac{x-1}{x}\right)\\\times\frac{dx}{x(1-x)}.
\end{multline}
Now we can replace $\HyGI\left(\cdot\right)$ in \eqref{hatW210 eq1} by its asymptotic expansion
given by \eqref{W02F1 asympt}. It is enough to investigate only the contribution of the main term:
\begin{equation}\label{W02F1 asymptMT}
\HyGI\left(\frac{1}{2}+ir,\frac{1}{2}+ir,1+2ir;\frac{-(1-x)}{x}\right)\sim\frac{(1-x)^{1/4}}{\sqrt{|r|+1}}
\left(\frac{x}{1-x}\right)^{\frac{1}{4}+ir}\left(\frac{1+\sqrt{x}}{\sqrt{1-x}}\right)^{-2ir},
\end{equation}
since all other terms are of the same shape and are smaller in absolute value. To estimate the contribution of the error term in \eqref{W02F1 asympt} we also apply
\eqref{Phik0 LGest0}, \eqref{Phik0 LGest1}, thus obtaining
\begin{equation}\label{hatW210 eq2}
\frac{1}{(k|r|)^A}+\frac{1}{|r|^A}
\int_{1-k^{-2+\epsilon}}^{1}\frac{dx}{(1-x)^{1/2-\epsilon}}\ll\frac{1}{k^{1-\epsilon}|r|^A},
\end{equation}
which is better than \eqref{hatW210 est1}. Using the notation \eqref{Yy def} one has
\begin{equation}\label{hatW210 eq3}
\widehat{W}_{2,1}^{0}(ir)\ll\frac{1}{k^{1-\epsilon}|r|^A}+\frac{1}{\sqrt{|r|}}
\int_{0}^{1}
\DU\left(\frac{1-x}{x}\right)\frac{\Y_k\left(x\right)}{x^{5/4}(1-x)}
\left(\frac{1+\sqrt{x}}{\sqrt{1-x}}\right)^{-2ir}dx.
\end{equation}
Integrating by parts twice, we infer
\begin{multline}\label{hatW210 eq4}
\widehat{W}_{2,1}^{0}(ir)\ll\frac{1}{k^{1-\epsilon}|r|^A}+\frac{1}{|r|^{5/2}}
\int_{0}^{1}\left(\frac{1+\sqrt{x}}{\sqrt{1-x}}\right)^{-2ir}
\frac{\partial}{\partial x}\Bigl(\frac{\Y_k(x)(1-x)}{x^{1/4}}\frac{\partial}{\partial x}\DU\left(\frac{1-x}{x}\right)\\+
\DU\left(\frac{1-x}{x}\right)\frac{\Y'_k(x)(1-x)}{x^{1/4}}+
\DU\left(\frac{1-x}{x}\right)\frac{\Y_k(x)(1-x)}{x^{5/4}}
\Bigr)dx.
\end{multline}
The derivative of $\DU\left(\frac{1-x}{x}\right)$ is nonzero only for $x_0/(1+x_0)<x<1/(1+x_0)$,
in which case $\Y_k(x)$ and its derivative are \eqref{Phik0 LGest0} negligible. Furthermore, all powers of $x$ can be estimated by constants since we integrate over
$x_0/(1+x_0)<x<1$. Hence
\begin{multline}\label{hatW210 eq5}
\widehat{W}_{2,1}^{0}(ir)\ll\frac{1}{k^{1-\epsilon}|r|^A}+\frac{1}{|r|^{5/2}}
\int_{0}^{1}\DU\left(\frac{1-x}{x}\right)\left(\frac{1+\sqrt{x}}{\sqrt{1-x}}\right)^{-2ir}\\\times
\Bigl(\frac{\Y''_k(x)(1-x)}{x^{1/4}}+\frac{\Y'_k(x)}{x^{1/4}}+\frac{\Y_k(x)}{x^{5/4}}
\Bigr)dx.
\end{multline}
Consider the summand with  $\Y'_k(x)$. Integrating by parts, we obtain  that it is bounded by
\begin{equation}\label{hatW210 eq6}
\frac{1}{|r|^{7/2}}
\int_{0}^{1}\DU\left(\frac{1-x}{x}\right)\left(\frac{1+\sqrt{x}}{\sqrt{1-x}}\right)^{-2ir}
\Bigl(\Y''_k(x)(1-x)x^{1/4}+\Y'_k(x)x^{1/4}+\frac{\Y'_k(x)(1-x)}{x^{3/4}}
\Bigr)dx.
\end{equation}
Therefore,
\begin{multline}\label{hatW210 eq7}
\widehat{W}_{2,1}^{0}(ir)\ll\frac{1}{k^{1-\epsilon}|r|^A}+\frac{1}{|r|^{5/2}}
\int_{0}^{1}\DU\left(\frac{1-x}{x}\right)\left(\frac{1+\sqrt{x}}{\sqrt{1-x}}\right)^{-2ir}\\\times
\Bigl(\frac{\Y''_k(x)(1-x)}{x^{1/4}}+\frac{\Y_k(x)}{x^{5/4}}
\Bigr)dx+\frac{1}{|r|^{7/2}}
\int_{0}^{1}\DU\left(\frac{1-x}{x}\right)\left(\frac{1+\sqrt{x}}{\sqrt{1-x}}\right)^{-2ir}
\Y'_k(x)x^{1/4}dx.
\end{multline}
In order to get rid of $\Y'_k(x)$, we integrate by parts once again
\begin{multline}\label{hatW210 eq8}
\frac{1}{|r|^{7/2}}
\int_{0}^{1}\DU\left(\frac{1-x}{x}\right)\left(\frac{1+\sqrt{x}}{\sqrt{1-x}}\right)^{-2ir}\Y'_k(x)x^{1/4}dx\\\ll
\frac{1}{|r|^{5/2}}
\int_{0}^{1}\DU\left(\frac{1-x}{x}\right)\left(\frac{1+\sqrt{x}}{\sqrt{1-x}}\right)^{-2ir}\frac{\Y_k(x)}{x^{1/4}(1-x)}dx.
\end{multline}
Substituting \eqref{hatW210 eq8}  into \eqref{hatW210 eq7} one has
\begin{multline}\label{hatW210 eq9}
\widehat{W}_{2,1}^{0}(ir)\ll\frac{1}{k^{1-\epsilon}|r|^A}+\frac{1}{|r|^{5/2}}
\int_{0}^{1}\DU\left(\frac{1-x}{x}\right)\left(\frac{1+\sqrt{x}}{\sqrt{1-x}}\right)^{-2ir}\\\times
\Bigl(\frac{\Y''_k(x)(1-x)}{x^{1/4}}+\frac{\Y_k(x)}{x^{1/4}(1-x)}
\Bigr)dx.
\end{multline}
Applying \eqref{ymdifeq} to replace $\Y''_k(x)$, we show that
\begin{equation}\label{hatW210 eq10}
\widehat{W}_{2,1}^{0}(ir)\ll\frac{1}{k^{1-\epsilon}|r|^A}+\frac{1}{|r|^{5/2}}
\int_{0}^{1}\DU\left(\frac{1-x}{x}\right)\left(\frac{1+\sqrt{x}}{\sqrt{1-x}}\right)^{-2ir}
\Bigl(\frac{k^2\Y_k(x)}{x^{9/4}}+\frac{\Y_k(x)}{x^{9/4}(1-x)}
\Bigr)dx.
\end{equation}
Finally, using \eqref{Phik0 LGest0} and \eqref{Phik0 LGest1}, we prove \eqref{hatW210 est1}.

In order to prove \eqref{EdEC21ZEROest}, we apply \eqref{hatW210 est1} together with the mean Lindel\"{o}f estimate (see \cite{Ivic}):
\begin{equation}\label{Lj4 meanLindelef}
\sum_{T<t_j<2T}\omega_jL_j(1/2)^4\ll T^{2+\epsilon},
\end{equation}
and  the upper bound on the eighth moment of the Riemann zeta function (see \cite[(4.3)]{Ivic}):
\begin{equation}\label{zeta8 est}
\int_T^{2T}|\zeta(1/2+it)|^8dt\ll T^{3/2+\epsilon}.
\end{equation}
The last estimate follows from the upper bound on the twelfth moment due to Heath-Brown \cite{HeathBr}, the mean Lindel\"{o}f estimate on the fourth moment and the Cauchy-Schwartz inequality:
\begin{equation}
\int_T^{2T}|\zeta(1/2+it)|^8dt\ll \left(\int_T^{2T}|\zeta(1/2+it)|^{12}dt\right)^{1/2}
\left(\int_T^{2T}|\zeta(1/2+it)|^4dt\right)^{1/2}\ll T^{3/2+\epsilon}.
\end{equation}

\end{proof}
We are left to investigate the behaviour of $\widehat{W}_{1,1}(t)$. This is the most important ingredient of the proof of Theorem \ref{thm:4thmomkuznetsov}.
Let $u=k-1/2$ and  for $r\le u$ let $\xi_0$ be such that
\begin{equation}\label{xi0 def}
\sin\xi_0=\frac{r}{u}.
\end{equation}
\begin{lem}\label{hatW011 ZERO-lem1}
For $r\gg k^{1+\epsilon}$ one has
\begin{equation}\label{hatW011 est1}
\widehat{W}_{1,1}^{0}(ir)\ll\frac{1}{k(r/k)^A}.
\end{equation}
For $\frac{u(1+\delta)}{\sqrt{1+x_0}}<r\ll k^{1+\epsilon}$ one has
\begin{equation}\label{hatW011 est2}
\widehat{W}_{1,1}^{0}(ir)\ll\frac{1}{k^A}.
\end{equation}
For $u^{\epsilon}\ll r<\frac{u(1+\delta)}{\sqrt{1+x_0}}$ one has
\begin{equation}\label{hatW011 asympt}
\widehat{W}_{1,1}^{0}(ir)=
\frac{c_0(r)e^{iw(r,u)}}{\sqrt{ur}(u^2-r^2)^{1/4}}\DU\left(\frac{r^2}{u^2-r^2}\right)+O\left(\frac{1}{ur^{3/2}}\right),
\end{equation}
where $|c_0(r)|\ll1$ and
\begin{equation}\label{w(r,u)def}
w(r,u)=-2u\arcsin\frac{r}{u}+2r\log\frac{r}{u+\sqrt{u^2-r^2}}.
\end{equation}
Furthermore,
\begin{equation}\label{hatW011 est0}
\widehat{W}_{1,1}^{0}(ir)\ll\frac{1}{k^{1/2}}.
\end{equation}
\end{lem}
\begin{proof}
Using \eqref{hatW0}, \eqref{W21def} and making the change of variable $\frac{1}{1+y}=x$, we obtain
\begin{multline}\label{hatW110 eq1}
\widehat{W}_{1,1}^{0}(ir)=
\int_{0}^{1}
\DU\left(\frac{x}{1-x}\right)\phi_k\left(x\right)
\left(\frac{x}{1-x}\right)^{ir}\\\times
\HyGI\left(\frac{1}{2}+ir,\frac{1}{2}+ir,1+2ir;\frac{-x}{1-x}\right)\frac{dx}{\sqrt{x(1-x)}}.
\end{multline}
Estimating the integral trivially by absolute value with the use of Lemma \ref{lem:phikest0} and Lemma \ref{lem:W0 2F1}, we show that $\widehat{W}_{1,1}^{0}(ir)\ll k^{-1/2}$, thus proving \eqref{hatW011 est0}.

Next, we replace $\HyGI\left(\cdot\right)$ in \eqref{hatW110 eq1} by its asymptotic expansion
\eqref{W02F1 asympt}. Again it is enough to investigate only the contribution of the main term
\begin{multline}\label{W02F1 asymptMTv2}
\HyGI\left(\frac{1}{2}+ir,\frac{1}{2}+ir,1+2ir;\frac{-x}{1-x}\right)\sim
\frac{(1-x)^{1/4}}{\sqrt{r}}
\left(\frac{x}{1-x}\right)^{-ir}\left(\frac{1+\sqrt{1-x}}{\sqrt{x}}\right)^{-2ir},
\end{multline}
since all other terms are of the same shape and are smaller in absolute value. To estimate the contribution  of the error term in \eqref{W02F1 asympt} to \eqref{hatW110 eq1} we apply Lemma \ref{lem:phikest0} , getting $O(k^{-1/2}r^{-A}).$ Substituting \eqref{W02F1 asymptMTv2} into  \eqref{hatW110 eq1} and using the definition \eqref{Yy def}, we infer
\begin{equation}\label{hatW110 eq2}
\widehat{W}_{1,1}^{0}(ir)\ll k^{-1/2}r^{-A}+
\frac{1}{\sqrt{r}}\int_{0}^{1}
\DU\left(\frac{x}{1-x}\right)\left(\frac{1+\sqrt{1-x}}{\sqrt{x}}\right)^{-2ir}\frac{\y_k(x)dx}{(1-x)^{3/4}x}.
\end{equation}
Integration by parts repeated $n$ times yields the upper bound:
\begin{multline}
\widehat{W}_{1,1}^{0}(ir)\ll k^{-1/2}r^{-A}+
\frac{1}{r^{1/2+2n}}\int_{0}^{1}
D^n\left(\DU\left(\frac{x}{1-x}\right)\frac{\y_k(x)}{(1-x)^{3/4}x}\right)
\left(\frac{1+\sqrt{1-x}}{\sqrt{x}}\right)^{-2ir}dx,
\end{multline}
where $D(f(x))=\frac{\partial}{\partial x}(xf(x)\sqrt{1-x})$. In order to study the integral above, we apply \eqref{Ykdifeq} in a way that avoids high derivatives of $\y_k(x)$. More precisely, every time we obtain $\y_k''(x)$, we replace it using \eqref{Ykdifeq} so that  we finally get a combination of $\y_k(x)$ and $\y'_k(x)$. In order not to deal with $\y'_k(x)$ we integrate by parts "back", obtaining only $\y_k(x)$.  It turns out that the worst case scenario is when on each step of integration by parts we take the derivative of $\y_k(x)$ and later replace $\y_k''(x)$ by $k^2\y_k(x)/(x(1-x))$. In this case, we  show that
\begin{multline}\label{hatW110 eq4}
\widehat{W}_{1,1}^{0}(ir)\ll k^{-1/2}r^{-A}+
\frac{k^{2n}}{r^{1/2+2n}}\int_{0}^{1}
\DU\left(\frac{x}{1-x}\right)|\y_k(x)|x^{n-1}dx\ll
k^{-1/2}r^{-A}+\frac{k^{2n-1/2}}{r^{1/2+2n}}\\ \ll\frac{1}{k(r/k)^{1/2+2n}},
\end{multline}
thus proving \eqref{hatW011 est1}.

Now let us consider the case $k^{\epsilon}\ll r\ll k^{1+\epsilon}$. Making in \eqref{hatW110 eq1} the change of variable $x=\sin^2\frac{\sqrt{\xi}}{2}$, we obtain
\begin{multline}\label{hatW110 eq5}
\widehat{W}_{1,1}^{0}(ir)=
\int_{0}^{\pi^2}
\DU\left(\tan^2\frac{\sqrt{\xi}}{2}\right)\phi_k\left(\sin^2\frac{\sqrt{\xi}}{2}\right)
\left(\tan\frac{\sqrt{\xi}}{2}\right)^{2ir}\\\times
\HyGI\left(\frac{1}{2}+ir,\frac{1}{2}+ir,1+2ir;-\tan^2\frac{\sqrt{\xi}}{2}\right)\frac{d\xi}{2\sqrt{\xi}}.
\end{multline}
Next, we write asymptotic expansion for the functions $\HyGI\left(\cdot\right)$  (see  \eqref{W02F1 asympt}) and  $\phi_k\left(\sin^2\frac{\sqrt{\xi}}{2}\right)$
(see \cite[Theorem 5.14]{BFJEMS}). This
shows that the leading term in the asymptotic expansion of $\widehat{W}_{1,1}^{0}(ir)$ is given by
\begin{multline}\label{hatW110 eq6}
\frac{-\pi 4^{ir}\Gamma^2(1/2+ir)}{\Gamma(1+2ir)}\int_{0}^{\pi^2}
\DU\left(\tan^2\frac{\sqrt{\xi}}{2}\right)\frac{\cos^{1/2}\frac{\sqrt{\xi}}{2}}{\xi^{1/4}\sin^{1/2}\sqrt{\xi}}Y_0(u\sqrt{\xi})
\left(\frac{\sin\frac{\sqrt{\xi}}{2}}{1+\cos\frac{\sqrt{\xi}}{2}}\right)^{2ir}d\xi\\=
\frac{c_0(r)}{\sqrt{r}}\int_{0}^{\pi/2}\DU\left(\tan^2x\right)Y_0(2ux)(\tan\frac{x}{2})^{2ir}\frac{x^{1/2}dx}{\sqrt{\sin x}}.
\end{multline}
We remark that while further transforming  the integral on the right-hand side of \eqref{hatW110 eq6}, some constants will occur. For the sake of simplicity, we absorb all these constants in $c_0(r)$ without changing the notation.

Now we can split smoothly the last integral into two parts:  first one over $ux<u^{\epsilon}$ and the second  one over $ux>u^{\epsilon}$. To this end, we use the relation
\begin{equation}
\DU(u^{1-\epsilon}x)+\DU(u^{-1+\epsilon}x^{-1})=1.
\end{equation}

Consider the first integral. Note that in this case $\DU\left(\tan^2x\right)=1$.
Let $D(f(x))=\frac{\partial}{\partial x}\left(f(x)\sin x\right)$.  Integrating by parts, one has
\begin{multline}\label{hatW110 eq7}
\frac{c_0(r)}{\sqrt{r}}\int_{0}^{\pi/2}\DU(u^{1-\epsilon}x)Y_0(2ux)(\tan\frac{x}{2})^{2ir}\frac{x^{1/2}dx}{\sqrt{\sin x}}=\\=
\frac{c_0(r)}{r^{1/2+n}}\int_{0}^{\pi/2}D^n\left(\DU(u^{1-\epsilon}x)\frac{x^{1/2}Y_0(2ux)}{\sqrt{\sin x}}\right)
(\tan\frac{x}{2})^{2ir}dx\ll\frac{k^{\epsilon}}{kr^{1/2+n}}.
\end{multline}

Consider the second integral. In this  case, the argument of the Bessel function  is sufficiently large so that we can write its asymptotic expansion. Once again we can consider only the contribution of the main term given by
\begin{equation}\label{hatW110 eq8}
\frac{c_0(r)}{\sqrt{ru}}\sum_{\pm}\int_{0}^{\pi/2}\DU\left(\tan^2x\right)\DU\left(\frac{u^{\epsilon}}{ux}\right)e^{iw_{\pm}(x)}\frac{dx}{\sqrt{\sin x}},
\end{equation}
\begin{equation}\label{hatW110 wpm def}
w_{\pm}(x)=\pm2ux+2r\log\left(\tan\frac{x}{2}\right).
\end{equation}

Note that there is no saddle point in the plus case, and therefore, the derivative $w'_{+}(x)$ is sufficiently large. Integrating by parts, we show that this integral is negligible:
\begin{multline}\label{hatW110 eq9}
\frac{c_0(r)}{\sqrt{ru}}\int_{0}^{\pi/2}\DU\left(\tan^2x\right)\DU\left(\frac{u^{\epsilon}}{ux}\right)e^{iw_{+}(x)}\frac{dx}{\sqrt{\sin x}}\\\ll
\frac{1}{\sqrt{ru}}\int_{0}^{\pi/2}D^n\left(\DU\left(\frac{u^{\epsilon}}{ux}\right)\frac{\DU\left(\tan^2x\right)}{\sqrt{\sin x}}\right)dx\ll
\frac{1}{r^{1/2+n}\sqrt{u}},
\end{multline}
where $D(f(x))=\frac{\partial}{\partial x}\left(f(x)\frac{\sin x}{r+u\sin x}\right).$ 

Since $w'_{-}(x)=-2u+\frac{2r}{\sin x}$,  this function has the saddle point
$\xi_0$, which is defined as $\sin\xi_0=\frac{r}{u}.$  This point belongs to the interval of integration only if
\begin{equation}
\DU\left(\tan^2\xi_0\right)\DU\left(\frac{u^{\epsilon}}{u\xi_0}\right)=\DU\left(\frac{r^2}{u^2-r^2}\right)\neq0.
\end{equation}
Note that  $\DU\left(\frac{u^{\epsilon}}{u\xi_0}\right)=1$ because of the restriction  $r\gg k^{\epsilon}$. Therefore, if $r>\frac{u(1+\delta)}{\sqrt{1+x_0}}$ (see \eqref{DU def}), the saddle point is located outside the interval of integration, and  $$w'_{-}(x)=\frac{2r-2u\sin x}{\sin x}\gg\frac{u}{\sin x}.$$ As a result, integration by parts  yields the estimate $O(r^{-1/2}u^{-1/2-n})$ (similarly to \eqref{hatW110 eq9}). This completes the proof of \eqref{hatW011 est2}.

We are left to investigate the case $u^{\epsilon}\ll r\le\frac{u(1+\delta)}{\sqrt{1+x_0}}$. First, we introduce a smooth partition of unity in the integral \eqref{hatW110 eq8} for the purpose of localizing a saddle point. Accordingly, let $\beta(y)$ be a smooth characteristic function of some interval $(b_1,b_2)$ such that $b_1<1<b_2$ and
\begin{equation}
\beta\left(\frac{u\sin x}{r}\right)\DU\left(\tan^2x\right)\DU\left(\frac{u^{\epsilon}}{ux}\right)\neq0
\end{equation}
for $c_1<x<c_2$ with $c_2-c_1\gg r/u.$ In the remaining integral with $1-\beta\left(\frac{u\sin x}{r}\right)$ we have a good lower bound on
$w'_{-}(x)=\frac{2r-2u\sin x}{\sin x}\gg\frac{r}{\sin x}$, and thus integration by parts works and yields the estimate $O(r^{-1/2-n}u^{-1/2})$. It is left to study the integral:
\begin{equation}\label{hatW110 eq10}
I=\frac{c_0(r)}{r}\int_{0}^{\pi/2}\DU\left(\tan^2x\right)\DU\left(\frac{u^{\epsilon}}{ux}\right)\beta\left(\frac{u\sin x}{r}\right)e^{iw_{-}(x)}
\sqrt{\frac{r}{u\sin x}}dx.
\end{equation}
Since $\beta(u\sin(x)/r)$ localizes $x$ approximately near the point $r/u$, and  $r\gg k^{\epsilon}$, it is possible to replace  $\DU\left(\frac{u^{\epsilon}}{ux}\right)$ by 1.
To obtain an asymptotic  expansion of the integral \eqref{hatW110 eq10}, we apply Lemma \ref{Lemma Huxley} with
\begin{equation}\label{hatW110 eq11}
\Omega_f=\Omega_g=\kappa\asymp\frac{r}{u},\quad
\Theta_f\asymp r.
\end{equation}
Consequently, we show that
\begin{equation}\label{hatW110 eq12}
I=\frac{c_0(r)}{r\sqrt{|w''_{-}(\xi_0)|}}
\DU\left(\tan^2\xi_0\right)e^{iw_{-}(\xi_0)}+O\left(\frac{1}{ur^{3/2}}\right).
\end{equation}
Note that when $r$ is close to $\frac{u}{\sqrt{1+x_0}}$, the saddle  point $\xi_0$ is close to the point where $\DU(\tan^2x)$ becomes 0. To overcome this difficulty, one can simply use the relation $\DU(\tan^2x)=1-\DU(\cot^2x)$ and evaluate asymptotically two integrals. Combining the results, we again obtain \eqref{hatW110 eq12}.

Using
\begin{equation}\label{w2derxi0}
w_{-}(\xi_0)=-2u\arcsin\frac{r}{u}+2r\log\frac{r}{u+\sqrt{u^2-r^2}},\quad
|w''_{-}(\xi_0)|=\frac{2u}{r}\sqrt{u^2-r^2}
\end{equation}
and \eqref{hatW110 eq12}, we finally conclude the  proof of \eqref{hatW011 asympt}.
\end{proof}

\begin{lem}\label{EdEcW11 lem1}
Let $u=k-1/2$. Then the following asymptotic formulas hold:
\begin{equation}\label{EdW11 asympt}
\term_{d}^{1}=
\sum_{t_j} \omega_jL^4_j\left(\frac{1}{2}\right)\W_{1,1}(t_j,u)
+O\left(\frac{1}{u^{1/2-\epsilon}}\right),
\end{equation}
\begin{equation}\label{Ec0W11 asympt}
\term_{c}^{1}=
\frac{1}{\pi}\int_{-\infty}^{\infty}
\frac{|\zeta\left(\frac{1}{2}+it\right)|^8}{|\zeta(1+2it)|^2}\W_{1,1}(t,u)dt+O\left(\frac{1}{u^{1/2-\epsilon}}\right),
\end{equation}
where
\begin{equation}\label{W11 asympt def}
\W_{1,1}(r,u)=\frac{c_0(r)e^{iw(r,u)}+c_0(-r)e^{-iw(r,u)}}{\sqrt{ur}(u^2-r^2)^{1/4}}\DU\left(\frac{r^2}{u^2-r^2}\right),
\end{equation}
$|c_0(\pm r)|\ll1$ and $w(r,u)$ is given by \eqref{w(r,u)def}.
\end{lem}
\begin{proof}
It follows from \eqref{Ed11}, \eqref{hatW+ + hatW-} and Lemma \ref{hatW011 ZERO-lem1} that
\begin{equation}\label{EdW11 asympt1}
\term_{d}^{1}=
\frac{(-1)^k}{4}\sum_{t_j\ll k} \omega_jL^4_j\left(\frac{1}{2}\right)
\varepsilon_j\left(\widehat{W}_{1,1}^{0}(t_j)+\widehat{W}_{1,1}^{0}(-t_j)\right)+O(k^{-1/2}).
\end{equation}
Note that $L_j(1/2)=0$ if $\varepsilon_j=-1$, and therefore,  we can omit $\varepsilon_j$ in \eqref{EdW11 asympt1}. Now \eqref{EdW11 asympt} is a consequence of Lemma \ref{hatW011 ZERO-lem1} and \eqref{Lj4 meanLindelef}.

Using \eqref{Ec011def},  \eqref{ZE11def}, \eqref{Z4 def}, \eqref{hatW+ + hatW-} and Lemma \ref{hatW011 ZERO-lem1}, we obtain
\begin{equation}\label{Ec011def2}
\term_{c}^{1}=
\frac{(-1)^k}{4\pi}
\int_{-\infty}^{\infty}
\frac{|\zeta\left(\frac{1}{2}+it\right)|^8}{|\zeta(1+2it)|^2}
\left(\widehat{W}_{1,1}^{0}(it)+\widehat{W}_{1,1}^{0}(-it)\right)dt+O\left(\frac{1}{u^{1/2-\epsilon}}\right).
\end{equation}
Finally, \eqref{Ec0W11 asympt} follows from Lemma \ref{hatW011 ZERO-lem1} and \eqref{zeta8 est}.
\end{proof}


The next result is derived using \eqref{4mom result}, \eqref{EhW11W12 est}, \eqref{EdEC21ZEROest} and Lemma \ref{EdEcW11 lem1}.
Note, that $L_f(1/2)=0$ if $k$ is odd. Therefore, we can assume that $k\equiv 0\pmod{2}$.
\begin{thm}\label{thm:4thmomzero}
For $k\equiv 0\pmod{2}$ one has
\begin{equation}\label{4thmom zero}
\sum_{f\in H_{2k}} \omega_fL^4_f(1/2)=P_6(\log k)+2\term_{d}^{1}+2\term_{c}^{1}+O(k^{-1+\epsilon}),
\end{equation}
where  $\term_{d}^{1}$ and $\term_{c}^{1}$ are given by \eqref{EdW11 asympt} and  \eqref{Ec0W11 asympt}, respectively.
\end{thm}
\begin{rem}\label{rem:4thmomzero}
Note that $\W_{1,1}(r,u)=0$ if $|r|\gg k.$ Thus we are summing in \eqref{EdW11 asympt} over $|t_j|\ll k$, and the integral in \eqref{Ec0W11 asympt} is taken over $|t|\ll k$. Using  \eqref{Lj4 meanLindelef} and \eqref{zeta8 est}, we prove the following trivial estimates:
\begin{equation}\label{EdEc trivial est}
\term_{d}^{1}\ll k^{1/2+\epsilon}, \quad \term_{c}^{1}\ll k^{\epsilon}.
\end{equation}
\end{rem}
\begin{proof}[Proof of Theorem \ref{thm:4thmomkuznetsov}]
To prove \eqref{4thmom Kuznetsov} we consider the average of \eqref{4thmom zero}. In order to estimate
\begin{equation}\label{EdEc average1}
\sum_{\substack{L<k\le 2L\\k\equiv 0\pmod{2} }}\term_{d}^{1}+\term_{c}^{1},
\end{equation}
we  just need to obtain an upper bound for the average of $\W_{1,1}(r,u)$ over $k.$ Using \eqref{w(r,u)def}, we show that
\begin{equation}\label{w(r,u)deriv}
\frac{\partial}{\partial u}w(r,u)=-2\arcsin\frac{r}{u}.
\end{equation}
Therefore, $$|\frac{\partial}{\partial k}w(r,2k-1/2)|<2\pi-\delta_0.$$ Using the generalization of \cite[Lemma 4.10]{Tit} given in \cite[Lemma 1 sec.3]{Kar},
we have
\begin{equation}\label{W11 average1}
\sum_{\substack{L/2<k\le L}}\W_{1,1}(r,2k-1/2)=\int_{L/2}^{L}\W_{1,1}(r,2x-1/2)dx+O\left(\frac{L^{\epsilon}}{L\sqrt{r}}\right).
\end{equation}
To estimate the integral we apply the first derivative test \cite[Lemma 4.3]{Tit}, proving that
\begin{equation}\label{W11 average2}
\sum_{\substack{L/2<k\le L}}\W_{1,1}(r,2k-1/2)\ll\frac{1}{r^{3/2}}+\frac{L^{\epsilon}}{L\sqrt{r}}.
\end{equation}
Finally, using \eqref{W11 average2}, \eqref{Lj4 meanLindelef} and \eqref{zeta8 est}, we infer
\begin{equation}\label{EdEc average2}
\sum_{\substack{L<k\le 2L\\k\equiv 0\pmod{2} }}
\term_{d}^{1}+
\term_{c}^{1}\ll L^{1/2+\epsilon}.
\end{equation}
\end{proof}


\begin{proof}[Proof of Theorem \ref{cor:4thmom short}]
It is required to estimate the average of  $\term_{d}^{1}+\term_{c}^{1}$. To this end, it is enough to evaluate the average of $\W_{1,1}(r,u)$, namely
\begin{equation}\label{W11 smoothaverage1}
\sum_{k}\exp\left(-\left(\frac{k-K}{G}\right)^2\right)\W_{1,1}(r,2k-1/2).
\end{equation}
To do this, we argue in the same way as in \cite[Lemma 7.3]{BFJEMS}, starting with an application of the Poisson summation formula:
\begin{equation}
\sum_{k}\exp\left(-\left(\frac{k-K}{G}\right)^2\right)\W_{1,1}(r,2k-1/2)=\sum_{m \in \Z}P(m),\quad
\end{equation}
\begin{equation}\label{P(m)def}
P(m)=\int_{-\infty}^{\infty}\exp\left(-\left(\frac{y-K}{G}\right)^2\right)\W_{1,1}(r,2y-1/2)e(-my)dy.
\end{equation}
Making the change of variable $y=K+Gx$ and using \eqref{W11 asympt def}, one has
\begin{equation}\label{P(m)est1}
P(m)\ll\frac{G}{\sqrt{r}}\int_{-\infty}^{\infty}\exp\left(-x^2\right)\DU\left(\frac{r^2}{4(K+Gx)^2-r^2}\right)
\frac{e^{ig(r,K,m,x)}dx}{\sqrt{K+Gx}(4(K+Gx)^2-r^2)^{1/4}},
\end{equation}
where $g(r,K,G,m,x)=w(r,2K+2Gx-1/2)-2\pi mGx.$ It follows from \eqref{w(r,u)deriv} that for $m\neq0$
\begin{equation}\label{g geriv}
\frac{\partial}{\partial x}g(r,K,G,m,x)=-4G\arcsin\frac{r}{2(K+Gx)-1/2}-2\pi mG\gg mG.
\end{equation}
Therefore, multiple integration by parts leads to the bound $P(m)\ll(mG)^{-A}.$  In the case $m=0$,  one has
\begin{equation}\label{g geriv2}
\frac{\partial}{\partial x}g(r,K,G,0,x)=-4G\arcsin\frac{r}{2(K+Gx)-1/2}\gg \frac{Gr}{K},
\end{equation}
and multiple integration by parts yields $P(0)\ll (Gr/K)^{-A}.$  Hence for $r\gg K^{1+\epsilon}/G$ one has $P(0)\ll K^{-A}.$ In the case when
$r\ll K^{1+\epsilon}/G$, we have the trivial bound $P(0)\ll G/\sqrt{rK^2}$.  Using \eqref{EdW11 asympt}  and applying the estimates proved above, we show that
\begin{equation}\label{Ed averagesmooth1}
\sum_{k}\exp\left(-\left(\frac{k-K}{G}\right)^2\right)\term_{d}^{1}\ll
\frac{G}{K}\sum_{t_j\ll K^{1+\epsilon}/G}\frac{\omega_j}{\sqrt{t_j}}L^4_j\left(\frac{1}{2}\right)+\frac{G^{1+\epsilon}}{K^{1/2}}\ll
\frac{K^{1/2+\epsilon}}{G^{1/2}}+\frac{G^{1+\epsilon}}{K^{1/2}}.
\end{equation}
Arguing in the same way with the use of \eqref{zeta8 est}, we conclude that
\begin{equation}\label{Ed averagesmooth2}
\sum_{k}\exp\left(-\left(\frac{k-K}{G}\right)^2\right)\term_{c}^{1}\ll K^{\epsilon}+\frac{G^{1+\epsilon}}{K^{1/2}}.
\end{equation}
Finally, we remark that the estimates \eqref{Ed averagesmooth1} and \eqref{Ed averagesmooth2} are better than the averaged main term if  $G\gg K^{1/3+\epsilon}$. This completes the proof of  \eqref{4thmom short}.
\end{proof}


\end{document}